\documentclass[12pt,reqno]{amsart}
\usepackage{amsmath}
\usepackage{amssymb}
\usepackage{amstext}
\usepackage{a4wide}
\usepackage{graphicx}
\allowdisplaybreaks \numberwithin{equation}{section}
\usepackage{color}
\usepackage{cases}
\usepackage{hyperref}

\numberwithin{equation}{section}

\newtheorem{theorem}{Theorem}[section]
\newtheorem{proposition}[theorem]{Proposition}

\newtheorem{lemma}[theorem]{Lemma}

\theoremstyle{definition}

\newtheorem{innercustomthm}{Proposition}
\newenvironment{customthm}[1]
{\renewcommand\theinnercustomthm{#1}\innercustomthm}
{\endinnercustomthm}

\theoremstyle{remark}
\newtheorem{remark}[theorem]{Remark}

\begin{document}

\title[Nearly parallel helical vortices]
{Co-rotating nearly parallel helical vortices with small cross-section in 3D incompressible Euler equations}

\author{Daomin Cao, Jie Wan}
	
\address{State Key Laboratory of Mathematical Sciences, Academy of Mathematics and Systems Science, Chinese Academy of Sciences, Beijing 100190, and University of Chinese Academy of Sciences, Beijing 100049,  P.R. China}
\email{dmcao@amt.ac.cn}
\address{School of Mathematics and Statistics, Beijing Institute of Technology, Beijing 100081,  P.R. China}
\email{wanjie@bit.edu.cn}



\begin{abstract}
In this article, we consider clustered solutions to a semilinear elliptic equation in divergence form
\begin{equation*}
\begin{cases}
-\varepsilon^2\text{div}(K(x)\nabla u)=  (u-q|\ln\varepsilon|)^{p}_+,\ \ &x\in \Omega,\\
u=0,\ \ &x\in\partial \Omega
\end{cases}
\end{equation*}
for small values of $ \varepsilon $. Using Green’s function of the elliptic operator  $ -\text{div}(K(x)\nabla) $ and finite-dimensional reduction method, we prove that there exist  clustered solutions with cluster point $ 0 $ and cluster distance $ |\ln\varepsilon| ^{-\frac{1}{2}} $ whose small-structure is governed by some functional $ H_N $ determined by $ K $ and $ q $. As an application, we prove the existence of traveling-rotating helical vorticity fields to 3D incompressible Euler equations in infinite cylinders, whose support sets consist  of  helical tubes with small cross-section of radius $ \varepsilon $ and arbitrary circulation $ \kappa $ and concentrates near ``$ 2N $'' and ``$ 2N+1 $'' type of co-rotating helical solutions of nearly parallel vortex filaments model as $ \varepsilon\to0 $, which justifies the result in Klein, Majda and Damodaran [1995, JFM] and generalizes results in Guerra and Musso [arxiv: 2502.01470]. Several kinds of solutions such as ``2 asymmetric'', ``$ 2\times2 $ asymmetric'' and ``$ 2\times2+1 $ asymmetric'' type of co-rotating helical filaments are also considered.

\vspace{0.3cm}

\textbf{Keywords:} Incompressible Euler equation;  Helical symmetry; Semilinear elliptic equations; Clustered solution; Lyapunov-Schmidt reduction method.

\vspace{0.3cm}
\textbf{MSC:} Primary 76B47\,\,\,\,\,\,\,\,Secondary 76B03, 35A02, 35Q31
\end{abstract}

\maketitle

\section{Introduction and main results}
The motion of 3D incompressible ideal flow is governed by   Euler
equations
\begin{equation}\label{Euler eq}
\begin{cases}
\partial_t\mathbf{v}+(\mathbf{v}\cdot \nabla)\mathbf{v}=-\nabla P,\ \ &D\times (0,T),\\
\nabla\cdot \mathbf{v}=0,\ \ &D\times (0,T),\\
\mathbf{v}\cdot \mathbf{n}=0,\ \ &\partial D\times (0,T),
\end{cases}
\end{equation}
where $ D\subseteq \mathbb{R}^3 $, $ \mathbf{v}=(v_1,v_2,v_3) $ is the velocity field, $ P$ is the scalar pressure, $ \mathbf{n} $ is the outward unit normal to $ \partial D $. 

For velocity field $\mathbf{v}$, we define the  vorticity field $ \mathbf{w}=\nabla\times \mathbf{v} $.  Then $ \mathbf{w} $ satisfies vorticity equations  
\begin{equation}\label{Euler eq2}
\begin{cases}
\partial_t\mathbf{w}+(\mathbf{v}\cdot \nabla)\mathbf{w}=(\mathbf{w}\cdot \nabla)\mathbf{v},\\
\nabla\times \mathbf{v}=\mathbf{w}.
\end{cases}
\end{equation}

In this paper, we consider the existence of vorticity field of 3D incompressible Euler equations \eqref{Euler eq} concentrating in a small neighborhood of  nearly parallel vortex filaments. The study of the existence   of vorticity field of Euler equations concentrating near some curves  has a long history, traced back to the classical work of Helmholtz \cite{He}. 
 Consider the motion of a vortex tube with a small cross-section of radius $ \varepsilon $ and a fixed circulation $ \kappa $, uniformly distributed around one curve $ \gamma(t) $. Da Rios  and Levi-Civita  \cite{DR, LC} formally found that $ \gamma $ asymptotically obeys a law of the form 
\begin{equation}\label{1003}
\partial_t \gamma=\frac{\kappa}{4\pi}|\ln\varepsilon|(\partial_s\gamma\times\partial_{ss}\gamma)=\frac{\kappa\bar{K}}{4\pi}|\ln\varepsilon|\mathbf{b}_{\gamma(t)},
\end{equation}
where $ s $ is the parameter of arclength, $ \bar{K} $ is the local curvature and $ \mathbf{b}_{\gamma(t)}  $  is the unit binormal vector. $ \gamma $ and \eqref{1003} are  called the vortex filament and the vortex filament equation respectively.

From mathematical justification, Jerrard and Seis \cite{JS} first 
 showed that under some assumptions on   solutions $ \mathbf{w}_\varepsilon $ of \eqref{Euler eq2}, there holds in the sense of distribution 
\begin{equation}\label{1005}
\mathbf{w}_\varepsilon(\cdot,t)-\kappa \delta_{\gamma(t)}\mathbf{t}_{\gamma(t)}\rightharpoonup0,\ \ \text{as}\ \varepsilon\to0,
\end{equation}
where $ \gamma(t) $ satisfies \eqref{1003}, $ \mathbf{t}_{\gamma } $ is the tangent unit vector of $ \gamma $ and $ \delta_{\gamma } $ is the uniform Dirac measure on the curve $ \gamma $. For  any curve $ \gamma $ satisfying \eqref{1003}, whether there are vorticity fields of \eqref{Euler eq} concentrating near $ \gamma $ in sense of \eqref{1005} is   well-known as the vortex filament conjecture, which is unsolved except for   several kinds of special curves: the straight lines, the traveling circles and the traveling-rotating helices, see, e.g., \cite{CLW,CW,CW1,DDMW,DDMW2,DLM,Fr1,FB,GZ,MB} and reference therein.

As for the motion of multiple filaments, Klein, Majda and Damodaran \cite{KMD} have formally derived a simplified asymptotic law
of the evolution of several  nearly parallel vortex filaments. Consider the motion of $ N $ slender vortex tube which are all nearly parallel to the $ x_3 $-axis, whose centerlines are curves parameterized as $ \tilde{X}_{j,\varepsilon}(s,t) $ in the complex plane $ \mathbb{C} $, circulations are $ \kappa_j $ and cross-sections are of radius $ \varepsilon $ for $ j=1,\cdots,N $. The wave-length of each vortex tube is of the order $ O(1) $ and the mutual distance of any two vortex tubes is of the order $ O\left (|\ln\varepsilon|^{-\frac{1}{2}}\right ) $, i.e., $ dist(\tilde{X}_{i,\varepsilon}, \tilde{X}_{j,\varepsilon}) =O\left (|\ln\varepsilon|^{-\frac{1}{2}}\right )$ for $ i\neq j $. Denote $   \tau=|\ln\varepsilon|t $ and $ X_j(s, \tau)=\sqrt{|\ln\varepsilon|}\tilde{X}_{j,\varepsilon}(s, t) $. Then $ X_j(s, \tau) $ asymptotically obeys  (see (2.20) in \cite{KMD})
\begin{equation}\label{NP vor fi} 
\partial_\tau X_j=\frac{1}{4\pi}\left(\mathbf{i}\alpha_j\kappa_j\partial_{ss}X_j+2\mathbf{i}\sum_{k\neq j}\kappa_k\frac{X_j-X_k}{|X_j-X_k|^2} \right),
\end{equation}
where $ \mathbf{i} $ is  the imaginary unit and $ \alpha_j $ are structure constants of the vortex core, which are generally assumed to be 1.


System \eqref{NP vor fi} has been extensively studied  (see, for example, \cite{BM,BFM,CGY,KPV,KVG,LM}), particularly concerning its well-posedness and the potential for filament collisions. A natural question arises as to whether a rigorous justification can be provided to establish the model itself as a limit from  the Euler equations, thereby  giving a  mathematical explanation of the model proposed by \cite{KMD}.

Recently, some progress has been made  for this problem when  filaments $ X_j(s, \tau) $ are some co-rotating helices. In \cite{GM2}, Guerra and Musso considered two special type of  helices $ X_j(s, \tau) $ satisfying \eqref{NP vor fi}: ``$ N $  vortices of equal strength $ 8\pi $" placed at each vertex of a regular $ N $-polygon, and ``$ N + 1 $ vortices   of equal strength $ 8\pi $" $ N $ of them placed at the corners of a regular $ N $-polygon with the $ N + 1 $ vortex   located at the center of vorticity. By considering the existence of solutions to a 2D semilinear elliptic equations of divergence type
\begin{equation*} 
\begin{split}
-\text{div}\left(K_H(x)\nabla v\right)=f_\varepsilon(v) ,\ \ \ \ \text{in}\ \ \mathbb{R}^2,
\end{split}
\end{equation*}
where $ K_H $ is some positive definite matrix (defined by \eqref{coef matrix}) and the profile function $ f_\varepsilon $ is some exponential type function which is chosen properly, the authors constructed vorticity fields of \eqref{Euler eq} with the orthogonality condition whose vortex sets tend to rescaled versions of the   the above two kinds of co-rotating helices as $ \varepsilon\to0. $ Note that by the choice of $ f_\varepsilon $, the associated circulation of each component of vorticity field must be $ 8\pi $ and  support sets of vorticity field are the whole space, rather than a collection of helical vortex tubes with small cross-section. So as $ \varepsilon\to0 $, the vorticity fields tend to filaments of \eqref{NP vor fi} in distributional sense.

Our goal of this paper is to establish the existence of a family of solutions of \eqref{Euler eq}   with vorticities concentrating near co-rotating helical filaments of  \eqref{NP vor fi} with a more general setting. First, as for the two configurations considered in \cite{GM2}, we prove the existence of a family of solutions with vorticities $ \textbf{w}_\varepsilon $ concentrating near helical filaments as $ \varepsilon\to0 $, whose cross-sections are   compact sets with radius $ \varepsilon $ and circulations are arbitrary rather than $8\pi$. So the concentration is in both distributional and topological sense, which coincides with the model deduced in \cite{KMD}. Second, we show the existence of a family of solutions with vorticities $ \textbf{w}_\varepsilon $ concentrating near 2 asymmetric co-rotating helical filaments of \eqref{NP vor fi} with small cross-section as $ \varepsilon\to0 $. Both circulations and  distances between helical tubes and $ x_3 $-axis are different. Some other asymmetric co-rotating helical filaments of \eqref{NP vor fi} are also considered. Following the ansatz in \cite{CW,CW1}, we assume that the fluid domain is an infinitely long cylinder, that is, $ D=B_{R^*}(0)\times\mathbb{R} $ for some $ R^*>0. $


We define co-rotating nearly parallel helical filaments as follows. Let $ N\geq 2 $ be an integer. We say $ X_j(s, \tau) $, for $ j=1,\cdots,N $,  co-rotating helical filaments of model \eqref{NP vor fi} with pitch $ h>0 $ and angular velocity $ \alpha\in\mathbb{R} $, if $ X_j(s, \tau)=Z_j(\tau)e^{\mathbf{i}\frac{s}{h}} $ for some $ Z_j(\tau) $ and $   Z_j(\tau)=  Z_j(0)e^{-\mathbf{i}\alpha\tau} $ for $ j=1,\cdots,N $. 
Note that from \eqref{NP vor fi}, $ Z_j $ satisfies for $ j=1,\cdots,N $
\begin{equation}\label{NP heli vor}
\partial_\tau Z_j=\frac{1}{4\pi}\left(-\frac{\mathbf{i}\kappa_j}{h^2}Z_j+2\mathbf{i}\sum_{k\neq j}\kappa_k\frac{Z_j-Z_k}{|Z_j-Z_k|^2} \right).
\end{equation} 

Now we give some exact co-rotating helical solutions of \eqref{NP vor fi} and state our main results.

\textbf{Case 1:} ``$ N $  vortices of equal strength $ \kappa $", which was considered in \cite{GM2}. Let $ N\geq 2 $, $ \kappa_j=\kappa $ for   $ \kappa>0 $. Then 
\begin{equation*}
X_{j}(s,\tau)=Z_j(\tau)e^{\mathbf{i}\frac{s}{h}}=r_*e^{-\mathbf{i}\alpha\tau}e^{\mathbf{i}\frac{s}{h}}e^{\mathbf{i}\frac{2\pi (j-1)}{N}},\ \ \ \ \ \ \ \ \  j=1,\cdots,N, 
\end{equation*}
satisfies \eqref{NP vor fi}, where $ \alpha=\frac{\kappa}{4\pi}\left( \frac{1}{h^2}-\frac{N-1}{r_*^2}\right)  $, see also    Appendix \ref{lemA1}. This corresponds to   nearly parallel vortex filaments being of the form
\begin{equation}\label{fila 1}
\tilde{X}_{j,\varepsilon}(s, t)=\frac{r_*}{\sqrt{|\ln\varepsilon|}}e^{-\mathbf{i}\alpha|\ln\varepsilon|t}e^{\mathbf{i}\frac{s}{h}}e^{\mathbf{i}\frac{2\pi (j-1)}{N}},\ \ \ \ j=1,\cdots,N.
\end{equation}
We have 
\begin{theorem}\label{thm01}
	Let $ R^*>0 $, $ h>0 $ and $ N\geq 2 $ be an integer. For $ r_*\in (0,R^*) , \kappa>0 $, let $ \tilde{X}_{j,\varepsilon}$   $(j=1,\cdots,N)$ be $ N $ filaments satisfying \eqref{fila 1}. Then  there exists $ \varepsilon_0>0 $, such that for every $ \varepsilon\in(0,\varepsilon_0] $, \eqref{Euler eq2} has a  family of solutions with vorticities  $ \mathbf{w}_\varepsilon\in C^1(B_{R^*}(0)\times \mathbb{R}\times\mathbb{R}^+) $ satisfying in distributional sense
	\begin{equation*} 
	\mathbf{w}_\varepsilon(x,t)- \kappa\sum_{j=1}^N\delta_{\tilde{X}_{j,\varepsilon}}\textbf{t}_{\tilde{X}_{j,\varepsilon}}\rightharpoonup0\ \ \ \ \text{as}\ \varepsilon\to0.
	\end{equation*}
	Moreover, if we set $ A_{\varepsilon,i}=supp\left (\mathbf{w}_\varepsilon(x_1,x_2,0,0)\right ) \cap B_{\rho_0|\ln\varepsilon|^{-\frac{1}{2}}}\left ( \tilde{X}_{i,\varepsilon}(0,0)\right )  $, where $ \rho_0>0 $ is a small constant, then there exist constants $ R_1,R_2>0 $ independent of $ \varepsilon $  such that
	\begin{equation*}
	R_1\varepsilon \leq diam(A_{\varepsilon,i})\leq R_2\varepsilon .
	\end{equation*}
\end{theorem}

\begin{remark}
In \cite{GM2}, the authors proved the existence of $ \textbf{w}_\varepsilon $ of \eqref{Euler eq} concentrating near $ \eqref{fila 1} $ in sense of distribution, when $ \kappa=8\pi $. In contrast to \cite{GM2}, the circulation of vorticity field in our construction is allowed to be any positive constant, and the support set of each component of $ \textbf{w}_\varepsilon $ is a slender helical tube with cross-section of radius $ \varepsilon $. Our result coincides with formal computational results in \cite{KMD}.
\end{remark}

\textbf{Case 2:} ``$ N+1 $ vortices", $ N $ of them having identical circulation and placed  at each vertex of a regular $ N $-polygon   and the $ (N +1) $th vortex of arbitrary strength
at the center of vorticity. Let  $ N\geq 2 $, $ \kappa_0=\mu $, $ \kappa_j=\kappa $ for $ j=1,\cdots,N $ for   $ \mu, \kappa>0 $. Then 
\begin{equation*} 
X_0(s,\tau)=0,\ \ \ \ X_{j}(s,\tau)=r_*e^{-\mathbf{i}\alpha\tau}e^{\mathbf{i}\frac{s}{h}}e^{\mathbf{i}\frac{2\pi (j-1)}{N}},\ \ \ \ \ \ \  j=1,\cdots,N,
\end{equation*}
satisfies \eqref{NP vor fi}, where $ \alpha=\frac{\kappa}{4\pi}\left( \frac{1}{h^2}-\frac{N-1}{r_*^2}\right)-\frac{\mu}{2\pi r_*^2}  $. This corresponds to  
\begin{equation}\label{fila 2}
\tilde{X}_{0,\varepsilon}(s,t)=0,\ \ \  \tilde{X}_{j,\varepsilon}(s,t)=\frac{r_*}{\sqrt{|\ln\varepsilon|}}e^{-\mathbf{i}\alpha|\ln\varepsilon|t}e^{\mathbf{i}\frac{s}{h}}e^{\mathbf{i}\frac{2\pi (j-1)}{N}},
 \ \ \ \ j=1,\cdots,N.
\end{equation}
We have 
\begin{theorem}\label{thm02}
	Let $ R^*>0 $, $ h>0 $ and $ N\geq 2 $ be an integer. For $ r_*\in (0,R^*) , \kappa>0, \mu>0 $, let $ \tilde{X}_{j,\varepsilon} $ $(j=0,1,\cdots,N)$ be $ N+1 $ filaments satisfying \eqref{fila 2}. Then  there exists $ \varepsilon_0>0 $, such that for every $ \varepsilon\in(0,\varepsilon_0] $, \eqref{Euler eq2} has a family of solutions with vorticities  $ \mathbf{w}_\varepsilon\in C^1(B_{R^*}(0)\times \mathbb{R}\times\mathbb{R}^+) $ satisfying in distributional sense
	\begin{equation*} 
	\mathbf{w}_\varepsilon(x,t)-\mu\delta_{\tilde{X}_{0,\varepsilon}}\textbf{t}_{\tilde{X}_{0,\varepsilon}}- \kappa\sum_{j=1}^N\delta_{\tilde{X}_{j,\varepsilon}}\textbf{t}_{\tilde{X}_{j,\varepsilon}}\rightharpoonup0\ \ \ \ \text{as}\ \varepsilon\to0.
	\end{equation*}
	Moreover, if we set $ A_{\varepsilon,i}=supp\left (\mathbf{w}_\varepsilon(x_1,x_2,0,0)\right ) \cap B_{\rho_0|\ln\varepsilon|^{-\frac{1}{2}}}\left ( \tilde{X}_{i,\varepsilon}(0,0)\right )  $, then there exist constants $ R_1,R_2>0 $ independent of $ \varepsilon $  such that
	\begin{equation*}
	R_1\varepsilon \leq diam(A_{\varepsilon,i})\leq R_2\varepsilon .
	\end{equation*}
\end{theorem}

\begin{remark}
\cite{GM2} considered the above case with the circulation $ \kappa=\mu=8\pi $. In contrast to \cite{GM2}, $ \kappa,\mu $ in our construction can be arbitrary positive constants. 
\end{remark}

We also prove the existence  of a family of solutions of  \eqref{Euler eq} with vorticity fields  concentrating near co-rotating helical solutions of   model \eqref{NP vor fi} when filaments are not symmetric about $ x_3 $-axis. 

\textbf{Case 3:}   ``Co-rotating asymmetric $2 $ vortices" with different circulations and rotation radii. Let $ N= 2 $, $ \kappa_1,\kappa_2>0 $ and $ \lambda_{1,*},\lambda_{2,*}>0 $ be constants satisfying the compatibility condition
\begin{equation}\label{compati cond1} 
\frac{\kappa_2}{\kappa_1}=\frac{2\lambda_{1,*}h^2+\lambda_{1,*}\lambda_{2,*}(\lambda_{1,*}+\lambda_{2,*})}{2\lambda_{2,*}h^2+\lambda_{1,*}\lambda_{2,*}(\lambda_{1,*}+\lambda_{2,*})}.
\end{equation}
Then 
\begin{equation*} 
X_{1}(s,\tau)=\lambda_{1,*}e^{-\mathbf{i}\alpha\tau}e^{\mathbf{i}\frac{s}{h}},\ \ \ \ 	X_{2}(s,\tau)=-\lambda_{2,*}e^{-\mathbf{i}\alpha\tau}e^{\mathbf{i}\frac{s}{h}},
\end{equation*}
satisfies \eqref{NP vor fi}, where $ \alpha=\frac{\kappa_1}{4\pi h^2}  -\frac{\kappa_2}{2\pi \lambda_{1,*}(\lambda_{1,*}+\lambda_{2,*})}=\frac{\kappa_2}{4\pi h^2}  -\frac{\kappa_1}{2\pi \lambda_{2,*}(\lambda_{1,*}+\lambda_{2,*})}  $. This corresponds to nearly parallel vortex filaments  
\begin{equation}\label{fila 3}
\tilde{X}_{j,\varepsilon}(s,t)=\frac{(-1)^{j-1}\lambda_{j,*}}{\sqrt{|\ln\varepsilon|}}e^{-\mathbf{i}\alpha|\ln\varepsilon|t}e^{\mathbf{i}\frac{s}{h}},\ \ \ \ j=1,2. 
\end{equation}
We note that \eqref{compati cond1} is necessary since \eqref{NP vor fi} has solutions being of the form $ X_{1}(s,\tau)=\lambda_{1,*}e^{-\mathbf{i}\alpha\tau}e^{\mathbf{i}\frac{s}{h}}, X_{2}(s,\tau)=-\lambda_{2,*}e^{-\mathbf{i}\alpha\tau}e^{\mathbf{i}\frac{s}{h}}  $ if and only if \eqref{compati cond1} holds. There are also  infinitely many  constant pairs $ (\kappa_1, \kappa_2,\lambda_{1,*},\lambda_{2,*}) $ satisfying \eqref{compati cond1} with $ \kappa_1\neq \kappa_2, \lambda_{1,*}\neq \lambda_{2,*} $. We have 
\begin{theorem}\label{thm03}
	Let $ R^*>0 , h>0 $ and $ N=2 $. For $ \kappa_1,\kappa_2>0 $ and $ \lambda_{1,*},\lambda_{2,*}\in (0,R^*) $   satisfying the compatibility condition \eqref{compati cond1},
let $ \tilde{X}_{j,\varepsilon} $ $(j=1,2)$ be $ 2 $ filaments satisfying \eqref{fila 3}. Then  there exists $ \varepsilon_0>0 $, such that for every $ \varepsilon\in(0,\varepsilon_0] $, \eqref{Euler eq2} has a family of solutions with vorticities  $ \mathbf{w}_\varepsilon\in C^1(B_{R^*}(0)\times \mathbb{R}\times\mathbb{R}^+) $ satisfying in distributional sense
	\begin{equation*} 
	\mathbf{w}_\varepsilon(x,t)-  \sum_{j=1}^2\kappa_j\delta_{\tilde{X}_{j,\varepsilon}}\textbf{t}_{\tilde{X}_{j,\varepsilon}}\rightharpoonup0\ \ \ \ \text{as}\ \varepsilon\to0.
	\end{equation*}
	Moreover, if we define   $ A_{\varepsilon,j}=supp\left (\mathbf{w}_\varepsilon(x_1,x_2,0,0)\right ) \cap B_{\rho_0|\ln\varepsilon|^{-\frac{1}{2}}}\left ( \tilde{X}_{j,\varepsilon}(0,0)\right )  $, then there exist constants $ R_1,R_2>0 $ independent of $ \varepsilon $  such that
	\begin{equation*}
	R_1\varepsilon \leq diam(A_{\varepsilon,j})\leq R_2\varepsilon .
	\end{equation*}
\end{theorem}

\textbf{Case 4:} ``Co-rotating  $ 2\times2 $ vortices" with different circulations and rotation radii. In this case, $ N= 4 $, $ \kappa_1=\kappa_3=\kappa>0 $, $ \kappa_2=\kappa_4=\mu>0$. $ \kappa,\mu, \lambda_{1,*},\lambda_{2,*}  $ are constants satisfying the compatibility condition
\begin{equation}\label{compa cond2}
\frac{\kappa}{4\pi h^2} -\frac{\kappa}{4\pi \lambda_{1,*}^2} -\frac{\mu}{\pi \left (\lambda_{1,*}^2+\lambda_{2,*}^2\right )}=\frac{\mu}{4\pi h^2} -\frac{\mu}{4\pi \lambda_{2,*}^2} -\frac{\kappa}{\pi \left (\lambda_{1,*}^2+\lambda_{2,*}^2\right )}.
\end{equation}
Then 
\begin{equation*} 
X_{2j-1}(s,\tau)=\lambda_{1,*}e^{-\mathbf{i}\alpha\tau}e^{\mathbf{i}\frac{s}{h}}e^{\mathbf{i}\pi(j-1)},\ \ \ \ 	X_{2j}(s,\tau)=\lambda_{2,*}e^{-\mathbf{i}\alpha\tau}e^{\mathbf{i}\frac{s}{h}}e^{\mathbf{i}\frac{\pi(2j-1)}{2}},\ \ j=1,2,
\end{equation*}
satisfies \eqref{NP vor fi}, where $ \alpha=\frac{\kappa}{4\pi h^2} -\frac{\kappa}{4\pi \lambda_{1,*}^2} -\frac{\mu}{\pi \left (\lambda_{1,*}^2+\lambda_{2,*}^2\right )}=\frac{\mu}{4\pi h^2} -\frac{\mu}{4\pi \lambda_{2,*}^2} -\frac{\kappa}{\pi \left (\lambda_{1,*}^2+\lambda_{2,*}^2\right )}  $. This corresponds to  
\begin{equation}\label{fila 4}
\tilde{X}_{2j-1,\varepsilon}(s,t)=\frac{\lambda_{1,*}}{\sqrt{|\ln\varepsilon|}}e^{-\mathbf{i}\alpha|\ln\varepsilon| t}e^{\mathbf{i}\frac{s}{h}}e^{\mathbf{i}\pi(j-1)},\ \  	\tilde{X}_{2j,\varepsilon}(s,t)=\frac{\lambda_{2,*}}{\sqrt{|\ln\varepsilon|}}e^{-\mathbf{i}\alpha|\ln\varepsilon| t}e^{\mathbf{i}\frac{s}{h}}e^{\mathbf{i}\frac{\pi(2j-1)}{2}},\ \ j=1,2.
\end{equation}
We have
\begin{theorem}\label{thm04}
	Let $ R^*>0, h>0, N=4 $. For $ \kappa,\mu>0 $ and $ \lambda_{1,*},\lambda_{2,*}\in (0,R^*) $ satisfying the compatibility condition \eqref{compa cond2},
let $ \tilde{X}_{j,\varepsilon} $ $(j=1,\cdots,4)$ be $ 4 $ filaments satisfying \eqref{fila 4}. Then  there exists $ \varepsilon_0>0 $, such that for every $ \varepsilon\in(0,\varepsilon_0] $, \eqref{Euler eq2} has a family of solutions with vorticities $ \mathbf{w}_\varepsilon\in C^1(B_{R^*}(0)\times \mathbb{R}\times\mathbb{R}^+) $ satisfying in distributional sense
	\begin{equation*} 
	\mathbf{w}_\varepsilon(x,t)-  \kappa\sum_{j=1}^2\delta_{\tilde{X}_{2j-1,\varepsilon}}\textbf{t}_{\tilde{X}_{2j-1,\varepsilon}}-\mu\sum_{j=1}^2\delta_{\tilde{X}_{2j,\varepsilon}}\textbf{t}_{\tilde{X}_{2j,\varepsilon}}\rightharpoonup0\ \ \ \ \text{as}\ \varepsilon\to0.
	\end{equation*}
	Moreover, if we denote   $ A_{\varepsilon,j}=supp\left (\mathbf{w}_\varepsilon(x_1,x_2,0,0)\right ) \cap B_{\rho_0|\ln\varepsilon|^{-\frac{1}{2}}}\left ( \tilde{X}_{j,\varepsilon}(0,0)\right )  $, then there exist constants $ R_1,R_2>0 $ independent of $ \varepsilon $  such that
	\begin{equation*}
	R_1\varepsilon \leq diam(A_{\varepsilon,j})\leq R_2\varepsilon .
	\end{equation*}
\end{theorem}

\textbf{Case 5:}  ``Co-rotating  $ 2\times2+1 $ vortices" with different circulations and rotation radii and the $ 5 $-th vortex of arbitrary strength
at the center of vorticity. Let $ N= 5 $, $ \kappa_1=\kappa_3=\kappa>0 $, $ \kappa_2=\kappa_4=\mu>0$, $ \kappa_0>0 $ and $ \lambda_{1,*},\lambda_{2,*}>0 $. $ \kappa_0,\kappa,\mu, \lambda_{1,*},\lambda_{2,*}  $  satisfy  the compatibility condition
\begin{equation}\label{compa cond3} 
\frac{\kappa}{4\pi h^2}-\frac{\kappa_0}{2\pi \lambda_{1,*}^2} -\frac{\kappa}{4\pi \lambda_{1,*}^2} -\frac{\mu}{\pi \left (\lambda_{1,*}^2+\lambda_{2,*}^2\right )}=\frac{\mu}{4\pi h^2}-\frac{\kappa_0}{2\pi \lambda_{2,*}^2} -\frac{\mu}{4\pi \lambda_{2,*}^2} -\frac{\kappa}{\pi \left (\lambda_{1,*}^2+\lambda_{2,*}^2\right )}.
\end{equation}
Then $ X_0(s,\tau)=0 $,
\begin{equation*}
X_{2j-1}(s,\tau)=\lambda_{1,*}e^{-\mathbf{i}\alpha\tau}e^{\mathbf{i}\frac{s}{h}}e^{\mathbf{i}\pi(j-1)},\ \ \ \ 	X_{2j}(s,\tau)=\lambda_{2,*}e^{-\mathbf{i}\alpha\tau}e^{\mathbf{i}\frac{s}{h}}e^{\mathbf{i}\frac{\pi(2j-1)}{2}},\ \ j=1,2,
\end{equation*}
satisfy \eqref{NP vor fi}, where $ \alpha=\frac{\kappa}{4\pi h^2}-\frac{\kappa_0}{2\pi \lambda_{1,*}^2} -\frac{\kappa}{4\pi \lambda_{1,*}^2} -\frac{\mu}{\pi \left (\lambda_{1,*}^2+\lambda_{2,*}^2\right )}=\frac{\mu}{4\pi h^2}-\frac{\kappa_0}{2\pi \lambda_{2,*}^2} -\frac{\mu}{4\pi \lambda_{2,*}^2} -\frac{\kappa}{\pi \left (\lambda_{1,*}^2+\lambda_{2,*}^2\right )}  $. This corresponds to  $ \tilde{X}_{0,\varepsilon}(s,t)=0 $,
\begin{equation} \label{fila 5}
\tilde{X}_{2j-1,\varepsilon}(s,t)=\frac{\lambda_{1,*}}{\sqrt{|\ln\varepsilon|}}e^{-\mathbf{i}\alpha|\ln\varepsilon|t}e^{\mathbf{i}\frac{s}{h}}e^{\mathbf{i}\pi(j-1)},\ \ 	\tilde{X}_{2j,\varepsilon}(s,t)=\frac{\lambda_{2,*}}{\sqrt{|\ln\varepsilon|}}e^{-\mathbf{i}\alpha|\ln\varepsilon|t}e^{\mathbf{i}\frac{s}{h}}e^{\mathbf{i}\frac{\pi(2j-1)}{2}},\ \ j=1,2.
\end{equation}
We have
\begin{theorem}\label{thm05}
	Let $ R^*>0, h>0, N=5 $. For $ \kappa_0,\kappa,\mu>0 $ and $ \lambda_{1,*},\lambda_{2,*}\in(0,R^*) $ satisfying \eqref{compa cond3}, 
let $ \tilde{X}_{j,\varepsilon} $ $(j=0,1,\cdots,4)$ be $ 5 $ filaments satisfying \eqref{fila 5}. Then  there exists $ \varepsilon_0>0 $, such that for every $ \varepsilon\in(0,\varepsilon_0] $, \eqref{Euler eq2} has a family of solutions with vorticities $ \mathbf{w}_\varepsilon\in C^1(B_{R^*}(0)\times \mathbb{R}\times\mathbb{R}^+) $ satisfying in distributional sense
	\begin{equation*} 
	\mathbf{w}_\varepsilon(x,t)- \kappa_0\delta_{\tilde{X}_{0,\varepsilon}}\textbf{t}_{\tilde{X}_{0,\varepsilon}}- \kappa\sum_{j=1}^2\delta_{\tilde{X}_{2j-1,\varepsilon}}\textbf{t}_{\tilde{X}_{2j-1,\varepsilon}}-\mu\sum_{j=1}^2\delta_{\tilde{X}_{2j,\varepsilon}}\textbf{t}_{\tilde{X}_{2j,\varepsilon}}\rightharpoonup0\ \ \ \ \text{as}\ \varepsilon\to0.
	\end{equation*}
	Moreover, if we define   $ A_{\varepsilon,j}=supp\left (\mathbf{w}_\varepsilon(x_1,x_2,0,0)\right ) \cap B_{\rho_0|\ln\varepsilon|^{-\frac{1}{2}}}\left ( \tilde{X}_{j,\varepsilon}(0,0)\right )  $, then there exist constants $ R_1,R_2>0 $ independent of $ \varepsilon $  such that
	\begin{equation*}
	R_1\varepsilon \leq diam(A_{\varepsilon,j})\leq R_2\varepsilon .
	\end{equation*}
\end{theorem}

\begin{remark}
It is possible to extend our results to more general cases, for example, ``$ 2N $ filaments": $ N $ of them locate at vertices of a regular $ N $-polygon with equal circulations $ \kappa_1 $ and rotation radii $ \lambda_1 $, and another $ N $ of them locate at vertices of another regular $ N $-polygon with  equal circulations $ \kappa_2 $ and rotation radii $ \lambda_2 $; and  ``$ 2N+1 $ filaments": $ 2N $ of them are ``$ 2N $ filaments" and the $ 2N+1 $ vortex locates at the center of vorticity with arbitrary positive circulation.

We note that recently, Cao et al. \cite{CFQW2} considered dynamics of nearly parallel helical vortex filaments of 3D incompressible Euler equations. If initial vorticity is
 concentrated in $ \varepsilon- $neighborhoods of $ N $ distinct helical filaments  with 
distances of order $ O(|\ln\varepsilon|^{-\frac{1}{2}}) $, then as time evolves the vorticity field of \eqref{Euler eq} still concentrates near helical vortex filaments evolved by  nearly parallel  vortex filaments model \ref{NP vor fi}, at least for a small time interval. For more results, see \cite{CFQW1}.

\end{remark}

Our proofs are based on the study of concentrated ``clustered solutions" to some semilinear second-order elliptic equations (see \eqref{key equa} in the following) with cluster point being the origin $ 0 $ and cluster distance being $ |\ln\varepsilon|^{-\frac{1}{2}} $, which means that solutions consist of several bubbles collapsing into $ 0 $ as parameter $ \varepsilon\to0 $  rather than several distinct points, and the distance of any two bubbles is $ O\left ( |\ln\varepsilon|^{-\frac{1}{2}}\right ) $. We call $ 0 $ the cluster point and  $ |\ln\varepsilon|^{-\frac{1}{2}} $  the cluster distance. Following the deduction argument in \cite{CW,CW1,DDMW2, ET,GM}, to solve helical vorticity field of \eqref{Euler eq2} with pitch $ h $ and the orthogonal condition $ \textbf{v}\cdot (x_2,-x_1,h)^T=0 $, it is equivalent to solve the following 2D system

\begin{equation}\label{vor str eq}
\begin{cases}
\partial_t \omega+\nabla \omega\cdot \nabla^\perp\varphi=0, &\Omega\times (0,T),\\
\omega=\mathcal{L}_{K_H}\varphi, &\Omega\times (0,T),\\
\varphi =0,& \partial\Omega,
\end{cases}
\end{equation}
where $ \omega  $ is the third component of $ \mathbf{w} $, $ \varphi $ is the corresponding stream function, $ \perp $ denotes the clockwise rotation through $ \frac{\pi}{2}  $, $ \mathcal{L}_{K_H} =-\text{div}(K_H(x_1,x_2)\nabla ) $ is   a second-order elliptic operator of divergence type with the coefficient matrix
\begin{equation}\label{coef matrix}
K_H(x_1,x_2)=\frac{1}{h^2+x_1^2+x_2^2}
\begin{pmatrix}
h^2+x_2^2 & -x_1x_2 \\
-x_1x_2 &  h^2+x_1^2
\end{pmatrix}.
\end{equation}

For $ \alpha \in\mathbb{R} $, if we consider    solutions to \eqref{vor str eq} with the form
\begin{equation}\label{104}
\omega(x,t)=w\left (\bar{R}_{-\alpha|\ln\varepsilon| t}(x)\right ),\ \ \varphi(x,t)=u\left (\bar{R}_{-\alpha|\ln\varepsilon| t}(x)\right ),
\end{equation}
where   $ \bar{R}_{\alpha t}=\begin{pmatrix}
\cos\alpha t & \sin\alpha t \\
-\sin\alpha t &\cos\alpha t
\end{pmatrix} $, then 
\begin{equation}\label{rot eq-1}
\begin{cases}
\nabla w\cdot \nabla^\perp \left( u-\frac{\alpha}{2}|x|^2|\ln\varepsilon|\right) =0,\ \ &\Omega,\\
w=\mathcal{L}_{K_H}u,\ \ &\Omega,\\
u=0,\ \ &\partial\Omega.
\end{cases}
\end{equation}
A sufficient condition for $ (w,u) $ satisfying \eqref{rot eq-1} is that $ u $ solves 
\begin{equation}\label{105-1}
\mathcal{L}_{K_H} u=w=f_\varepsilon\left( u-\frac{\alpha}{2}|x|^2|\ln\varepsilon|\right)\ \  \text{in}\ \Omega,\ \ u=0\ \ \text{on}\ \partial\Omega,
\end{equation}
for some profile function $ f_\varepsilon $. Thus the key is to find concentrated clustered solutions to \eqref{105-1} under the proper choice of $ f_\varepsilon $. To make the support set of cross-section of $ \textbf{w}_\varepsilon $ compact, we choose $ f_\varepsilon(t)=\frac{1}{\varepsilon^2}t^p_+ $ with $ p>1 $ in \eqref{105-1}.

Let us directly consider   clustered solutions to the following general semilinear second-order elliptic equations 
\begin{equation}\label{key equa}
\begin{cases}
-\varepsilon^2\text{div}\left (K(x)\nabla v\right )= \left (v-q|\ln\varepsilon|\right )^{p}_+,\ \ &x\in \Omega,\\
v=0,\ \ &x\in\partial \Omega,
\end{cases}
\end{equation}
with cluster point being the origin $ 0 $ and cluster distance being $ |\ln\varepsilon|^{-\frac{1}{2}} $, where $ \Omega\subset \mathbb{R}^2 $ is a smooth bounded domain containing $ 0 $, $ \varepsilon\in(0,1) $. $ K=(K_{i,j})_{2\times2} $  and $ q $ satisfy the following assumptions
\begin{enumerate}
	\item[(K1).] 
	There exist  $ \Lambda_1,\Lambda_2>0 $ such that $$ \Lambda_1|\zeta|^2\le (K(x)\zeta|\zeta) \le \Lambda_2|\zeta|^2,\ \ \ \ \forall\ x\in \Omega, \ \zeta\in \mathbb{R}^2.$$
	\item[(Q1).]  $ q \in C^{\infty}(\overline{\Omega}) $ and $ q(x)>0 $ for any $ x\in\overline{\Omega}. $
	\item[(KQ).]   $ \nabla \left (q^2\sqrt{\det K}\right )(0)=0 $.
\end{enumerate}

Let us remark that the assumption $ (KQ) $ means that $ 0 $ is a critical point of $ q^2\sqrt{\det K} $, which seems to be artificial. However, when we prove Theorems \ref{thm01}-\ref{thm05} by the special form of $q$ and $K$, this assumption is satisfied automatically.

For any given integer $ N\geq 2 $, we define an auxiliary function from $ \Omega^{(N)} $ to $ \mathbb{R} $
\begin{equation}\label{def of H_N}
\begin{split}
H_N(\tilde{z}_1,\cdots, \tilde{z}_N)=&\frac{1}{2}\sum_{j=1}^N\tilde{z}_j\cdot \nabla^2\left(q^2\sqrt{\det K} \right)(0 )\cdot\tilde{z}_j^T\\
&+\sum_{1\leq i\neq j\leq N}\left( q^2  \sqrt{\det K}\right)(0)\ln| K(0)^{-\frac{1}{2}}\left( \tilde{z}_i - \tilde{z}_j\right)  |.
\end{split}
\end{equation}
Here $ \tilde{z}_j^T $ is the transpose of $ \tilde{z}_j $.
Under  assumptions $ (K1),(Q1)$  and  $(KQ) $, we can obtain the existence of clustered solutions to \eqref{key equa} with cluster point $ 0 $ and cluster distance $ |\ln\varepsilon|^{-\frac{1}{2}} $, whose bubbles concentrate near strict local maximizers (or minimizers) of $ H_N $ after $ |\ln\varepsilon|^{\frac{1}{2}} $ scaling. Indeed, we have the following result.
\begin{theorem}\label{thm0}
	Let $(K1)$, $ (Q1) $ and $ (KQ) $ hold.  
	Let $ \left( \tilde{z}_1^*,\cdots, \tilde{z}_N^*\right)  $ be a strict local maximizer or minimizer of $ H_N $ defined by \eqref{def of H_N}, i.e., there exists a $ \bar{\rho}- $neighborhood $ B_{\bar{\rho}} $ of $ \left( \tilde{z}_1^*,\cdots, \tilde{z}_N^*\right) $ for some $ \bar{\rho}>0 $ small  such that
	\begin{equation*}
	H_N\left( \tilde{z}_1,\cdots, \tilde{z}_N\right)<(\text{or} >) H_N\left( \tilde{z}_1^*,\cdots, \tilde{z}_N^*\right)\ \ \ \ \forall \left( \tilde{z}_1,\cdots, \tilde{z}_N\right)\in B_{\bar{\rho}}\backslash\{\left( \tilde{z}_1^*,\cdots, \tilde{z}_N^*\right) \}.
	\end{equation*} 
	Then  there exists $ \varepsilon_0>0 $, such that for every $ \varepsilon\in(0,\varepsilon_0] $, \eqref{key equa} has a  solution  $ v_\varepsilon $. Moreover,  
	\begin{enumerate}
		\item There exist $ \left (\tilde{z}_{1,\varepsilon}, \cdots, \tilde{z}_{N,\varepsilon}\right ) $ such that $ \tilde{z}_{i,\varepsilon}=\tilde{z}_{i}^*+o(1) $ for $ i=1,\cdots,N $ and 
		\begin{equation*}
		\left \{v_\varepsilon>q|\ln\varepsilon|\right \}\subseteq \cup_{i=1}^N B_{|\ln\varepsilon|^{-1}}\left (\frac{\tilde{z}_{i,\varepsilon}}{\sqrt{|\ln\varepsilon|}}\right ).
		\end{equation*}
		
		\item Define  sets $ A_{\varepsilon,i}=\left\{v_\varepsilon>q|\ln\varepsilon|\right\}\cap B_{|\ln\varepsilon|^{-1}}\left (\frac{\tilde{z}_{i,\varepsilon}}{\sqrt{|\ln\varepsilon|}}\right )  $. Then there exist constants $ R_1,R_2>0 $ independent of $ \varepsilon $  such that
		\begin{equation*}
		B_{R_1\varepsilon}\left (\frac{\tilde{z}_{i,\varepsilon}}{\sqrt{|\ln\varepsilon|}}\right )\subseteq A_{\varepsilon,i}\subseteq B_{R_2\varepsilon}\left (\frac{\tilde{z}_{i,\varepsilon}}{\sqrt{|\ln\varepsilon|}}\right ).
		\end{equation*}
		\item 
		\begin{equation*}
		\frac{1}{\varepsilon^2}\int_{A_{\varepsilon,i}}\left (v_\varepsilon-q|\ln \varepsilon|\right )^{p}_+dx=2\pi \left( q\sqrt{\det K}\right) (0)+O\left (\frac{1}{\sqrt{|\ln\varepsilon|}}\right ).
		\end{equation*}
	\end{enumerate}
\end{theorem}

Indeed, by choosing properly $ K $ and $ q $ in Theorem \ref{thm0}, we can directly prove Theorems \ref{thm01}, \ref{thm02}, \ref{thm03}, \ref{thm04} and \ref{thm05}.

\begin{remark}
Indeed, if the assumption $ (KQ) $ fails, using our idea of proof of Theorem \ref{thm0} it is still possible to construct clustered solutions to \eqref{key equa} with cluster point $ x_0\neq 0  $ and cluster radius $ |\ln\varepsilon|^{-1} $. In \cite{GM}, the authors used gluing method to construct clustered solutions to \eqref{key equa} with cluster point $ x_0\neq 0  $ and cluster radius $ |\ln\varepsilon|^{-1} $ when $ \Omega=\mathbb{R}^2 $, $ K=K_H $ and profile functions $ f_\varepsilon  $ are exponential type functions. Note that $ x_0 $ is not a critical point of $ q^2\sqrt{\det K_H} $ in their results. Using methods in our paper, it is possible to generalize results in \cite{GM}, that is, constructing clustered solutions to \eqref{key equa} in case that  $ \Omega $ is  any bounded domain and $ K $ is any positive-definite matrix.
\end{remark}

Our proof of Theorem \ref{thm0} is based on proper decomposition of the Green's function of elliptic operators in divergence form and the finite-dimensional reduction method. 
To get solutions to \eqref{key equa}, we  first  give the $ C^1- $expansion  of Green's function $ G_K $ of the operator $ -\text{div}(K(x)\nabla) $. 
 Then, we   give ansatz of approximate solutions of \eqref{key equa} of the form $ \sum_{j=1}^N\left (V_{\varepsilon, z_j, \hat{q}_j}+H_{\varepsilon, z_j, \hat{q}_j}\right )+\omega_{\varepsilon,Z} $,  where $ V_{\varepsilon, z_j, \hat{q}_j} $, $ H_{\varepsilon, z_j, \hat{q}_j} $ and $ \omega_{\varepsilon,Z}$ are the main term,  projection term and  error term respectively. The key ingredient is that the configuration space  $ \Lambda_{\varepsilon, N} $ and   parameters $ \hat{q}_j $ must be chosen properly   
to ensure that 
 the difference between $ V_{\varepsilon,Z}-q|\ln\varepsilon| $ and $V_{\varepsilon, z_i, \hat{q}_{i}}-\hat{q}_{i}|\ln\varepsilon| $ is small in some neighborgood of $ z_i $, see \eqref{203}. Then we give the linear and nonlinear theory of the linearized operator of \eqref{key equa} at $ \sum_{j=1}^N\left (V_{\varepsilon, z_j, \hat{q}_j}+H_{\varepsilon, z_j, \hat{q}_j}\right ) $.  
We get the existence, uniqueness  and the $ C^1 $ differentiability of $ \omega_{\varepsilon,Z}  $ with respect to $ Z $, see  Proposition \ref{exist and uniq of w} and  
 Proposition \ref{differ of w}. Finally, by calculating   coefficients of each order of energy $ K_\varepsilon(Z) $, that is, the coefficients of $ O(|\ln\varepsilon|), O(\ln|\ln\varepsilon|) $ and $ O(1) $ terms, 
we prove the existence of   clustered solutions to \eqref{key equa} with cluster point $ 0 $ and cluster radius $  |\ln\varepsilon|^{-\frac{1}{2}} $.

We give a remark that model \eqref{NP vor fi} has also appeared in many other PDE models, such as Gross-Pitaevskii equations and Ginzburg-Landau equations. \cite{CJ1} established connection between the Ginzburg–Landau functional and the energy of nearly parallel filaments, and constructed  solutions of the Ginzburg–Landau equations  whose small-scale structure is governed by the Ginzburg–Landau functional. \cite{JS2} rigorously
derived the corresponding asymptotic motion law in the context of the Gross–Pitaevskii
equation.  As for the helix,  \cite{DDMR,DDMR2} considered interacting helical traveling waves
for the Gross-Pitaevskii equations and the Ginzburg-Landau equations. For more results, see \cite{Cho, WY}.

This paper is organized as follows.  
In section 2, we give the $ C^1- $expansion of Green's function $ G_K $, the definition of   $ \Lambda_{\varepsilon, N} $ and ansatz of approximate solutions. In section 3, we study the linear theory of \eqref{key equa}. Using the non-degeneracy of solutions to limiting equations \eqref{limit eq}, we get coercive estimates of the linearized operator $ Q_\varepsilon L_\varepsilon $. In section 4, the nonlinear theory of \eqref{key equa} is considered. The existence and  the differentiability of the error term $ \omega_{\varepsilon,Z} $   are proved.
In sections 5 and 6, we give the main terms and remaining higher order term of the energy functional $ K_\varepsilon  $ and   complete the proof of Theorem \ref{thm0}. By proper choice of $ K $ and $ q $ in Theorem \ref{thm0}, we give proofs of  Theorems \ref{thm01},\ref{thm02},\ref{thm03},\ref{thm04} and \ref{thm05} in section 7. Some calculations such as  exact helical solutions of \eqref{NP vor fi} and $ C^1-$expansion of Green’s function are given in Appendix.

\section{Green's function and approximate solutions}

To  prove Theorem \ref{thm0}, the key ingredient is to show the existence of clustered solutions to \eqref{key equa}:
\begin{equation*} 
\begin{cases}
-\varepsilon^2\text{div}(K(x)\nabla v)=  \left (v-q(x)|\ln\varepsilon| \right )^{p}_+,\ \ &\text{in}\  \Omega,\\
v=0,\ \ &\text{on}\ \partial \Omega.
\end{cases}
\end{equation*}
Here $ 0<\varepsilon<1, $ $ p>1 $, $ K $ and $ q $ satisfy $(K1)$, $ (Q1) $ and $ (KQ) $. 
The purpose of this section is to give the $ C^1-$expansion of Green's function $ G_{K} $ of the elliptic operator $  -\text{div}(K(x)\nabla \cdot) $ and give the ansatz of  approximate solutions to \eqref{key equa}. 

\subsection{$ C^1-$Expansion of Green's function $ G_K $}
The expansion of Green's function $ G_K $ of the elliptic operator $ -\text{div}(K(x)\nabla \cdot) $ with  Dirichlet condition  plays an essential role, since it appears in the construction of approximate solutions of \eqref{key equa} and the $ C^1-$regularity of the finite-dimensional functional.
Let $ G_K(x,y) $ be the Green's function of $ -\text{div}(K(x)\nabla \cdot) $ with zero Dirichlet condition on  the boundary of $ \Omega $, that is, solutions of the following linear elliptic problem:
\begin{equation}\label{diver form}
\begin{cases}
-\text{div}(K(x)\nabla u)= f,\ \ & \Omega,\\
u=0,\ \ & \partial \Omega
\end{cases}
\end{equation}
can be expressed by $ u(x)=\int_{\Omega}G_K(x,y)f(y)dy $ for $ x\in \Omega $, see \cite{GT}.

  Note that \cite{CW,CW1}  obtained the following $ C^0-$asymptotic expansion of $ G_K(x,y)$. 
\begin{lemma}[lemma 3.1, \cite{CW}]\label{Green expansion}
For $ y\in\Omega $, let $ T_y=T(y) $ be the unique positive-definite matrix satisfying
$$ T_y^{-1}(T_y^{-1})^{T}=K(y). $$ 
Then for any $ \gamma\in(0,1) $, there exists a function $\bar{S}_K\in C^{0,\gamma}_{loc} (\Omega\times \Omega) $ such that
\begin{equation*}
G_K(x,y)=\sqrt{\det K(y)}^{-1}\Gamma\left (T_y(x-y)\right )+\bar{S}_K(x,y),\ \ \forall\ x,y\in\Omega.
\end{equation*}
\end{lemma}
For $ i,j=1,2 $, we denote $ T_{ij}=(T_y)_{ij}=(T(y))_{ij} $ the component of row $ i $, column $ j $ of the matrix $ T_y $. Based on Lemma \ref{Green expansion} and some more detailed calculations, we give $ C^1- $asymptotic expansion of $ G_K(x,y) $ as follows, the proof of which will be given in the Appendix (see Appendix A.2). 
\begin{lemma}\label{Green expansion2}
There holds
\begin{equation*}
\bar{S}_K(x,y)=-F_{1,y}(x)-F_{2,y}(x)+\bar{H}_1(x,y) \ \ \ \ \forall\ x,y\in\Omega,
\end{equation*}
where
\begin{equation}\label{exp of F_{1,y}}
\begin{split}
F_{1,y}(x)=-\frac{1}{4\pi}\sqrt{\det K(y)}^{-1}\sum_{i,j,m=1}^2T_{mj}\partial_{x_i}K_{ij}(y)\left( T_y(x-y)\right)_m\ln|T_y(x-y)|,
\end{split}
\end{equation}
\begin{equation*}
\begin{split}
F_{2,y}(x)=&\frac{1}{\pi}\sqrt{\det K(y)}^{-1}\sum_{i,j,\alpha=1}^2\partial_{x_\alpha}K_{ij}(y)\cdot\\
&\bigg\{T^{-1}_{\alpha1}T_{1j}T_{1i}\left( -\frac{1}{8}\frac{\left( T_y\left (x-y\right )\right) _1^3}{|T_y\left (x-y\right )|^2}+\frac{1}{8}\left( T_y\left (x-y\right )\right)_1\ln|T_y\left (x-y\right )|\right) \\
+&T^{-1}_{\alpha1}T_{1j}T_{2i}\left( -\frac{1}{8}\frac{\left( T_y\left (x-y\right )\right) _1^2\left( T_y\left (x-y\right )\right) _2}{|T_y\left (x-y\right )|^2}+\frac{1}{8}\left( T_y\left (x-y\right )\right)_2\ln|T_y\left (x-y\right )|\right) \\
+&T^{-1}_{\alpha1}T_{2j}T_{1i}\left( -\frac{1}{8}\frac{\left( T_y\left (x-y\right )\right) _1^2\left( T_y\left (x-y\right )\right) _2}{|T_y\left (x-y\right )|^2}+\frac{1}{8}\left( T_y\left (x-y\right )\right)_2\ln|T_y\left (x-y\right )|\right)\\
+&T^{-1}_{\alpha1}T_{2j}T_{2i}\left( -\frac{1}{8}\frac{\left( T_y\left (x-y\right )\right) _1\left( T_y\left (x-y\right )\right) _2^2}{|T_y\left (x-y\right )|^2}-\frac{1}{8}\left( T_y\left (x-y\right )\right)_1\ln|T_y\left (x-y\right )|\right)
\end{split}
\end{equation*}
\begin{equation}\label{exp of F_{2,y}}
\begin{split}
+&T^{-1}_{\alpha2}T_{1j}T_{1i}\left( -\frac{1}{8}\frac{\left( T_y\left (x-y\right )\right) _1^2\left( T_y\left (x-y\right )\right) _2}{|T_y\left (x-y\right )|^2}-\frac{1}{8}\left( T_y\left (x-y\right )\right)_2\ln|T_y\left (x-y\right )|\right)\\
+&T^{-1}_{\alpha2}T_{1j}T_{2i}\left( -\frac{1}{8}\frac{\left( T_y\left (x-y\right )\right) _1\left( T_y\left (x-y\right )\right) _2^2}{|T_y\left (x-y\right )|^2}+\frac{1}{8}\left( T_y\left (x-y\right )\right)_1\ln|T_y\left (x-y\right )|\right)\\
+&T^{-1}_{\alpha2}T_{2j}T_{1i}\left( -\frac{1}{8}\frac{\left( T_y\left (x-y\right )\right) _1\left( T_y\left (x-y\right )\right) _2^2}{|T_y\left (x-y\right )|^2}+\frac{1}{8}\left( T_y\left (x-y\right )\right)_1\ln|T_y\left (x-y\right )|\right)\\
+&T^{-1}_{\alpha2}T_{2j}T_{2i}\left( -\frac{1}{8}\frac{\left( T_y\left (x-y\right )\right) _2^3 }{|T_y\left (x-y\right )|^2}+\frac{1}{8}\left( T_y\left (x-y\right )\right)_2\ln|T_y\left (x-y\right )|\right)\bigg\},
\end{split}
\end{equation}
 and $ x\to \bar{H}_1(x,y)\in C^{1,\gamma}(\overline{\Omega})$ for all $ y\in \Omega $, $ \gamma\in (0, 1) $.  Moreover, the function $(x, y)\to \bar{H}_1(x,y)\in C^1(\Omega\times\Omega),$ and in particular the corresponding Robin function $x\to \bar{S}_K(x,x)\in C^1(\Omega)$.
\end{lemma}

\begin{remark}
From Lemma \ref{Green expansion2}, Green's function $ G_K $ has an expansion
\begin{equation*} 
G_K(x,y)=\sqrt{\det K(y)}^{-1}\Gamma\left (T_y(x-y)\right )-F_{1,y}(x)-F_{2,y}(x)+\bar{H}_1(x,y),\ \ \forall\ x,y\in\Omega.
\end{equation*}

It is also noted that recently, \cite{CW3} constructed the asymptotic expansion of Green's function and the regularity of Robin's function of an elliptic operator in divergence form in bounded domains in all dimensions. Especially when the dimension is 2 and the coefficient matrix is $ K_H $, they proved that $ \bar{S}_K $ is smooth in $ \Omega. $ In contrast to \cite{CW3}, we  calculate  in detail the specific expansion of  Green's function.
\end{remark}

\subsection{Ansatz for approximate solutions}

The argument that  $ \bar{S}_K(x,x)$ is a $ C^1 $  function  plays an important role in our argument to get the $ C^1 $ regularity
of the error term and the finite-dimensional variational functional with respect to parameter $ Z $, see sections 4 and 5.

We give approximate solutions of \eqref{key equa} and   the configuration space $ \Lambda_{\varepsilon, N} $ for the parameter $ Z=(z_1,\cdots,z_N) $  as follows. Let $ N> 1 $ be an integer.  For any $ \hat{x}\in \Omega, \hat{q}>0 $, we define
\begin{equation}\label{eq2-2}
V_{\varepsilon, \hat{x}, \hat{q}}(x)=\begin{cases}
\hat{q}\ln \frac{1}{\varepsilon}+\varepsilon^{\frac{2}{p-1}}s_\varepsilon^{-\frac{2}{p-1}}\phi\left(\frac{|T_{\hat{x}}(x-\hat{x})|}{s_\varepsilon}\right),\ \ &|T_{\hat{x}}(x-\hat{x})|\le s_\varepsilon,\\
 \hat{q} \ln \frac{1}{\varepsilon}\frac{\ln |T_{\hat{x}}(x-\hat{x})|}{\ln s_\varepsilon},\ \  &  |T_{\hat{x}}(x-\hat{x})|>s_\varepsilon,
\end{cases}
\end{equation}
where $ T_{\hat{x}}^{-1}\left (T_{\hat{x}}^{-1}\right )^{T}=K\left (\hat{x}\right ) $, $ \phi\in H^1_0(B_1(0)) $ satisfies (see, e.g., \cite{CLW})
\begin{equation*}
-\Delta\phi=\phi^p, \ \ \phi>0\ \ \text{in}\ B_1(0),
\end{equation*}
and $ s_\varepsilon $ satisfies
\begin{equation}\label{201}
\varepsilon^{\frac{2}{p-1}}s_\varepsilon^{-\frac{2}{p-1}}\phi'(1)=\hat{q}\frac{ \ln\frac{1}{\varepsilon}}{\ln s_\varepsilon}.
\end{equation}
Clearly, $ V_{\varepsilon, \hat{x}, \hat{q}}\in C^1 $ satisfies
\begin{equation}\label{eq2}
\begin{cases}
-\varepsilon^2\text{div}\left (K(\hat{x})\nabla V_{\varepsilon, \hat{x}, \hat{q}}\right )=  \left (V_{\varepsilon, \hat{x}, \hat{q}}-\hat{q}\ln\frac{1}{\varepsilon}\right )^{p}_+,\ \ & \text{in}\  \mathbb{R}^2,\\
V_{\varepsilon, \hat{x}, \hat{q}}=\hat{q}\ln\frac{1}{\varepsilon},\ \ &\text{on} \ \{x\mid |T_{\hat{x}}(x-\hat{x})|=s_\varepsilon\},
\end{cases}
\end{equation}
and  for $ \varepsilon $ sufficiently small, \eqref{201} is uniquely solvable with
\begin{equation*}
\frac{s_\varepsilon}{\varepsilon}\to \left( \frac{|\phi'(1)|}{\hat{q}}\right) ^{\frac{p-1}{2}}\ \ \ \ \text{as}\ \varepsilon\to0.
\end{equation*}
The Pohozaev identity implies
\begin{equation}\label{PI}
\int_{B_1(0)}\phi^{p+1}=\frac{\pi(p+1)}{2}|\phi'(1)|^2,\ \ \int_{B_1(0)}\phi^{p}= 2\pi|\phi'(1)|.
\end{equation}

Since $ V_{\varepsilon, \hat{x}, \hat{q}} $ is not 0 on $ \partial \Omega $, we need to make a projection on $ H^1_0(\Omega). $ 
Let $ H_{\varepsilon, \hat{x}, \hat{q}} $ be a solution of the following equations
\begin{equation}\label{eq3}
\begin{cases}
-\text{div}\left (K(x)\nabla H_{\varepsilon, \hat{x}, \hat{q}}\right )=\text{div}\left (\left (K(x)-K(\hat{x})\right )\nabla V_{\varepsilon, \hat{x}, \hat{q}}\right ),\ \ &\Omega,\\
H_{\varepsilon, \hat{x}, \hat{q}}=-V_{\varepsilon, \hat{x}, \hat{q}},\ \ &\partial\Omega.
\end{cases}
\end{equation}
Then  $ H_{\varepsilon, \hat{x}, \hat{q}}\in W^{2,p}(\Omega)\subset  C^{1,\alpha}(\overline{\Omega}) $ for any $ p>1, \alpha\in(0,1) $. Similar to lemma 3.2 in \cite{CW}, we give estimates of the difference between $ H_{\varepsilon, \hat{x}, \hat{q}} $ and  $ -\frac{2\pi\hat{q}\sqrt{\det K(\hat{x})}\ln\frac{1}{\varepsilon}}{\ln s_\varepsilon}\bar{S}_K(\cdot,\hat{x}) $.

\begin{lemma}\label{H estimate}
Define $ \zeta_{\varepsilon, \hat{x}, \hat{q}}(x)=H_{\varepsilon, \hat{x}, \hat{q}}(x)+\frac{2\pi\hat{q}\sqrt{\det K(\hat{x})}\ln\frac{1}{\varepsilon}}{\ln s_\varepsilon}\bar{S}_K(x,\hat{x}) $ for $ x\in\Omega $. Then for any $ p\in(1,2) $, there exists a constant $ C>0 $ independent of $ \varepsilon $ such that
\begin{equation*}
||\zeta_{\varepsilon, \hat{x}, \hat{q}}||_{C^{0,2-\frac{2}{p}}(\Omega)}\leq Cs_\varepsilon^{\frac{2}{p}-1}.
\end{equation*}
\end{lemma}
Since the idea of proof is similar to that in lemma 3.2 in \cite{CW}, we just give a sketch of   proof for the above lemma   in the Appendix (see Appendix A.3).

Using \eqref{eq2-2}, the definition of $ H_{\varepsilon, \hat{x}, \hat{q}} $ in \eqref{eq3} and the classical $ L^p $-theory of elliptic equations, it is also not hard to get 
\begin{equation}\label{3-001}
||H_{\varepsilon, \hat{x}, \hat{q}}||_{W^{2,p}(\Omega)}\leq
\begin{cases}
\frac{C}{\varepsilon^{1-\frac{2}{p}}},\ \ &p>2,\\
C|\ln\varepsilon|,\ \ &p=2,\\
C,\ \ &1\leq p<2.
\end{cases}
\end{equation}

Let $ \left( \tilde{z}_1^*,\cdots, \tilde{z}_N^*\right)  $ be a local strict maximizer or minimizer of $ H_N $. The configuration space $ \Lambda_{\varepsilon, N} $ for $ Z=(z_1,\cdots,z_N) $ is defined as follows:
\begin{equation}\label{admis set}
\begin{split}
\Lambda_{\varepsilon, N}=\bigg\{&Z=(z_1,\cdots, z_N)\in\Omega^{(N)}\mid z_i\in B_{\frac{\rho_0}{2\sqrt{|\ln\varepsilon|}}}\left ( \frac{\tilde{z}_i^*}{\sqrt{|\ln\varepsilon|}}\right )\bigg\}.
\end{split}
\end{equation}
Here $ \rho_0>0 $ is small. Note that by \eqref{admis set},
\begin{equation}\label{3-003}
G_K(z_i,z_j)\leq C|\ln|z_i-z_j||\leq C\ln|\ln\varepsilon|,\ \ \ \ \forall\  Z\in \Lambda_{\varepsilon, N}.
\end{equation}

In the following, we will construct solutions of \eqref{key equa} being of the form
\begin{equation}\label{solu config}
v_\varepsilon= V_{\varepsilon, Z}+\omega_\varepsilon=\sum_{j=1}^NV_{\varepsilon, Z,j}+\omega_{\varepsilon,Z}=\sum_{j=1}^N\left (V_{\varepsilon, z_j, \hat{q}_j}+H_{\varepsilon, z_j, \hat{q}_j}\right )+\omega_{\varepsilon,Z},
\end{equation}
where $ Z\in \Lambda_{\varepsilon, N} $, $ \sum_{j=1}^NV_{\varepsilon, Z,j} $ is the main term and $ \omega_{\varepsilon,Z}$ is an error term. The choice of $ \hat{q}_j $ will be determined later on. First from \eqref{key equa}, one computes directly that
\begin{equation*}
\begin{split}
0=&\sum_{j=1}^N-\varepsilon^2\text{div}(K(x)\nabla (V_{\varepsilon, z_j, \hat{q}_j}+H_{\varepsilon, z_j, \hat{q}_j}))-\varepsilon^2\text{div}(K(x)\nabla \omega_{\varepsilon,Z})-\left (V_{\varepsilon, Z}+\omega_{\varepsilon,Z}-q|\ln\varepsilon|\right )^{p}_+\\
=&-\sum_{j=1}^N\varepsilon^2\text{div}(K(z_j)\nabla V_{\varepsilon, z_j, \hat{q}_j})-\sum_{j=1}^N\varepsilon^2\text{div}((K(x)-K(z_j))\nabla V_{\varepsilon, z_j, \hat{q}_j})\\
&-\sum_{j=1}^N\varepsilon^2\text{div}(K(x)\nabla H_{\varepsilon, z_j, \hat{q}_j})+\left (-\varepsilon^2\text{div}(K(x)\nabla \omega_{\varepsilon,Z})-p\left ( V_{\varepsilon, Z}-q|\ln\varepsilon|\right )^{p-1}_+\omega_{\varepsilon,Z} \right )\\
&-\left (\left (V_{\varepsilon, Z}+\omega_{\varepsilon,Z}-q|\ln\varepsilon|\right )^{p}_+-p\left (V_{\varepsilon, Z}-q|\ln\varepsilon|\right )^{p-1}_+\omega_{\varepsilon,Z}\right ).
\end{split}
\end{equation*}
Thus it suffices to find solutions  $ \omega=\omega_{\varepsilon,Z} $ of the following equation 
\begin{equation}\label{3-03}
L_\varepsilon \omega=l_\varepsilon+ R_\varepsilon(\omega),
\end{equation}
where
\begin{equation*}
l_\varepsilon:=\left (\sum_{j=1}^NV_{\varepsilon, Z,j}-q|\ln\varepsilon|\right )^{p}_+-\sum_{j=1}^N\left (V_{\varepsilon, z_j, \hat{q}_j}-\hat{q}_j|\ln\varepsilon|\right )^{p}_+,
\end{equation*}
$ L_\varepsilon $ is the linearized operator of \eqref{key equa} at $  V_{\varepsilon, Z} $ defined by
\begin{equation*}
L_\varepsilon \omega:=-\varepsilon^2\text{div}(K(x)\nabla \omega)- p\left (\sum_{j=1}^NV_{\varepsilon, Z,j}-q|\ln\varepsilon|\right )^{p-1}_+\omega,
\end{equation*}
and $ R_\varepsilon(\omega_{\varepsilon,Z}) $ is the high-order error term defined by
\begin{equation*}
\begin{split}
R_\varepsilon(\omega):=&\left (V_{\varepsilon, Z}+\omega-q|\ln\varepsilon|\right )^{p}_+-\left (V_{\varepsilon, Z}-q|\ln\varepsilon|\right )^{p}_+-p\left (V_{\varepsilon, Z}-q|\ln\varepsilon|\right )^{p-1}_+\omega.
\end{split}
\end{equation*}

Next we will choose parameters $ \hat{q}_j $ suitably to make the term $ l_\varepsilon $ as small as possible in some norm. For any $ Z\in \Lambda_{\varepsilon, N}, $ let $ \hat{q}_i=\hat{q}_{\varepsilon,i}(Z) $   satisfy for $ i=1,\cdots,N $
\begin{equation}\label{q_i choice}
\hat{q}_i=q(z_i)+\frac{2\pi\hat{q}_i\sqrt{\det K(z_i)}}{\ln s_{\varepsilon,i}}\bar{S}_K(z_i, z_i)+\sum_{j\neq i}\frac{2\pi\hat{q}_j\sqrt{\det K(z_j)}}{\ln s_{\varepsilon,j}}G_K(z_i, z_j),
\end{equation}
where $ s_{\varepsilon,i} $ satisfies for $ i=1,\cdots,N $
\begin{equation*}
\varepsilon^{\frac{2}{p-1}}s_{\varepsilon,i}^{-\frac{2}{p-1}}\phi'(1)=\hat{q}_i\frac{ \ln\frac{1}{\varepsilon}}{\ln s_{\varepsilon,i}}.
\end{equation*}

Then from \cite{Ku}, 
for any $ \varepsilon $ sufficiently small  and $ Z\in \Lambda_{\varepsilon, N} $, there exists $ \hat{q}_{\varepsilon,i}(Z) $ satisfying \eqref{q_i choice}. Moreover, $ \hat{q}_i>0 $ and 
\begin{equation}\label{2000}
\hat{q}_i=q(z_i)+O\left (\frac{\ln|\ln\varepsilon|}{|\ln\varepsilon|}\right );\ \ \frac{1}{\ln\frac{1}{s_{\varepsilon,i}}}=\frac{1}{\ln\frac{1}{\varepsilon}}+O\left( \frac{\ln|\ln\varepsilon|}{|\ln\varepsilon|^2}\right).
\end{equation}

By the choice of $ \Lambda_{\varepsilon, N} $ in \eqref{admis set} and  $ \hat{q}_{\varepsilon,j} $ in \eqref{q_i choice}, we claim that for any $ Z\in \Lambda_{\varepsilon, N} $, $ \gamma\in(0,1), L>1 $ and $ x\in B_{Ls_{\varepsilon,i}}(z_i) $
\begin{equation}\label{203}
\begin{split}
\sum_{j=1}^NV_{\varepsilon, Z,j}(x)-q(x)|\ln\varepsilon|=V_{\varepsilon, z_i, \hat{q}_{\varepsilon,i}}(x)-\hat{q}_{\varepsilon,i}|\ln\varepsilon|+O\left( \varepsilon^\gamma\right).
\end{split}
\end{equation}
Indeed, for $ x\in B_{Ls_{\varepsilon,i}}(z_i) $
\begin{equation*}
\begin{split}
V_{\varepsilon,Z,i}(x)&-q(x)|\ln\varepsilon|\\
=&V_{\varepsilon, z_i, \hat{q}_{\varepsilon,i}}(x)+H_{\varepsilon, z_i, \hat{q}_{\varepsilon,i}}(x)-q(x)|\ln\varepsilon|\\
=&V_{\varepsilon, z_i, \hat{q}_{\varepsilon,i}}(x)-q(z_i)|\ln\varepsilon|-\frac{2\pi\hat{q}_{\varepsilon,i}\sqrt{\det K(z_i)}|\ln\varepsilon|}{\ln s_{\varepsilon,i}}\bar{S}_K(x, z_i)+O(s_{\varepsilon,i}|\ln\varepsilon|)+O\left( s_{\varepsilon,i}^{\gamma}\right) \\
=&V_{\varepsilon, z_i, \hat{q}_{\varepsilon,i}}(x)-q(z_i)|\ln\varepsilon|-\frac{2\pi\hat{q}_{\varepsilon,i}\sqrt{\det K(z_i)}|\ln\varepsilon|}{\ln s_{\varepsilon,i}}\bar{S}_K(z_i, z_i)+ O\left( \varepsilon^\gamma\right).
\end{split}
\end{equation*}
For any $ j\neq i $, since $ |\ln\varepsilon|^{-1}>2Ls_{\varepsilon,i} $, for $ \varepsilon $ sufficiently small  one has
\begin{equation*}
\begin{split}
V_{\varepsilon,Z,j}(x)=&V_{\varepsilon, z_j, \hat{q}_{\varepsilon,j}}(x)+H_{\varepsilon, z_j, \hat{q}_{\varepsilon,j}}(x)\\
=&\frac{\hat{q}_{\varepsilon,j}|\ln\varepsilon|}{\ln s_{\varepsilon,j}}\ln |T_{z_j}(x-z_j)|-\frac{2\pi\hat{q}_{\varepsilon,j}\sqrt{\det K(z_j)}|\ln\varepsilon|}{\ln s_{\varepsilon,j}}\bar{S}_K(x, z_j)+O\left(  s_{\varepsilon,j}^{\gamma}\right)\\
=&-\frac{2\pi\hat{q}_{\varepsilon,j}\sqrt{\det K(z_j)}|\ln\varepsilon|}{\ln s_{\varepsilon,j}}G_K(x, z_j)+O\left(  s_{\varepsilon,j}^{\gamma}\right) \\
=&-\frac{2\pi\hat{q}_{\varepsilon,j}\sqrt{\det K(z_j)}|\ln\varepsilon|}{\ln s_{\varepsilon,j}}G_K(z_i, z_j)+O\left( \varepsilon^\gamma\right),
\end{split}
\end{equation*}
where we have used  Lemma \ref{H estimate} and the fact that  for $ x\in B_{Ls_{\varepsilon,i}}(z_i) $
\begin{equation*}
G_K(x, z_j)=G_K(z_i, z_j)+O(|\nabla_{z_i}G_K(z_i, z_j)(x-z_j)|)=G_K(z_i, z_j)+O(\varepsilon|\ln\varepsilon|).
\end{equation*}
Adding up the above inequalities and using \eqref{q_i choice}, we get \eqref{203}.

Using the definition of $ V_{\varepsilon, z_i, \hat{q}_{\varepsilon,i}} $, we obtain

\begin{equation}\label{200}
\begin{split}
\frac{\partial V_{\varepsilon, z_i, \hat{q}_{\varepsilon,i}}(x)}{\partial x_\hbar}=\begin{cases}
\frac{1}{s_{\varepsilon,i}}(\frac{\varepsilon}{s_{\varepsilon,i}})^{\frac{2}{p-1}}\phi'(\frac{|T_{z_i}(x-z_i)|}{s_{\varepsilon,i}})\frac{(T_{z_i})_\hbar^T\cdot T_{z_i}(x-z_i)}{|T_{z_i}(x-z_i)|},\ \ &|T_{z_i}(x-z_i)|\leq s_{\varepsilon,i},\\
\frac{\hat{q}_{\varepsilon,i}|\ln\varepsilon|}{\ln s_{\varepsilon,i}}\frac{(T_{z_i})_\hbar^T\cdot T_{z_i}(x-z_i)}{|T_{z_i}(x-z_i)|^2},\ \ &|T_{z_i}(x-z_i)|> s_{\varepsilon,i},
\end{cases}
\end{split}
\end{equation}
for $ \hbar=1,2, $ where $ (T_{z_i})_\hbar^T $ is the $\hbar$-th row of $ (T_{z_i})^T $.

We end this section by giving 
 some estimates of approximate solutions $ V_{\varepsilon,Z} $, which will be frequently used in the subsequent sections.
\begin{lemma}\label{lemA-5}
	Let $\gamma\in(0,1)$ and $ Z\in \Lambda_{\varepsilon, N} $. Then there exists a constant $ L > 1 $ such that for $ \varepsilon $ small
	\begin{equation*}
	V_{\varepsilon,Z}-q|\ln\varepsilon|>0,\ \ \  \ \text{in}\ \   \cup_{j=1}^N\left( T_{z_j}^{-1}B_{\left( 1-L \varepsilon^\gamma \right) s_{\varepsilon,j}}(0)+z_j\right),
	\end{equation*}
	\begin{equation*}
	V_{\varepsilon,Z}-q|\ln\varepsilon|<0,\ \ \  \ \text{in}\ \  \Omega\backslash\cup_{j=1}^N\left( T_{z_j}^{-1}B_{Ls_{\varepsilon,j}}(0)+z_j\right).
	\end{equation*}
	
\end{lemma}

\begin{proof}
If $ |T_{z_j}(x-z_j)|\leq \left( 1-L \varepsilon^\gamma \right) s_{\varepsilon,j}   $, then by \eqref{203} and $ \phi'(1)<0 $ we have
	\begin{equation*}
	\begin{split}
	V_{\varepsilon,Z}(x)-q(x)|\ln\varepsilon|=&V_{\varepsilon,z_j,\hat{q}_{\varepsilon,j}}(x)-\hat{q}_{\varepsilon,j}|\ln\varepsilon|+O\left(  \varepsilon^\gamma\right) \\
	= &\frac{\hat{q}_{\varepsilon,j}|\ln\varepsilon|}{|\phi'(1)|\ln\frac{1}{s_{\varepsilon,j}}}\phi\left( \frac{|T_{z_j}(x-z_j)|}{s_{\varepsilon,j}}\right) +O\left( \varepsilon^\gamma\right) >0,
	\end{split}
	\end{equation*}
	if $ L $ is sufficiently large.

	On the other hand, if $ \tau>0 $ small, $ x\in\Omega$  
	and $ |T_{z_j}(x-z_j)|\geq s_{\varepsilon,j}^\tau $ for any $ j=1,\cdots,N $, then by the definition of $ V_{\varepsilon,z_j,\hat{q}_{\varepsilon,j}} $ and Lemma \ref{H estimate},
	\begin{equation*}
	\begin{split}
	V_{\varepsilon,Z}(x)-q(x)|\ln\varepsilon|=&\sum_{j=1}^N\left( V_{\varepsilon,z_j,\hat{q}_{\varepsilon,j}}(x)+H_{\varepsilon,z_j,\hat{q}_{\varepsilon,j}}(x)\right) -q(x)|\ln\varepsilon|\\
	\leq &\left( \sum_{j=1}^N\frac{\hat{q}_{\varepsilon,j}\ln s_{\varepsilon,j}^\tau}{\ln s_{\varepsilon,j}}-C\right)|\ln\varepsilon| \\
	\leq &\left( \tau \sum_{j=1}^N \hat{q}_{\varepsilon,j}-C\right) |\ln\varepsilon|<0.
	\end{split}
	\end{equation*}
	If $ Ls_{\varepsilon,j}\leq |T_{z_j}(x-z_j)|\leq s_{\varepsilon,j}^\tau $, then by \eqref{q_i choice} for $ L $ sufficiently large
	\begin{equation*}
	\begin{split}
	V_{\varepsilon,Z}&(x)-q(x)|\ln\varepsilon|\\
	=&V_{\varepsilon,z_j,\hat{q}_{\varepsilon,j}}(x)+H_{\varepsilon,z_j,\hat{q}_{\varepsilon,j}}(x)-q(x)|\ln\varepsilon|+\sum_{i\neq j}\left(V_{\varepsilon, z_i, \hat{q}_{\varepsilon,i}}(x)+H_{\varepsilon, z_i, \hat{q}_{\varepsilon,i}}(x)\right)\\
	=&V_{\varepsilon, z_j, \hat{q}_{\varepsilon,j}}(x)-q(z_j)|\ln\varepsilon|-\frac{2\pi\hat{q}_{\varepsilon,j}\sqrt{\det K(z_j)}|\ln\varepsilon|}{\ln s_{\varepsilon,j}}\bar{S}_K(x, z_j)\\
	&-\sum_{i\neq j}\frac{2\pi\hat{q}_{\varepsilon,i}\sqrt{\det K(z_i)}|\ln\varepsilon|}{\ln s_{\varepsilon,i}}G_K(x, z_i)+O(s_{\varepsilon,j}^\tau)
	\end{split}
\end{equation*}	
	\begin{equation*}
\begin{split}	
	=&V_{\varepsilon,z_j,\hat{q}_{\varepsilon,j}}(x)-q(z_j)|\ln\varepsilon|-\frac{2\pi\hat{q}_{\varepsilon,j}\sqrt{\det K(z_j)}|\ln\varepsilon|}{\ln s_{\varepsilon,j}}\bar{S}_K(z_j, z_j)\\
	&-\sum_{i\neq j}\frac{2\pi\hat{q}_{\varepsilon,i}\sqrt{\det K(z_i)}|\ln\varepsilon|}{\ln s_{\varepsilon,i}}G_K(z_j, z_i)
+O\left( \varepsilon^{\tau\gamma}\right) \\
	=&V_{\varepsilon,z_j,\hat{q}_{\varepsilon,j}}(x)-\hat{q}_{\varepsilon,j}|\ln\varepsilon|+O\left( \varepsilon^{\tau\gamma}\right) \\
	\leq &-\frac{\hat{q}_{\varepsilon,j}\ln L|\ln\varepsilon|}{\ln\frac{1}{s_{\varepsilon,j}}}+O\left( \varepsilon^{\tau\gamma}\right)<0.
	\end{split}
	\end{equation*}

\end{proof}
\section{The linear theory: solvability of  a linear problem}
In this section,  we prove the solvability of  a linear equation related to the linearized operator $ L_\varepsilon $ of \eqref{key equa} at the approximate solution $ \sum_{j=1}^NV_{\varepsilon, Z,j} $. For given $ \xi\in F_{\varepsilon, Z} $, following the road map used in \cite{CPY}, we obtain the existence and uniqueness of $ u\in E_{\varepsilon,Z} $ satisfying equations
\begin{equation*} 
 Q_\varepsilon L_\varepsilon u=\xi.
\end{equation*}
Definitions of $ E_{\varepsilon,Z}, F_{\varepsilon,Z} $ and $ Q_\varepsilon $ can be seen in \eqref{204}, \eqref{205} and \eqref{206} below.

We first note that the equation
\begin{equation}\label{eq5}
-\Delta w=w^p_+,\ \ \text{in}\ \mathbb{R}^2 
\end{equation}
has the unique $ C^1 $ solution 
\begin{equation*}
w(x)=\begin{cases}
\phi(x),\ \ &|x|\leq 1,\\
\phi'(1)\ln|x|,\ \ &|x|> 1.
\end{cases}
\end{equation*}
By the classical theory for elliptic equations, $ w\in C^{2,\alpha}(\mathbb{R}^2) $ for any $ \alpha\in(0,1) $. The linearized equation of \eqref{eq5} at $ w $ is
\begin{equation}\label{limit eq}
-\Delta v-pw^{p-1}_+v=0, \ \ v\in L^{\infty}(\mathbb{R}^2).
\end{equation}
Clearly, $ \frac{\partial w}{\partial x_\hbar} $ $(\hbar=1,2) $  are solutions of \eqref{limit eq}. Indeed, from \cite{DY} (see also \cite{CLW}), we have the non-degeneracy of $ w $ as follows.
\begin{proposition}[Non-degeneracy]\label{Non-degenerate}
$ w $ is non-degenerate, i.e., the kernel of the linearized equation \eqref{limit eq} is $$ span\left \{\frac{\partial w}{\partial x_1}, \frac{\partial w}{\partial x_2}\right \}. $$
\end{proposition}

Let $ \eta $ be a smooth  truncation function satisfying
\begin{equation*}
supp(\eta)\subseteq B_1(0),\ \ 0\leq \eta\leq 1\ \text{in}\ B_{1}(0),\   \  \eta\equiv1\ \text{in}\ B_{\frac{1}{2}}(0).
\end{equation*}
Define the scaled cut off function 
\begin{equation}\label{cut off fun}
\eta_i(x)=\eta\left ((x-z_i)|\ln\varepsilon|^2\right ).
\end{equation}
 Then $ supp(\eta_i)\subseteq B_{|\ln\varepsilon|^{-2}}(z_i) $ and $ supp(\eta_i)\cap supp(\eta_j)=\varnothing $ for $ i\neq j $ and $ \varepsilon $ sufficiently small. Moreover, $ ||\nabla\eta_i||_{L^\infty}\leq C|\ln\varepsilon|^{2} $ and $ ||\nabla^2\eta_i||_{L^\infty}\leq C|\ln\varepsilon|^{4} $.

Denote
\begin{equation}\label{204}
F_{\varepsilon,Z}=\left \{u\in L^p(\Omega)\mid \int_{\Omega}u\left( \eta_j\frac{\partial V_{\varepsilon, Z,j}}{\partial x_{\tilde{h}}}\right) =0,\ \ \forall j=1,\cdots,N,\ \tilde{h}=1,2\right \},
\end{equation}
and
\begin{equation}\label{205}
E_{\varepsilon,Z}=\left \{u\in W^{2,p}\cap H^1_0(\Omega)\mid \int_{\Omega}\left (K(x)\nabla u| \nabla \left( \eta_j\frac{\partial V_{\varepsilon, Z,j}}{\partial x_{\tilde{h}}}\right)\right )=0,\ \ \forall j=1,\cdots,N,\ \tilde{h}=1,2\right \}.
\end{equation}
Then $ F_{\varepsilon,Z}$ and $ E_{\varepsilon,Z} $ are co-dimensional $ 2N $ subspaces of $ L^p $ and $ W^{2,p}\cap H^1_0(\Omega) $, respectively.

For any $ u\in L^p(\Omega) $, we  define the projection operator $ Q_\varepsilon: L^p \to  F_{\varepsilon,Z} $
\begin{equation}\label{206}
Q_\varepsilon u:=u-\sum_{j=1}^N\sum_{\tilde{h}=1}^2b_{j,\tilde{h}}\left( -\varepsilon^2\text{div}\left (K(z_j)\nabla   \frac{\partial V_{\varepsilon, z_j, \hat{q}_{\varepsilon,j}}}{\partial x_{\tilde{h}}}\right )\right) ,
\end{equation}
where $ b_{j,\tilde{h}} (j=1,\cdots,N,\ \tilde{h}=1,2) $ satisfies a linear system as following
\begin{equation*}
\sum_{j=1}^N\sum_{\tilde{h}=1}^2b_{j,\tilde{h}}\int_{\Omega}\left( -\varepsilon^2\text{div}\left (K(z_j)\nabla   \frac{\partial V_{\varepsilon, z_j, \hat{q}_{\varepsilon,j}}}{\partial x_{\tilde{h}}}\right )\right)\left( \eta_i\frac{\partial V_{\varepsilon, Z,i}}{\partial x_\hbar}\right)=\int_{\Omega}u \left( \eta_i\frac{\partial V_{\varepsilon, Z,i}}{\partial x_\hbar}\right)
\end{equation*}
for $ i=1,\cdots,N,\ \hbar=1,2. $

We claim that $ Q_\varepsilon $ is a well-defined  linear projection  from $ L^p $ to $ F_{\varepsilon,Z}$. Indeed, using  \eqref{eq2}, \eqref{3-001} and \eqref{200},  for  $ Z\in \Lambda_{\varepsilon, N} $ the entry of coefficient matrix of the linear system satisfied by $b_{j,\tilde{h}}$ is
\begin{equation}\label{coef of C}
\begin{split}
\int_{\Omega}&\left( -\varepsilon^2\text{div}\left (K(z_j)\nabla   \frac{\partial V_{\varepsilon, z_j, \hat{q}_{\varepsilon,j}}}{\partial x_{\tilde{h}}}\right )\right)\left( \eta_i\frac{\partial V_{\varepsilon, Z,i}}{\partial x_\hbar}\right)\\
&=p\int_{\Omega}\eta_i\left (V_{\varepsilon, z_j, \hat{q}_{\varepsilon,j}}-\hat{q}_{\varepsilon,j}|\ln\varepsilon|\right )^{p-1}_+ \frac{\partial V_{\varepsilon, z_j, \hat{q}_{\varepsilon,j}}}{\partial x_{\tilde{h}}}\frac{\partial V_{\varepsilon, z_i, \hat{q}_{\varepsilon,i}}}{\partial x_\hbar}+O\left( \varepsilon^\gamma\right)\\
&=\delta_{i,j}(M_{i})_{\tilde{h},\hbar}+O\left( \varepsilon^\gamma\right) ,
\end{split}
\end{equation}
where $ \delta_{i,j}= 1 $ if $ i = j $ and $ \delta_{i,j}= 0 $ otherwise.  $ M_{i} $ are $ N $ positive definite matrices 
such that all eigenvalues of $ M_{i} $ belong to $ (\bar{c}_1,\bar{c}_2) $ for constants $ \bar{c}_1,\bar{c}_2>0 $.
This implies the existence and uniqueness of $ b_{j,\tilde{h}} $.
Note that for $ u\in L^p $, $ Q_{\varepsilon} u \equiv u $ in $ \Omega \backslash\cup_{i=1}^NB_{Ls_{\varepsilon,i}}(z_{i}) $ for some $ L>1 $. Moreover, one can easily get that there exists a constant $ C>0 $ independent of $ \varepsilon $, such that for any $ q\in[1,+\infty) $, $ u\in L^q(\Omega) $ with $ supp(u)\subset  \cup_{j=1}^NB_{Ls_{\varepsilon,j}}(z_j)$,
\begin{equation*}
||Q_\varepsilon u||_{L^q(\Omega)}\leq C||u||_{L^q(\Omega)}.
\end{equation*}

The linearized operator of \eqref{key equa} at $ V_{\varepsilon, Z} $ is
\begin{equation*}
L_\varepsilon \omega =-\varepsilon^2\text{div}(K(x)\nabla \omega)- p\left (V_{\varepsilon, Z}-q|\ln\varepsilon|\right )^{p-1}_+ \omega.
\end{equation*}
We give  coercive estimates of the linear operator $ Q_\varepsilon L_\varepsilon $.
\begin{proposition}\label{coercive esti}
There exist  $ \rho_0>0, \varepsilon_0>0 $ such that for any $ \varepsilon\in(0,\varepsilon_0],  Z\in \Lambda_{\varepsilon, N} $, if $ u\in E_{\varepsilon,Z} $ satisfying $ Q_\varepsilon L_\varepsilon u=0 $ in $ \Omega\backslash\cup_{j=1}^NB_{Ls_{\varepsilon,j}}(z_j) $ for some $ L>1 $ large, then
\begin{equation*}
||Q_\varepsilon L_\varepsilon u||_{L^p}\geq \rho_0\varepsilon^{\frac{2}{p}}||u||_{L^\infty}. 
\end{equation*}
\end{proposition}

\begin{proof}
We argue by contradiction. Suppose that there are $ \varepsilon_m \to 0 $, $ Z_m=(z_{m,1},\cdots,z_{m,N})\to (z_{1},\cdots,z_{N})\in \Omega^{(N)} $  and $  u_m \in E_{\varepsilon_m, Z_m} $ with $ Q_{\varepsilon_m} L_{\varepsilon_m} u_m = 0 $ in $ \Omega \backslash\cup_{j=1}^NB_{Ls_{\varepsilon_m,j}}(z_{m,j}) $ for some $ L $ large and $ ||u_m||_{L^\infty}=1 $ such that
\begin{equation*}
||Q_{\varepsilon_m} L_{\varepsilon_m} u_m||_{L^p}\leq \frac{1}{m}\varepsilon_m^{\frac{2}{p}}.
\end{equation*}
By definition of $ Q_{\varepsilon}  $,
\begin{equation}\label{208}
Q_{\varepsilon_m} L_{\varepsilon_m} u_m=  L_{\varepsilon_m} u_m-\sum_{j=1}^N\sum_{\tilde{h}=1}^2b_{j,\tilde{h},m}\left( -\varepsilon_m^2\text{div}\left (K(z_{m,j})\nabla   \frac{\partial V_{\varepsilon_m, z_{m,j},\hat{q}_{\varepsilon_m,j}}}{\partial x_{\tilde{h}}}\right )\right) .
\end{equation}
We now estimate $ b_{j,\tilde{h},m} $. For fixed $ i=1,\cdots,N, \hbar=1,2 $, multiplying \eqref{208} by $ \eta_i\frac{\partial V_{\varepsilon_m, Z_{m},i}}{\partial x_\hbar} $ and integrating on $ \Omega $ we get
\begin{equation*}
\begin{split}
\int_{\Omega}u_m&L_{\varepsilon_m}\left( \eta_i\frac{\partial V_{\varepsilon_m, Z_{m},i}}{\partial x_\hbar}\right) =\int_{\Omega}L_{\varepsilon_m}u_m\left( \eta_i\frac{\partial V_{\varepsilon_m, Z_{m},i}}{\partial x_\hbar}\right) \\
=&\sum_{j=1}^N\sum_{\tilde{h}=1}^2b_{j,\tilde{h},m}\int_{\Omega}-\varepsilon_m^2\text{div}\left (K(z_{m,j})\nabla   \frac{\partial V_{\varepsilon_m, z_{m,j},\hat{q}_{\varepsilon_m,j}}}{\partial x_{\tilde{h}}}\right )\left( \eta_i\frac{\partial V_{\varepsilon_m, Z_{m},i}}{\partial x_\hbar}\right) .
\end{split}
\end{equation*}

For the left hand side of the above equality, we have
\begin{equation*}
\begin{split}
&\int_{\Omega}u_mL_{\varepsilon_m}\left( \eta_i\frac{\partial V_{\varepsilon_m, Z_{m},i}}{\partial x_\hbar}\right) \\
=&-\int_{\Omega}u_m\varepsilon_m^2\text{div}\left (K(x)\nabla \left( \eta_i\frac{\partial V_{\varepsilon_m, Z_{m},i}}{\partial x_\hbar}\right) \right )- p\int_{\Omega}u_m\left (V_{\varepsilon_m, Z_m}-q|\ln\varepsilon_m|\right )^{p-1}_+\left( \eta_i\frac{\partial V_{\varepsilon_m, Z_{m},i}}{\partial x_\hbar}\right)  \\
=&-\int_{\Omega}\eta_iu_m\varepsilon_m^2\text{div}\left (K(x)\nabla \frac{\partial V_{\varepsilon_m, Z_{m},i}}{\partial x_\hbar} \right )-2\int_{\Omega}u_m\varepsilon_m^2\left (K(x)\nabla\eta_i|\nabla   \frac{\partial V_{\varepsilon_m, Z_{m},i}}{\partial x_\hbar}  \right )\\
&-\int_{\Omega}u_m\varepsilon_m^2\text{div}\left (K(x)\nabla   \eta_i  \right )\frac{\partial V_{\varepsilon_m, Z_{m},i}}{\partial x_\hbar}- p\int_{\Omega}u_m\left (V_{\varepsilon_m, Z_m}-q|\ln\varepsilon_m|\right )^{p-1}_+\left( \eta_i\frac{\partial V_{\varepsilon_m, Z_{m},i}}{\partial x_\hbar}\right)
\end{split}
\end{equation*}
\begin{equation}\label{301}
\begin{split}
=&\int_{\Omega}\eta_iu_mp\left( V_{\varepsilon_m, z_{m,i}, \hat{q}_{\varepsilon_m,i}}-\hat{q}_{\varepsilon_m,i}|\ln\varepsilon_m|\right)^{p-1}_+\frac{\partial V_{\varepsilon_m, z_{m,i}, \hat{q}_{\varepsilon_m,i}}}{\partial x_\hbar}+\int_{\Omega}\eta_iu_m\varepsilon_m^2\text{div}\left (\frac{\partial K(x)}{\partial x_\hbar}\nabla V_{\varepsilon_m, Z_{m},i}  \right )\\
&-2\int_{\Omega}u_m\varepsilon_m^2\left (K(x)\nabla\eta_i|\nabla   \frac{\partial V_{\varepsilon_m, Z_{m},i}}{\partial x_\hbar}  \right ) -\int_{\Omega}u_m\varepsilon_m^2\text{div}\left (K(x)\nabla   \eta_i  \right )\frac{\partial V_{\varepsilon_m, Z_{m},i}}{\partial x_\hbar}\\
&- p\int_{\Omega}u_m\left (V_{\varepsilon_m, Z_m}-q|\ln\varepsilon_m|\right )^{p-1}_+\left( \eta_i\frac{\partial V_{\varepsilon_m, Z_{m},i}}{\partial x_\hbar}\right).
\end{split}
\end{equation}

By \eqref{3-001}, \eqref{203} and Lemma \ref{lemA-5},  
\begin{equation}\label{302}
\begin{split}
\int_{\Omega}\eta_i&u_mp\left( V_{\varepsilon_m, z_{m,i}, \hat{q}_{\varepsilon_m,i}}-\hat{q}_{\varepsilon_m,i}|\ln\varepsilon_m|\right)^{p-1}_+\frac{\partial V_{\varepsilon_m, z_{m,i}, \hat{q}_{\varepsilon_m,i}}}{\partial x_\hbar}-  p\int_{\Omega}u_m\left (V_{\varepsilon_m, Z_m}-q|\ln\varepsilon_m|\right )^{p-1}_+\left( \eta_i\frac{\partial V_{\varepsilon_m, Z_{m},i}}{\partial x_\hbar}\right)\\
=&p\int_{\Omega}u_m\left( V_{\varepsilon_m, z_{m,i}, \hat{q}_{\varepsilon_m,i}}-\hat{q}_{\varepsilon_m,i}|\ln\varepsilon_m|\right)^{p-1}_+\frac{\partial V_{\varepsilon_m, z_{m,i}, \hat{q}_{\varepsilon_m,i}}}{\partial x_\hbar}\\
&-p\int_{\Omega}u_m\left( V_{\varepsilon_m, z_{m,i}, \hat{q}_{\varepsilon_m,i}}-\hat{q}_{\varepsilon_m,i}|\ln\varepsilon_m|+O\left( \varepsilon_m^\gamma\right) \right)^{p-1}_+\frac{\partial V_{\varepsilon_m, z_{m,i}, \hat{q}_{\varepsilon_m,i}}}{\partial x_\hbar}+O\left( \varepsilon_m^{1+\gamma}\right)\\
=&O\left( \varepsilon_m^{1+\gamma}\right).
\end{split}
\end{equation}
As for the remaining terms of \eqref{301}, 
using $ ||\nabla\eta_i||_{L^\infty}\leq C|\ln\varepsilon_m|^{2}$ and $ ||\nabla^2\eta_i||_{L^\infty}\leq C|\ln\varepsilon_m|^{4} $ we have
\begin{equation}\label{304}
\begin{array}{ll}
\int_{\Omega}\eta_iu_m\varepsilon_m^2\text{div}\left (\frac{\partial K(x)}{\partial x_\hbar}\nabla V_{\varepsilon_m, Z_{m},i}  \right )&\\
\qquad=\int\limits_{|T_{z_{m,i}}(x-z_{m,i})|\leq s_{\varepsilon_m,i}}\eta_iu_m \left( \varepsilon_m^2\text{div}\left( \frac{\partial K(x)}{\partial x_\hbar}\nabla V_{\varepsilon_m, z_{m,i}, \hat{q}_{\varepsilon_m,i}}\right)\right)&\\
\qquad\,\,\,\,\,+\int\limits_{ s_{\varepsilon_m,i}<|T_{z_{m,i}}(x-z_{m,i})|\leq |\ln\varepsilon_m|^{2}}
\eta_iu_m \left( \varepsilon_m^2\text{div}\left( \frac{\partial K(x)}{\partial x_\hbar}\nabla V_{\varepsilon_m, z_{m,i}, \hat{q}_{\varepsilon_m,i}}\right)\right)+O\left( \varepsilon_m^2\right)  &\\
\qquad=O\left(\varepsilon_m^2\right) +O\left (\varepsilon_m^2|\ln\varepsilon_m|\right )+O\left( \varepsilon_m^2\right)&\\
\qquad=O\left( \varepsilon_m^2|\ln\varepsilon_m|\right),
\end{array}
\end{equation}
\begin{equation}\label{305}
\begin{split}
-2\int_{\Omega}&u_m\varepsilon_m^2\left (K(x)\nabla\eta_i|\nabla   \frac{\partial V_{\varepsilon_m, Z_{m},i}}{\partial x_\hbar}  \right )\\
=&-2\int_{B_{|\ln\varepsilon_m|^{2}}(z_{m,i})\backslash B_{\frac{|\ln\varepsilon_m|^{2}}{2}}(z_{m,i})}u_m\varepsilon_m^2\left (K(x)\nabla\eta_i|\nabla   \frac{\partial V_{\varepsilon_m, z_{m,i},\hat{q}_{\varepsilon_m,i}}}{\partial x_\hbar}  \right )+O\left( \varepsilon_m^2|\ln\varepsilon_m|^{2}\right)\\
=&O\left( \varepsilon_m^2|\ln\varepsilon_m|^{2}\right),
\end{split}
\end{equation}
and
\begin{equation}\label{306}
\begin{split}
-\int_{\Omega}&u_m\varepsilon_m^2\text{div}\left (K(x)\nabla   \eta_i  \right )\frac{\partial V_{\varepsilon_m, Z_{m},i}}{\partial x_\hbar}\\
=&-\int_{B_{|\ln\varepsilon_m|^{2}}(z_{m,i})\backslash B_{\frac{|\ln\varepsilon_m|^{2}}{2}}(z_{m,i})}u_m\varepsilon_m^2\text{div}\left (K(x)\nabla   \eta_i  \right )\frac{\partial V_{\varepsilon_m, z_{m,i},\hat{q}_{\varepsilon_m,i}}}{\partial x_\hbar}+O\left( \varepsilon_m^2|\ln\varepsilon_m|^{4}\right)\\
=&O\left( \varepsilon_m^2|\ln\varepsilon_m|^{4}\right),
\end{split}
\end{equation}
where we have used \eqref{200}. Inserting \eqref{302}, \eqref{304}, \eqref{305} and \eqref{306} into \eqref{301}, we finally get
\begin{equation*}
\begin{split}
&\int_{\Omega}u_mL_{\varepsilon_m}\left( \eta_i\frac{\partial V_{\varepsilon_m, Z_{m},i}}{\partial x_\hbar}\right) =O\left( \varepsilon_m^{1+\gamma}\right),
\end{split}
\end{equation*}
which combined with \eqref{coef of C} deduces
\begin{equation*}
b_{j,\tilde{h},m}=O\left (\varepsilon_m^{1+\gamma}\right ).
\end{equation*}
Consequently
\begin{equation*}
\begin{split}
\sum_{j=1}^N\sum_{\tilde{h}=1}^2&b_{j,\tilde{h},m}\left( -\varepsilon_m^2\text{div}\left (K(z_{m,j})\nabla   \frac{\partial V_{\varepsilon_m, z_{m,j},\hat{q}_{\varepsilon_m,j}}}{\partial x_{\tilde{h}}}\right )\right)\\
&=O\left( \sum_{j=1}^N\sum_{\tilde{h}=1}^2\varepsilon_m^{\frac{2}{p}-1}|b_{j,\tilde{h},m}|\right)
=O\left( \varepsilon_m^{\frac{2}{p}+\gamma}\right),\ \ \ \ \text{in}\ L^p(\Omega).
\end{split}
\end{equation*}
Taking this into \eqref{208}, we have
\begin{equation}\label{307}
\begin{split}
L_{\varepsilon_m} u_m=&Q_{\varepsilon_m} L_{\varepsilon_m} u_m+\sum_{j=1}^N\sum_{\tilde{h}=1}^2b_{j,\tilde{h},m}\left( -\varepsilon_m^2\text{div}\left (K(z_{m,j})\nabla   \frac{\partial V_{\varepsilon_m, z_{m,j},\hat{q}_{\varepsilon_m,j}}}{\partial x_{\tilde{h}}}\right )\right)\\
=&O\left( \frac{1}{m}\varepsilon_m^{\frac{2}{p}}\right) +O\left(  \varepsilon_m^{\frac{2}{p}+\gamma}\right)
=o\left( \varepsilon_m^{\frac{2}{p}}\right),\ \ \ \ \text{in}\ L^p(\Omega).
\end{split}
\end{equation}

For fixed $ i, $ we define the scaled function $ \tilde{u}_{m,i}(y)=u_m(s_{\varepsilon_m, i}y+z_{m,i}) $ for $ y\in \Omega_{m,i}:=\{y\in\mathbb{R}^2\mid s_{\varepsilon_m, i}y+z_{m,i}\in \Omega\} $.
Define
\begin{equation*}
\begin{split}
\tilde{L}_{m,i}u=&-\text{div}\left (K(s_{\varepsilon_m, i}y+z_{m,i})\nabla u\right )\\
&-p\frac{s_{\varepsilon_m,i}^2}{\varepsilon_m^2}\left (V_{\varepsilon_m, Z_m}(s_{\varepsilon_m, i}y+z_{m,i})-q(s_{\varepsilon_m, i}y+z_{m,i})|\ln\varepsilon_m|\right )^{p-1}_+ u.
\end{split}
\end{equation*}
Then
\begin{equation*}
\begin{split}
||\tilde{L}_{m,i}\tilde{u}_{m,i}||_{L^p(\Omega_{m,i})}
=\frac{s_{\varepsilon_m, i}^2}{s_{\varepsilon_m, i}^{\frac{2}{p}}\varepsilon_m^2}||L_{\varepsilon_m}u_m||_{L^p(\Omega)}.
\end{split}
\end{equation*}
Since  $ s_{\varepsilon_m, i}=O(\varepsilon_m) $, by \eqref{307} we have
\begin{equation*}
\tilde{L}_{m,i}\tilde{u}_{m,i}=o(1)\ \ \ \  \text{in}\ \ L^p(\Omega_{m,i}).
\end{equation*}
Since $ ||\tilde{u}_{m,i}||_{L^\infty(\Omega_{m,i})}=1 $, by the classical regularity theory of  elliptic equations, $ \tilde{u}_{m,i} $ is uniformly bounded in $ W^{2,p}_{loc}(\mathbb{R}^2) $, which implies that there is $u_i$ such that
\begin{equation*}
\tilde{u}_{m,i}\to u_i\ \ \ \ \text{in}\ \ C^1_{loc}(\mathbb{R}^2).
\end{equation*}

We claim that $ u_i\equiv 0. $ On the one hand, in our ansatz for $ Z\in \Lambda_{\varepsilon, N} $, $ |z_i-z_j|\geq |\ln\varepsilon|^{-1} $. So by \eqref{203}, $ z_{m,i}\to z_i $ as $ m\to\infty $ and the fact that $ \lim_{\varepsilon\to 0}\varepsilon|\ln\varepsilon|=0 $, we get
\begin{equation*}
\begin{split}
\frac{s_{\varepsilon_m,i}^2}{\varepsilon_m^2}&\left (V_{\varepsilon_m, Z_m}(s_{\varepsilon_m, i}y+z_{m,i})-q\left (s_{\varepsilon_m, i}y+z_{m,i}\right )|\ln\varepsilon_m|\right )^{p-1}_+\\
=&\frac{s_{\varepsilon_m,i}^2}{\varepsilon_m^2}\left( V_{\varepsilon_m, z_{m,i}, \hat{q}_{\varepsilon_m,i}}(s_{\varepsilon_m, i}y+z_{m,i})-\hat{q}_{\varepsilon_m,i}|\ln\varepsilon_m|+O\left (\varepsilon_m^{\gamma}\right )\right)^{p-1}_+ \\
\to&\phi(T_{z_i}y)^{p-1}_+\ \  \  \text{in}\ C^0_{loc}(\mathbb{R}^2)\ \ \ \ \ \ \  \text{as}\ m\to\infty.
\end{split}
\end{equation*}
So $ u_i $ satisfies
\begin{equation*}
-\text{div}(K(z_i)\nabla u_i(x))-p\phi(T_{z_i}x)^{p-1}_+u_i(x)=0,\ \ x\in\mathbb{R}^2.
\end{equation*}
Let $ \hat{u}_i(x)=u_i(T_{z_i}^{-1}x) $. Since $ T_{z_i}^{-1}(T_{z_i}^{-1})^T=K(z_i) $, we have
\begin{equation*}
-\Delta \hat{u}_i(x)=-\text{div}(K(z_i)\nabla u_i)(T_{z_i}^{-1}x)=p\phi(x)^{p-1}_+\hat{u}_i(x),\ \ \ \   \ x\in \mathbb{R}^2.
\end{equation*}
By Proposition \ref{Non-degenerate}, there exist $ c_1,c_2 $ such that
\begin{equation}\label{210}
\hat{u}_i=c_1\frac{\partial \phi}{\partial x_1}+c_2\frac{\partial \phi}{\partial x_2}.
\end{equation}

On the other hand, since $ u_m\in E_{\varepsilon_m,Z_m} $, we get
\begin{equation*}
\int_{\Omega}-\varepsilon_m^2\text{div}\left( K(x)\nabla \left( \eta_i\frac{\partial V_{\varepsilon_m, Z_{m}, i}}{\partial x_\hbar}\right)\right)  u_m=0,\ \ \forall\ i=1,\cdots,N,\  \hbar=1,2,
\end{equation*}
which implies that
\begin{equation}\label{209}
\begin{split}
0
=&p\int_{\Omega}u_m\left( V_{\varepsilon_m, z_{m,i}, \hat{q}_{\varepsilon_m,i}}-\hat{q}_{\varepsilon_m,i}|\ln\varepsilon_m|\right)^{p-1}_+\frac{\partial V_{\varepsilon_m, z_{m,i}, \hat{q}_{\varepsilon_m,i}}}{\partial x_\hbar}+\int_{\Omega}\eta_iu_m\varepsilon_m^2\text{div}\left (\frac{\partial K(x)}{\partial x_\hbar}\nabla V_{\varepsilon_m, Z_{m},i}  \right )\\
&-2\int_{\Omega}u_m\varepsilon_m^2\left (K(x)\nabla\eta_i|\nabla   \frac{\partial V_{\varepsilon_m, Z_{m},i}}{\partial x_\hbar}  \right ) -\int_{\Omega}u_m\varepsilon_m^2\text{div}\left (K(x)\nabla   \eta_i  \right )\frac{\partial V_{\varepsilon_m, Z_{m},i}}{\partial x_\hbar}.
\end{split}
\end{equation}
Using again \eqref{304}, \eqref{305} and \eqref{306}, we have
\begin{equation}\label{209-1}
\begin{split}
\int_{\Omega}&\eta_iu_m\varepsilon_m^2\text{div}\left (\frac{\partial K(x)}{\partial x_\hbar}\nabla V_{\varepsilon_m, Z_{m},i}  \right )-2\int_{\Omega}u_m\varepsilon_m^2\left (K(x)\nabla\eta_i|\nabla   \frac{\partial V_{\varepsilon_m, Z_{m},i}}{\partial x_\hbar}  \right ) \\
&-\int_{\Omega}u_m\varepsilon_m^2\text{div}\left (K(x)\nabla   \eta_i  \right )\frac{\partial V_{\varepsilon_m, Z_{m},i}}{\partial x_\hbar}
=O\left( \varepsilon_m^2|\ln\varepsilon_m|^{4}\right).
\end{split}
\end{equation}
From  \eqref{200},
\begin{equation}\label{209-3}
\begin{split}
&p\int_{\Omega} u_m\left (V_{\varepsilon_m, z_{m,i}, \hat{q}_{\varepsilon_m,i}}-\hat{q}_{\varepsilon_m,i}|\ln\varepsilon_m|\right )^{p-1}_+ \frac{\partial V_{\varepsilon_m, z_{m,i}, \hat{q}_{\varepsilon_m,i}}}{\partial x_\hbar}\\
=&p\int_{\Omega}\frac{1}{s_{\varepsilon_m,i}}\left( \frac{\varepsilon_m}{s_{\varepsilon_m,i}}\right)^{\frac{2p}{p-1}} \phi\left( \frac{T_{z_{m,i}}(x-z_{m,i})}{s_{\varepsilon_m,i}}\right)^{p-1}_+\phi'\left( \frac{T_{z_{m,i}}(x-z_{m,i})}{s_{\varepsilon_m,i}}\right)\frac{(T_{z_{m,i}})_\hbar^T\cdot T_{z_{m,i}}(x-z_{m,i})}{|T_{z_{m,i}}(x-z_{m,i})|}u_m\\
=&ps_{\varepsilon_m,i}\left( \frac{\varepsilon_m}{s_{\varepsilon_m,i}}\right)^{\frac{2p}{p-1}}\int_{\mathbb{R}^2}\phi(T_{z_{m,i}}y)^{p-1}_+\phi'(T_{z_{m,i}}y)\frac{(T_{z_{m,i}})_\hbar^T\cdot T_{z_{m,i}}y}{|T_{z_{m,i}}y|}\tilde{u}_{m,i}(y)dy.
\end{split}
\end{equation}
Taking \eqref{209-1}  and \eqref{209-3} into \eqref{209}
, dividing both sides of \eqref{209} into $ ps_{\varepsilon_m,i}(\frac{\varepsilon_m}{s_{\varepsilon_m,i}})^{\frac{2p}{p-1}} $ and passing $ m $ to the limit, we get
\begin{equation*}
\begin{split}
0
=\int_{\mathbb{R}^2}\phi(x)^{p-1}_+\phi'(x)\frac{(T_{z_{i}})_\hbar^T\cdot x}{|x|}\hat{u}_{i}(x)\sqrt{\det(K(z_i))}dx,\ \  \hbar=1,2,
\end{split}
\end{equation*}
which implies that
\begin{equation}\label{211}
0=\int_{B_1(0)}\phi^{p-1}_+\frac{\partial \phi}{\partial x_\hbar}\hat{u}_i.
\end{equation}
So $ c_1=c_2=0. $ That is, $ u_i\equiv 0. $ We conclude that $ \tilde{u}_{N,i}\to 0 $ in  $ C^1(B_L(0)) $, which implies that
\begin{equation}\label{212}
||u_m||_{L^\infty(B_{Ls_{\varepsilon_m,i}}(z_{m,i}))}=o(1).
\end{equation}

Since $ Q_{\varepsilon_m} L_{\varepsilon_m} u_m = 0 $ in $ \Omega \backslash\cup_{i=1}^NB_{Ls_{\varepsilon_m,i}}(z_{m,i}) $, by the definition of $  Q_{\varepsilon_m} $ we have for $ L $ large
\begin{equation*}
 L_{\varepsilon_m} u_m = 0\  \ \text{in} \ \Omega \backslash\cup_{i=1}^NB_{Ls_{\varepsilon_m,i}}(z_{m,i}).
\end{equation*}
From Lemma \ref{lemA-5},  $ (V_{\varepsilon_m,Z_m}-q|\ln\varepsilon_m|)_+=0\  \ \text{in} \ \Omega \backslash\cup_{i=1}^NB_{Ls_{\varepsilon_m,i}}(z_{m,i}). $
So $ -\text{div}(K(x)\nabla u_m)=0 $ in $ \Omega \backslash\cup_{i=1}^NB_{Ls_{\varepsilon_m,i}}(z_{m,i}). $
Thus by the maximum principle, 
\begin{equation*}
||u_m||_{L^\infty(\Omega\backslash \cup_{i=1}^NB_{Ls_{\varepsilon_m,i}}(z_{m,i}))}=o(1),
\end{equation*}
which combined with \eqref{212} leads to
\begin{equation*}
||u_m||_{L^\infty(\Omega)}=o(1).
\end{equation*}
This is a contradiction since $ ||u_m||_{L^\infty(\Omega)}=1. $

\end{proof}

As a  consequence of Proposition \ref{coercive esti}, we have that $ Q_\varepsilon L_\varepsilon $ is indeed a one to one and onto map from $ E_{\varepsilon,Z} $ to $ F_{\varepsilon, Z}. $
\begin{proposition}\label{one to one and onto}
$ Q_\varepsilon L_\varepsilon $ is  a one to one and onto map from $ E_{\varepsilon,Z} $ to $ F_{\varepsilon, Z}. $
\end{proposition}

\begin{proof}
If $ Q_\varepsilon L_\varepsilon u\equiv0 $, by Lemma \ref{coercive esti}, $ u\equiv0 $. So $ Q_\varepsilon L_\varepsilon $ is  one to one.

For any $ \hat{h}\in F_{\varepsilon,Z}  $, by the Riesz representation theorem there is a unique $ u\in H^1_0(\Omega) $ such that
\begin{equation}\label{213}
\varepsilon^2\int_{\Omega}\left( K(x)\nabla u|\nabla\varphi\right) =\int_{\Omega}\hat{h}\varphi,\ \ \   \ \forall \varphi\in H^1_0(\Omega).
\end{equation}

Since $ \hat{h}\in F_{\varepsilon,Z} $, using the classical $ L^p $ theory of elliptic equations, we conclude that $ u\in W^{2,p}(\Omega) $, which implies that $ u\in E_{\varepsilon,Z}. $ Thus $ -\varepsilon^2\text{div}(K(x)\nabla)=Q_{\varepsilon}(-\varepsilon^2\text{div}(K(x)\nabla)) $ is a one to one and onto map from $ E_{\varepsilon,Z} $ to $ F_{\varepsilon, Z}. $

For any $ \xi\in F_{\varepsilon,Z} $,  $ Q_\varepsilon L_\varepsilon u=\xi $ is equivalent to
\begin{equation}\label{214}
\begin{split}
u=&\left (Q_{\varepsilon}\left (-\varepsilon^2\text{div}(K(x)\nabla)\right )\right )^{-1}pQ_\varepsilon\left (V_{\varepsilon,Z}-q|\ln\varepsilon|\right )^{p-1}_+u\\
&+\left (Q_{\varepsilon}\left (-\varepsilon^2\text{div}(K(x)\nabla)\right )\right )^{-1}\xi.
\end{split}
\end{equation}
Note that $  (Q_{\varepsilon}(-\varepsilon^2\text{div}(K(x)\nabla)))^{-1}pQ_\varepsilon\left (V_{\varepsilon,Z}-q|\ln\varepsilon|\right )^{p-1}_+u $ is a compact operator in $ E_{\varepsilon,Z}, $ by the Fredholm alternative, \eqref{214} is solvable if and only if
\begin{equation*}
u=\left (Q_{\varepsilon}\left (-\varepsilon^2\text{div}(K(x)\nabla)\right )\right )^{-1}pQ_\varepsilon\left (V_{\varepsilon,Z}-q|\ln\varepsilon|\right )^{p-1}_+u
\end{equation*}
has only trivial solution, which is true since $ Q_\varepsilon L_\varepsilon  $ is one to one. 
\end{proof}

\section{The nonlinear theory: solvability of a nonlinear equation}


In this section, we look for  solutions $ \omega\in E_{\varepsilon,Z}  $ of the following nonlinear equation
\begin{equation}\label{216}
Q_\varepsilon L_\varepsilon \omega=Q_\varepsilon l_\varepsilon+Q_\varepsilon R_\varepsilon(\omega),
\end{equation}
or equivalently,
\begin{equation*}
\begin{split}
\omega=T_\varepsilon(\omega):=(Q_\varepsilon L_\varepsilon)^{-1}Q_\varepsilon l_\varepsilon+(Q_\varepsilon L_\varepsilon)^{-1}Q_\varepsilon R_\varepsilon(\omega).
\end{split}
\end{equation*}
We have
\begin{proposition}\label{exist and uniq of w}
There exists $ \varepsilon_0>0, $ such that for any $ \gamma\in (0,1) $, $ 0<\varepsilon<\varepsilon_0 $ and $ Z\in \Lambda_{\varepsilon,N} $, \eqref{216} has the unique solution $ \omega_{\varepsilon,Z}\in E_{\varepsilon,Z} $ with
\begin{equation*}
||\omega_{\varepsilon,Z}||_{L^\infty(\Omega)}=O\left( \varepsilon^\gamma \right).
\end{equation*}
\end{proposition}

\begin{proof}
It follows from  Lemma \ref{lemA-5} that for $ L $ sufficiently large and $ \varepsilon $ small,
\begin{equation*}
\left (V_{\varepsilon,Z}-q|\ln\varepsilon|\right )_+=0,\ \ \ \ \text{in}\ \Omega \backslash\cup_{i=1}^NB_{Ls_{\varepsilon,i}}(z_{i}).
\end{equation*}
Let $ \mathcal{N}= E_{\varepsilon,Z} \cap\{\omega\mid ||\omega||_{L^\infty(\Omega)}\leq \frac{1}{|\ln\varepsilon|^{1-\theta_0}}\}$ for some $ \theta_0\in(0,1). $ Then $ \mathcal{N} $ is complete under $ L^\infty $ norm and $ T_\varepsilon $ is a map from $ E_{\varepsilon,Z} $ to $ E_{\varepsilon,Z} $. We prove that $ T_\varepsilon $ is a contraction map from $ \mathcal{N} $ to $ \mathcal{N} $.

First, we claim that $ T_\varepsilon $ is a map from $ \mathcal{N} $ to $ \mathcal{N} $. For any $ \omega\in \mathcal{N} $, by Lemma \ref{lemA-5} we get that for $ L>1 $ large and  $ \varepsilon $ small,
\begin{equation*}
\left (V_{\varepsilon,Z}+\omega-q|\ln\varepsilon|\right )_+=0,\ \ \ \ \text{in}\ \Omega \backslash\cup_{i=1}^NB_{Ls_{\varepsilon,i}}(z_{i}).
\end{equation*}
So $ l_\varepsilon=R_\varepsilon(\omega)=0 $ in $ \Omega \backslash\cup_{i=1}^NB_{Ls_{\varepsilon,i}}(z_{i}). $ By the definition of $ Q_\varepsilon $,
\begin{equation*}
Q_\varepsilon l_\varepsilon+Q_\varepsilon R_\varepsilon(\omega)=0,\ \ \ \ \text{in}\ \Omega \backslash\cup_{i=1}^NB_{Ls_{\varepsilon,i}}(z_{i}).
\end{equation*}
Using Proposition \ref{coercive esti}, we obtain
\begin{equation*}
||T_\varepsilon(\omega)||_{L^\infty}=||(Q_\varepsilon L_\varepsilon)^{-1}(Q_\varepsilon l_\varepsilon+Q_\varepsilon R_\varepsilon(\omega))||_{L^\infty}\leq C\varepsilon^{-\frac{2}{p}}||Q_\varepsilon l_\varepsilon+Q_\varepsilon R_\varepsilon(\omega)||_{L^p}.
\end{equation*}
Note that $ ||Q_\varepsilon l_\varepsilon+Q_\varepsilon R_\varepsilon(\omega)||_{L^p}\leq C(|| l_\varepsilon||_{L^p}+ ||R_\varepsilon(\omega)||_{L^p}). $
It follows from \eqref{203}, the definition of $ l_\varepsilon, R_\varepsilon(\omega) $ and Lemma \ref{lemA-5} that
\begin{equation*}
\begin{split}
||l_\varepsilon||_{L^p}=&||\left (V_{\varepsilon,Z}-q|\ln\varepsilon|\right )^p_+-\sum_{j=1}^N\left (V_{\varepsilon,z_j,\hat{q}_{\varepsilon,j}}-\hat{q}_{\varepsilon,j}|\ln\varepsilon|\right )^p_+||_{L^p}\leq C\varepsilon^{\frac{2}{p}+\gamma},
\end{split}
\end{equation*}
and
\begin{equation*}
\begin{split}
||R_\varepsilon(\omega)||_{L^p}&=||\left (V_{\varepsilon,Z}+\omega-q|\ln\varepsilon|\right )^p_+-\left (V_{\varepsilon,Z}-q|\ln\varepsilon|\right )^p_+-p\left (V_{\varepsilon,Z}-q|\ln\varepsilon|\right )^{p-1}_+\omega||_{L^p}\\
&\leq C\varepsilon^{\frac{2}{p}}||\omega||_{L^\infty}^2.
\end{split}
\end{equation*}
Hence we get
\begin{equation}\label{217}
\begin{split}
||T_\varepsilon(\omega)||_{L^\infty}\leq C\varepsilon^{-\frac{2}{p}}\left( \varepsilon^{\frac{2}{p}+\gamma}+\varepsilon^{\frac{2}{p}}||\omega||_{L^\infty}^2\right)
\leq \frac{1}{|\ln\varepsilon|^{1-\theta_0}}.
\end{split}
\end{equation}
So $ T_\varepsilon $ is a map from $ \mathcal{N}$ to $ \mathcal{N} $.

Then we prove that  $ T_\varepsilon $ is a contraction map. For any $ \omega_1,\omega_2\in \mathcal{N} $,
\begin{equation*}
T_\varepsilon(\omega_1)-T_\varepsilon(\omega_2)=(Q_\varepsilon L_\varepsilon)^{-1}Q_\varepsilon(R_\varepsilon(\omega_1)-R_\varepsilon(\omega_2)).
\end{equation*}
Note that $ R_\varepsilon(\omega_1)=R_\varepsilon(\omega_2)=0 $ in $ \Omega \backslash\cup_{i=1}^NB_{Ls_{\varepsilon,i}}(z_{i}). $ By Proposition \ref{coercive esti} and the definition of $ \mathcal{N} $, for $ \varepsilon $ sufficiently small
\begin{equation*}
\begin{split}
||T_\varepsilon(\omega_1)-T_\varepsilon(\omega_2)||_{L^\infty}\leq &C\varepsilon^{-\frac{2}{p}}||R_\varepsilon(\omega_1)-R_\varepsilon(\omega_2)||_{L^p}\\
\leq &C\varepsilon^{-\frac{2}{p}}\varepsilon^{\frac{2}{p}}\left(||\omega_1||_{L^\infty}+||\omega_2||_{L^\infty} \right) ||\omega_1-\omega_2||_{L^\infty}\\
\leq& \frac{1}{2}||\omega_1-\omega_2||_{L^\infty}.
\end{split}
\end{equation*}
So $ T_\varepsilon $ is a contraction map.

To conclude,  
there is a unique $ \omega_{\varepsilon,Z}\in \mathcal{N} $ such that $ \omega_{\varepsilon,Z}=T_\varepsilon(\omega_{\varepsilon,Z}) $. Moreover, by \eqref{217}  we have $ ||\omega_{\varepsilon,Z}||_{L^\infty(\Omega)}=O\left( \varepsilon^\gamma\right).  $

\end{proof}


The result of Proposition \ref{exist and uniq of w} implies that there exists a unique solution $ \omega_{\varepsilon,Z}\in E_{\varepsilon,Z} $ to \eqref{216}. 
This implies that for some $ b_{j,\tilde{h}}=b_{j,\tilde{h}}(Z) $
\begin{equation}\label{5-000}
L_\varepsilon \omega_{\varepsilon,Z}= l_\varepsilon+  R_\varepsilon(\omega_{\varepsilon,Z})+\sum_{j=1}^N\sum_{\tilde{h}=1}^2b_{j,\tilde{h}}\left( -\varepsilon^2\text{div}\left (K(z_j)\nabla   \frac{\partial V_{\varepsilon, z_j, \hat{q}_{\varepsilon,j}}}{\partial x_{\tilde{h}}}\right )\right),
\end{equation}
or equivalently
\begin{equation}\label{5-00}
\begin{split}
-\varepsilon^2\text{div}&\left (K(x)\nabla (V_{\varepsilon,Z}+\omega_{\varepsilon,Z})\right )-\left (V_{\varepsilon,Z}+\omega_{\varepsilon,Z}-q|\ln\varepsilon|\right )^p_+\\
=&\sum_{j=1}^N\sum_{\tilde{h}=1}^2b_{j,\tilde{h}}\left( -\varepsilon^2\text{div}\left (K(z_j)\nabla   \frac{\partial V_{\varepsilon, z_j, \hat{q}_{\varepsilon,j}}}{\partial x_{\tilde{h}}}\right )\right).
\end{split}
\end{equation}

In the following result, we prove  the differentiability of $ \omega_{\varepsilon,Z}  $ with respect to the variable $ Z $, which will be used to show the $ C^1-$regularity of the finite-dimensional functional $ K_\varepsilon(Z) $. Similar to results in \cite{CPY,WYZ}, we give the $ L^\infty $ estimate of  $ \frac{\partial \omega_{\varepsilon,Z}}{\partial z_{i,\hbar}} $ and show that $ \omega_{\varepsilon,Z} $ is a $ C^1 $ map of $ Z $ in $ H^1_0 (\Omega) $.

\begin{proposition}\label{differ of w}
Let $ \omega_{\varepsilon,Z} $ be the function obtained in Proposition \ref{exist and uniq of w}. Then $ \omega_{\varepsilon,Z} $ is a $ C^1 $ map of $ Z $ in the
norm of $ H^1_0 (\Omega) $, and for any $ \gamma\in (0,1) $, $ l=1,\cdots,N $, $ \bar{h}=1,2 $
\begin{equation*}
\bigg|\bigg|\frac{\partial \omega_{\varepsilon,Z}}{\partial z_{l,\bar{h}}}\bigg|\bigg|_{L^\infty(\Omega)}=O\left (\frac{1}{\varepsilon^{1-\gamma}}\right ).
\end{equation*}

\end{proposition}

\begin{proof}
Note that from Lemma \ref{Green expansion2}, the regular part of Green's function $ \bar{S}_K(x,x)\in C^1(\Omega) $. Thus taking $ \frac{\partial}{\partial z_{l,\bar{h}}} $ in  \eqref{q_i choice}, we get
\begin{equation}\label{5-01}
\begin{split}
\frac{\partial \hat{q}_{\varepsilon,i}}{\partial z_{l,\bar{h}}}=&\frac{\partial q}{\partial x_{\bar{h}}}(z_i)\delta_{i,l}+\sum_{j\neq l}\frac{2\pi\hat{q}_{\varepsilon,j}\sqrt{\det K(z_j)}}{\ln s_{\varepsilon,j}}\frac{\partial G_K(z_l,z_j)}{\partial z_{l,\bar{h}}}+o(1)\sum_{j=1}^N\bigg|\frac{\partial \hat{q}_{\varepsilon,i}}{\partial z_{l,\bar{h}}}\bigg|+o(1)\\
=&O\left (|\ln\varepsilon|\right ),
\end{split}
\end{equation}
where we have used $ |\nabla_{z_l} G(z_l, z_j)|\leq\frac{ C}{|z_l-z_j|}\leq C|\ln\varepsilon| $ for $ Z\in \Lambda_{\varepsilon, N}. $ By the definition of $ V_{\varepsilon, z_j, \hat{q}_{\varepsilon,j}} $ and \eqref{5-01}, we get
\begin{equation}\label{5-02}
\bigg|\bigg|\frac{\partial V_{\varepsilon, z_j, \hat{q}_{\varepsilon,j}}}{\partial z_{l,\bar{h}}}\bigg|\bigg|_{L^\infty(\Omega)}=O\left (\frac{1}{\varepsilon}\right )+O\left (|\ln\varepsilon|^2\right )=O\left (\frac{1}{\varepsilon}\right ).
\end{equation}
Using the definition of $ H_{\varepsilon, z_j, \hat{q}_{\varepsilon,j}} $ in \eqref{eq3} and the $ L^p $-theory of elliptic equations, one has
\begin{equation}\label{5-03}
\bigg|\bigg|\frac{\partial H_{\varepsilon, z_j, \hat{q}_{\varepsilon,j}}}{\partial z_{l,\bar{h}}}\bigg|\bigg|_{W^{1,p}(\Omega)}\leq
\begin{cases}
\frac{C}{\varepsilon^{1-\frac{2}{p}}},\ \ &p>2,\\
C|\ln\varepsilon|,\ \ &p=2,\\
C,\ \ &1\leq p<2.
\end{cases}
\end{equation}
Combining \eqref{5-02} and \eqref{5-03}, we have
\begin{equation}\label{5-04}
\bigg|\bigg|\frac{\partial V_{\varepsilon, Z, j}}{\partial z_{l,\bar{h}}}\bigg|\bigg|_{L^\infty(\Omega)}=O\left (\frac{1}{\varepsilon}\right ).
\end{equation}

Now we calculate the $ L^\infty $ norm of $ \frac{\partial \omega_{\varepsilon,Z}}{\partial z_{l,\bar{h}}}. $ We first estimate $ b_{j,\tilde{h}} $ in \eqref{5-00}. Recall the definition of $ \eta_i $ in \eqref{cut off fun}. Using \eqref{5-00}
, $ b_{j,h} $ is determined by
\begin{equation}\label{5-05}
\begin{split}
\sum_{j=1}^N&\sum_{\tilde{h}=1}^2b_{j,\tilde{h}}\int_\Omega\left( -\varepsilon^2\text{div}\left (K(z_j)\nabla   \frac{\partial V_{\varepsilon, z_j, \hat{q}_{\varepsilon,j}}}{\partial x_{\tilde{h}}}\right )\right)\left( \eta_i\frac{\partial V_{\varepsilon, Z,i}}{\partial x_\hbar}\right)\\
=&\int_\Omega\varepsilon^2\left (K(x)\nabla(V_{\varepsilon,Z}+\omega_{\varepsilon,Z})|\nabla\left( \eta_i\frac{\partial V_{\varepsilon, Z,i}}{\partial x_\hbar}\right)\right )-\int_\Omega\left (V_{\varepsilon,Z}+\omega_{\varepsilon,Z}-q|\ln\varepsilon|\right )^p_+\left( \eta_i\frac{\partial V_{\varepsilon, Z,i}}{\partial x_\hbar}\right)\\
=&\int_\Omega\varepsilon^2\left (K(x)\nabla V_{\varepsilon,Z}|\nabla\left( \eta_i\frac{\partial V_{\varepsilon, Z,i}}{\partial x_\hbar}\right)\right )-\int_\Omega\left (V_{\varepsilon,Z}+\omega_{\varepsilon,Z}-q|\ln\varepsilon|\right )^p_+\left( \eta_i\frac{\partial V_{\varepsilon, Z,i}}{\partial x_\hbar}\right)\\
=&\int_\Omega\left( \sum_{j=1}^N\left (V_{\varepsilon, z_j, \hat{q}_{\varepsilon,j}}-\hat{q}_{\varepsilon,j}|\ln\varepsilon|\right )^{p}_+-\left (V_{\varepsilon,Z}+\omega_{\varepsilon,Z}-q|\ln\varepsilon|\right )^{p}_+\right) \left( \eta_i\frac{\partial V_{\varepsilon, Z,i}}{\partial x_\hbar}\right),
\end{split}
\end{equation}
where we have used $ \omega_{\varepsilon,Z}\in E_{\varepsilon,Z} $. By Lemma \ref{lemA-5},
\begin{equation*}
\begin{split}
\int_\Omega&\left( \sum_{j=1}^N\left (V_{\varepsilon, z_j, \hat{q}_{\varepsilon,j}}-\hat{q}_{\varepsilon,j}|\ln\varepsilon|\right )^{p}_+-\left (V_{\varepsilon,Z}+\omega_{\varepsilon,Z}-q|\ln\varepsilon|\right )^{p}_+\right) \left( \eta_i\frac{\partial V_{\varepsilon, Z,i}}{\partial x_\hbar}\right)\\
&=O\left (\left (\varepsilon^\gamma+||\omega_{\varepsilon,Z}||_{L^\infty}\right )\int_{\cup_{j=1}^NB_{L\varepsilon}(z_j)}\bigg|\frac{\partial V_{\varepsilon, Z,i}}{\partial x_\hbar} \bigg|\right )=O\left (\varepsilon^{1+\gamma}\right ).
\end{split}
\end{equation*}
Thus combining this with \eqref{coef of C} and \eqref{5-05}, we get
\begin{equation}\label{5-06}
b_{j,\tilde{h}}=O\left ( \varepsilon^{1+\gamma} \right ).
\end{equation}

We now estimate  $  \frac{\partial b_{j,\tilde{h}}}{\partial z_{l,\bar{h}}} $. On the one hand, taking $ \frac{\partial}{\partial z_{l,\bar{h}}} $ in both sides of \eqref{5-05}, we obtain
\begin{equation}\label{5-07}
\begin{split}
\sum_{j=1}^N&\sum_{\tilde{h}=1}^2\frac{\partial b_{j,\tilde{h}}}{\partial z_{l,\bar{h}}}\int_\Omega\left( -\varepsilon^2\text{div}\left (K(z_j)\nabla   \frac{\partial V_{\varepsilon, z_j, \hat{q}_{\varepsilon,j}}}{\partial x_{\tilde{h}}}\right )\right)\left( \eta_i\frac{\partial V_{\varepsilon, Z,i}}{\partial x_\hbar}\right)\\
=&-\sum_{j=1}^N\sum_{\tilde{h}=1}^2b_{j,\tilde{h}} \frac{\partial }{\partial z_{l,\bar{h}}}\left\{  \int_\Omega\left( -\varepsilon^2\text{div}\left (K(z_j)\nabla   \frac{\partial V_{\varepsilon, z_j, \hat{q}_{\varepsilon,j}}}{\partial x_{\tilde{h}}}\right )\right)\left( \eta_i\frac{\partial V_{\varepsilon, Z,i}}{\partial x_\hbar}\right)\right\} \\
&+\frac{\partial }{\partial z_{l,\bar{h}}}\left\{\int_\Omega\left( \sum_{j=1}^N\left (V_{\varepsilon, z_j, \hat{q}_{\varepsilon,j}}-\hat{q}_{\varepsilon,j}|\ln\varepsilon|\right )^{p}_+-\left (V_{\varepsilon,Z}+\omega_{\varepsilon,Z}-q|\ln\varepsilon|\right )^{p}_+\right) \left( \eta_i\frac{\partial V_{\varepsilon, Z,i}}{\partial x_\hbar}\right)\right\}.
\end{split}
\end{equation}
Note that from \eqref{5-02}, \eqref{5-03} and \eqref{5-04},
\begin{equation*}
 \frac{\partial }{\partial z_{l,\bar{h}}}\left\{  \int_\Omega\left( -\varepsilon^2\text{div}\left (K(z_j)\nabla   \frac{\partial V_{\varepsilon, z_j, \hat{q}_{\varepsilon,j}}}{\partial x_{\tilde{h}}}\right )\right)\left( \eta_i\frac{\partial V_{\varepsilon, Z,i}}{\partial x_\hbar}\right)\right\}=O\left( \frac{1}{\varepsilon}\right)
\end{equation*}
and
\begin{equation*}
\begin{split}
 \int_\Omega&\left( \sum_{j=1}^N\left (V_{\varepsilon, z_j, \hat{q}_{\varepsilon,j}}-\hat{q}_{\varepsilon,j}|\ln\varepsilon|\right )^{p}_+-\left (V_{\varepsilon,Z}+\omega_{\varepsilon,Z}-q|\ln\varepsilon|\right )^{p}_+\right)  \frac{\partial }{\partial z_{l,\bar{h}}}\left( \eta_i\frac{\partial V_{\varepsilon, Z,i}}{\partial x_\hbar}\right) \\
 &=O\left (\left ( \varepsilon^\gamma+||\omega_{\varepsilon,Z}||_{L^\infty}\right )\int_{\cup_{j=1}^NB_{L\varepsilon}(z_j)}\bigg|\frac{\partial^2 V_{\varepsilon, Z,i}}{\partial z_{l,\bar{h}}\partial x_\hbar} \bigg|+|\ln\varepsilon|^2\bigg|\frac{\partial V_{\varepsilon, Z,i}}{\partial x_\hbar} \bigg|\right )=O\left (\varepsilon^{\gamma}\right ).
\end{split}
\end{equation*}
Taking these into \eqref{5-07}, we obtain
\begin{equation}\label{5-08}
\begin{split}
 \frac{\partial b_{j,\tilde{h}}}{\partial z_{l,\bar{h}}}=&O\left (\frac{ |b_{j,\tilde{h}}|}{\varepsilon}+\varepsilon^{\gamma}\right )\\
 &+O\left(  \int_\Omega\frac{\partial }{\partial z_{l,\bar{h}}}\left( \sum_{j=1}^N\left (V_{\varepsilon, z_j, \hat{q}_{\varepsilon,j}}-\hat{q}_{\varepsilon,j}|\ln\varepsilon|\right )^{p}_+-\left (V_{\varepsilon,Z}+\omega_{\varepsilon,Z}-q|\ln\varepsilon|\right )^{p}_+\right) \left( \eta_i\frac{\partial V_{\varepsilon, Z,i}}{\partial x_\hbar}\right) \right) .
\end{split}
\end{equation}
Using \eqref{203}, \eqref{5-01}, \eqref{5-02} and \eqref{5-03}, we get
\begin{equation*}\label{5-09}
\begin{split}
\frac{\partial }{\partial z_{l,\bar{h}}}&\left( \sum_{j=1}^N\left (V_{\varepsilon, z_j, \hat{q}_{\varepsilon,j}}-\hat{q}_{\varepsilon,j}|\ln\varepsilon|\right )^{p}_+-\left (V_{\varepsilon,Z}+\omega_{\varepsilon,Z}-q|\ln\varepsilon|\right )^{p}_+\right)\\
=& p\sum_{j=1}^N\left (V_{\varepsilon, z_j, \hat{q}_{\varepsilon,j}}-\hat{q}_{\varepsilon,j}|\ln\varepsilon|\right )^{p-1}_+\frac{\partial }{\partial z_{l,\bar{h}}}\left (V_{\varepsilon, z_j, \hat{q}_{\varepsilon,j}}-\hat{q}_{\varepsilon,j}|\ln\varepsilon|\right )\\
&-p\left (V_{\varepsilon,Z}+\omega_{\varepsilon,Z}-q|\ln\varepsilon|\right )^{p-1}_+\frac{\partial }{\partial z_{l,\bar{h}}}\left(V_{\varepsilon,Z}+\omega_{\varepsilon,Z} \right)  \\
=&-p\left (V_{\varepsilon,Z}+\omega_{\varepsilon,Z}-q|\ln\varepsilon|\right )^{p-1}_+\frac{\partial \omega_{\varepsilon,Z} }{\partial z_{l,\bar{h}}}\\
&+p\sum_{j=1}^N\left( \left (V_{\varepsilon, z_j, \hat{q}_{\varepsilon,j}}-\hat{q}_{\varepsilon,j}|\ln\varepsilon|\right )^{p-1}_+-\left (V_{\varepsilon,Z}+\omega_{\varepsilon,Z}-q|\ln\varepsilon|\right )^{p-1}_+\right)\frac{\partial V_{\varepsilon, z_j, \hat{q}_{\varepsilon,j}}}{\partial z_{l,\bar{h}}} \\
&-p\sum_{j=1}^N\left (V_{\varepsilon,Z}+\omega_{\varepsilon,Z}-q|\ln\varepsilon|\right )^{p-1}_+\frac{\partial H_{\varepsilon, z_j, \hat{q}_{\varepsilon,j}}}{\partial z_{l,\bar{h}}}-p|\ln\varepsilon|\sum_{j=1}^N\left (V_{\varepsilon, z_j, \hat{q}_{\varepsilon,j}}-\hat{q}_{\varepsilon,j}|\ln\varepsilon|\right )^{p-1}_+\frac{\partial \hat{q}_{\varepsilon,j} }{\partial z_{l,\bar{h}}} \\
=&-p\left (V_{\varepsilon,Z}+\omega_{\varepsilon,Z}-q|\ln\varepsilon|\right )^{p-1}_+\frac{\partial \omega_{\varepsilon,Z} }{\partial z_{l,\bar{h}}}+O\left( \frac{1}{\varepsilon^{1-\gamma}}\right),
\end{split}
\end{equation*}
from which we deduce,
\begin{equation}\label{5-10}
\begin{split}
\int_\Omega&\frac{\partial }{\partial z_{l,\bar{h}}}\left( \sum_{j=1}^N\left (V_{\varepsilon, z_j, \hat{q}_{\varepsilon,j}}-\hat{q}_{\varepsilon,j}|\ln\varepsilon|\right )^{p}_+-\left (V_{\varepsilon,Z}+\omega_{\varepsilon,Z}-q|\ln\varepsilon|\right )^{p}_+\right) \left( \eta_i\frac{\partial V_{\varepsilon, Z,i}}{\partial x_\hbar}\right)\\
&=-p\int_\Omega \left (V_{\varepsilon,Z}+\omega_{\varepsilon,Z}-q|\ln\varepsilon|\right )^{p-1}_+\frac{\partial \omega_{\varepsilon,Z} }{\partial z_{l,\bar{h}}}\left( \eta_i\frac{\partial V_{\varepsilon, Z,i}}{\partial x_\hbar}\right)+O\left( \varepsilon^\gamma\right).
\end{split}
\end{equation}
Inserting \eqref{5-10} into \eqref{5-08}, we obtain
\begin{equation}\label{5-11}
\begin{split}
\frac{\partial b_{j,\tilde{h}}}{\partial z_{l,\bar{h}}}=&O\left( \int_\Omega \left (V_{\varepsilon,Z}+\omega_{\varepsilon,Z}-q|\ln\varepsilon|\right )^{p-1}_+\frac{\partial \omega_{\varepsilon,Z} }{\partial z_{l,\bar{h}}}\left( \eta_i\frac{\partial V_{\varepsilon, Z,i}}{\partial x_\hbar}\right)\right) +O\left( \varepsilon^\gamma\right).
\end{split}
\end{equation}

On the other hand, we note that   $ \int_{\Omega}\varepsilon^2\left (K(x) \nabla\omega_{\varepsilon,Z}| \nabla \left( \eta_i\frac{\partial V_{\varepsilon, Z,i}}{\partial x_\hbar}\right)\right )=0$. Taking $ \frac{\partial}{\partial z_{l,\bar{h}}} $ into both sides of this equality, we have
\begin{equation}\label{5-12}
\int_{\Omega}\varepsilon^2 \left (K(x) \nabla \frac{\partial \omega_{\varepsilon,Z} }{\partial z_{l,\bar{h}}}| \nabla \left( \eta_i\frac{\partial V_{\varepsilon, Z,i}}{\partial x_\hbar}\right)\right )=-\int_{\Omega} \varepsilon^2 \left (K(x) \nabla\omega_{\varepsilon,Z}| \nabla \frac{\partial  }{\partial z_{l,\bar{h}}}\left( \eta_i\frac{\partial V_{\varepsilon, Z,i}}{\partial x_\hbar}\right)\right ) .
\end{equation}
One computes directly  that
\begin{equation}\label{5-13}
\begin{split}
-\int_{\Omega} \varepsilon^2 \left (K(x) \nabla\omega_{\varepsilon,Z}| \nabla \frac{\partial  }{\partial z_{l,\bar{h}}}\left( \eta_i\frac{\partial V_{\varepsilon, Z,i}}{\partial x_\hbar}\right)\right )=\int_{\Omega} \varepsilon^2 \text{div} \left (K(x) \nabla\omega_{\varepsilon,Z}  \right )\frac{\partial  }{\partial z_{l,\bar{h}}}\left( \eta_i\frac{\partial V_{\varepsilon, Z,i}}{\partial x_\hbar}\right).
\end{split}
\end{equation}
Since $ ||l_\varepsilon+R_\varepsilon(\omega_{\varepsilon,Z})||_{L^p}\leq C\varepsilon^{\frac{2}{p}+\gamma} $, by \eqref{5-000}  and \eqref{5-06}, we get
\begin{equation*}
L_\varepsilon \omega_{\varepsilon,Z}= l_\varepsilon+  R_\varepsilon(\omega_{\varepsilon,Z})+\sum_{j=1}^N\sum_{\tilde{h}=1}^2b_{j,\tilde{h}}\left( -\varepsilon^2\text{div}\left (K(z_j)\nabla   \frac{\partial V_{\varepsilon, z_j, \hat{q}_{\varepsilon,j}}}{\partial x_{\tilde{h}}}\right )\right)=O\left( \varepsilon^{\frac{2}{p}+\gamma}\right) \  \ \ \text{in}\ L^p(\Omega),
\end{equation*}
which implies that
\begin{equation*}
\begin{split}
-\varepsilon^2 \text{div}\left (K(x) \nabla\omega_{\varepsilon,Z}  \right )=&L_\varepsilon \omega_{\varepsilon,Z}
+p\left( V_{\varepsilon,Z}-q|\ln\varepsilon|\right)^{p-1}_+\omega_{\varepsilon,Z}
= O\left( \varepsilon^{\frac{2}{p}+\gamma}\right) \  \ \ \text{in}\ L^p(\Omega).
\end{split}
\end{equation*}
Taking this into \eqref{5-13}, using the definition of $ V_{\varepsilon, z_j, \hat{q}_{\varepsilon,j}} $ and \eqref{5-03}, one has
\begin{equation}\label{5-14}
\begin{split}
-\int_{\Omega} \varepsilon^2 \left (K(x) \nabla\omega_{\varepsilon,Z}| \nabla \frac{\partial  }{\partial z_{l,\bar{h}}}\left( \eta_i\frac{\partial V_{\varepsilon, Z,i}}{\partial x_\hbar}\right)\right )=O\left( \varepsilon^\gamma\right).
\end{split}
\end{equation}
We also note that
\begin{equation}\label{5-15}
\begin{split}
\int_{\Omega}\varepsilon^2& \left (K(x) \nabla \frac{\partial \omega_{\varepsilon,Z} }{\partial z_{l,\bar{h}}}| \nabla \left( \eta_i\frac{\partial V_{\varepsilon, Z,i}}{\partial x_\hbar}\right)\right )=\int_{\Omega}-\varepsilon^2  \text{div}\left (K(x) \nabla \left( \eta_i\frac{\partial V_{\varepsilon, Z,i}}{\partial x_\hbar}\right)\right )\frac{\partial \omega_{\varepsilon,Z} }{\partial z_{l,\bar{h}}}\\
=&\int_{\Omega}\bigg\{\eta_ip\left (V_{\varepsilon, z_i, \hat{q}_{\varepsilon,i}}-\hat{q}_{\varepsilon,i}|\ln\varepsilon|\right )^{p-1}_+\frac{\partial V_{\varepsilon, z_i, \hat{q}_{\varepsilon,i}}}{\partial x_\hbar}+\eta_i\varepsilon^2\text{div}\left (\frac{\partial K(x)}{\partial x_\hbar}\nabla V_{\varepsilon, Z,i}  \right ) \\
& -2 \varepsilon^2\left (K(x)\nabla\eta_i|\nabla   \frac{\partial V_{\varepsilon, Z,i}}{\partial x_\hbar}  \right ) -\varepsilon^2\text{div}\left (K(x)\nabla   \eta_i  \right )\frac{\partial V_{\varepsilon, Z,i}}{\partial x_\hbar} \bigg\}\frac{\partial \omega_{\varepsilon,Z} }{\partial z_{l,\bar{h}}}\\
=& p\int_{\Omega}\eta_i\frac{\partial \omega_{\varepsilon,Z} }{\partial z_{l,\bar{h}}} \left (V_{\varepsilon, z_i, \hat{q}_{\varepsilon,i}}-\hat{q}_{\varepsilon,i}|\ln\varepsilon|\right )^{p-1}_+\frac{\partial V_{\varepsilon, z_i, \hat{q}_{\varepsilon,i}}}{\partial x_\hbar}+O\left( \varepsilon^2|\ln\varepsilon|^{4}\bigg|\bigg|\frac{\partial \omega_{\varepsilon,Z} }{\partial z_{l,\bar{h}}}\bigg|\bigg|_{L^\infty}\right).
\end{split}
\end{equation}
Combining \eqref{5-12} with \eqref{5-14} and \eqref{5-15}, we obtain
\begin{equation}\label{5-16}
\begin{split}
\int_{\Omega}\eta_i\frac{\partial \omega_{\varepsilon,Z} }{\partial z_{l,\bar{h}}} \left (V_{\varepsilon, z_i, \hat{q}_{\varepsilon,i}}-\hat{q}_{\varepsilon,i}|\ln\varepsilon|\right )^{p-1}_+\frac{\partial V_{\varepsilon, z_i, \hat{q}_{\varepsilon,i}}}{\partial x_\hbar}=O\left(  \varepsilon^\gamma +\varepsilon^2|\ln\varepsilon|^{4}\bigg|\bigg|\frac{\partial \omega_{\varepsilon,Z} }{\partial z_{l,\bar{h}}}\bigg|\bigg|_{L^\infty}\right).
\end{split}
\end{equation}

Taking \eqref{5-16} into \eqref{5-11} and using \eqref{203} and Proposition \ref{exist and uniq of w}, we conclude that
\begin{equation}\label{5-17}
\begin{split}
\frac{\partial b_{j,\tilde{h}}}{\partial z_{l,\bar{h}}}=& O\left( \varepsilon^\gamma+ \varepsilon^{1+\gamma}\bigg|\bigg|\frac{\partial \omega_{\varepsilon,Z} }{\partial z_{l,\bar{h}}}\bigg|\bigg|_{L^\infty}\right).
\end{split}
\end{equation}

Now we calculate the $ L^\infty $ norm of $ \frac{\partial \omega_{\varepsilon,Z} }{\partial z_{l,\bar{h}}} $. Taking $ \frac{\partial}{\partial z_{l,\bar{h}}} $ in both sides of \eqref{5-00}, we get
\begin{equation}\label{5-18}
\begin{split}
-\varepsilon^2\text{div}&\left( K(x)\nabla \frac{\partial\omega_{\varepsilon,Z}}{\partial z_{l,\bar{h}}}\right)-p\left( V_{\varepsilon,Z}+\omega_{\varepsilon,Z}-q|\ln\varepsilon|\right)^{p-1}_+\frac{\partial\omega_{\varepsilon,Z}}{\partial z_{l,\bar{h}}}\\
=&\varepsilon^2\text{div}\left( K(x)\nabla \frac{\partial V_{\varepsilon,Z}}{\partial z_{l,\bar{h}}}\right)+p\left( V_{\varepsilon,Z}+\omega_{\varepsilon,Z}-q|\ln\varepsilon|\right)^{p-1}_+\frac{\partial V_{\varepsilon,Z}}{\partial z_{l,\bar{h}}}\\
&+\sum_{j=1}^N\sum_{\tilde{h}=1}^2\frac{\partial b_{j,\tilde{h}}}{\partial z_{l,\bar{h}}}p\left(   V_{\varepsilon, z_j, \hat{q}_{\varepsilon,j}}-\hat{q}_{\varepsilon,j}|\ln\varepsilon| \right )^{p-1}_+\frac{\partial V_{\varepsilon, z_j, \hat{q}_{\varepsilon,j}}}{\partial x_{\tilde{h}}}\\
&+\sum_{j=1}^N\sum_{\tilde{h}=1}^2b_{j,\tilde{h}}\frac{\partial }{\partial z_{l,\bar{h}}}\left( p\left(   V_{\varepsilon, z_j, \hat{q}_{\varepsilon,j}}-\hat{q}_{\varepsilon,j}|\ln\varepsilon| \right )^{p-1}_+\frac{\partial V_{\varepsilon, z_j, \hat{q}_{\varepsilon,j}}}{\partial x_{\tilde{h}}}\right) .
\end{split}
\end{equation}
Note that the function $ \frac{\partial\omega_{\varepsilon,Z}}{\partial z_{l,\bar{h}}} $ may not be in $ E_{\varepsilon,Z} $. We make the following decomposition:
\begin{equation}\label{5-19}
\frac{\partial\omega_{\varepsilon,Z}}{\partial z_{l,\bar{h}}}=\omega_{\varepsilon}^*+\sum_{j=1}^N\sum_{\tilde{h}=1}^2C_{j,\tilde{h}}\zeta_j\frac{\partial V_{\varepsilon, z_j, \hat{q}_{\varepsilon,j}}}{\partial x_{\tilde{h}}},
\end{equation}
where $ \omega_{\varepsilon}^*\in E_{\varepsilon,Z} $,  $\zeta_j(x):=\eta\left (\frac{|T_{z_j}\left( x-z_j\right) |}{s_{\varepsilon,j}}\right )$ and $ C_{j,\tilde{h}} $ is determined by
\begin{equation*}
\begin{split}
\sum_{j=1}^N\sum_{\tilde{h}=1}^2&C_{j,\tilde{h}}\int_{\Omega}\zeta_j\frac{\partial V_{\varepsilon, z_j, \hat{q}_{\varepsilon,j}}}{\partial x_{\tilde{h}}}\cdot\left( -\varepsilon^2\text{div}\left( K(x)\nabla\left(\eta_i\frac{\partial V_{\varepsilon,Z,i}}{\partial x_\hbar} \right) \right)\right) \\
=& \int_{\Omega}\frac{\partial\omega_{\varepsilon,Z}}{\partial z_{l,\bar{h}}}\cdot\left( -\varepsilon^2\text{div}\left( K(x)\nabla\left(\eta_i\frac{\partial V_{\varepsilon,Z,i}}{\partial x_\hbar} \right)\right)\right)  \ \ \ \ i=1,\cdots,N,\ \hbar=1,2.
\end{split}
\end{equation*}
Note that
\begin{equation*}
\int_{\Omega}\zeta_j\frac{\partial V_{\varepsilon, z_j, \hat{q}_{\varepsilon,j}}}{\partial x_{\tilde{h}}}\cdot\left( -\varepsilon^2\text{div}\left( K(x)\nabla\left(\eta_i\frac{\partial V_{\varepsilon,Z,i}}{\partial x_\hbar} \right) \right)\right) = (\tilde{M}_i)_{\tilde{h},\hbar}\delta_{i,j}+o(1),
\end{equation*}
where $ \tilde{M}_{i} $ are $ m $ positive definite matrices. Combining this with \eqref{5-12} and \eqref{5-14}, we obtain
\begin{equation}\label{5-20}
C_{j,\tilde{h}}=O\left( \varepsilon^\gamma \right).
\end{equation}
Taking \eqref{5-19} into \eqref{5-18}, we get
\begin{equation}\label{5-21}
\begin{split}
-\varepsilon^2&\text{div}\left( K(x)\nabla  \omega_{\varepsilon}^*\right)-p\left( V_{\varepsilon,Z} -q|\ln\varepsilon|\right)^{p-1}_+\omega_{\varepsilon}^*\\
=&\sum_{j=1}^N\sum_{\tilde{h}=1}^2C_{j,\tilde{h}} \varepsilon^2\text{div}\left( K(x)\nabla \left( \zeta_j\frac{\partial V_{\varepsilon, z_j, \hat{q}_{\varepsilon,j}}}{\partial x_{\tilde{h}}}\right) \right)+ p\sum_{j=1}^N\sum_{\tilde{h}=1}^2C_{j,\tilde{h}}\left( V_{\varepsilon,Z} -q|\ln\varepsilon|\right)^{p-1}_+\zeta_j\frac{\partial V_{\varepsilon, z_j, \hat{q}_{\varepsilon,j}}}{\partial x_{\tilde{h}}}\\
&+p\left( \left( V_{\varepsilon,Z}+\omega_{\varepsilon,Z}-q|\ln\varepsilon|\right)^{p-1}_+-\left( V_{\varepsilon,Z} -q|\ln\varepsilon|\right)^{p-1}_+\right) \frac{\partial\omega_{\varepsilon,Z}}{\partial z_{l,\bar{h}}}\\
&-p\sum_{j=1}^N\left (V_{\varepsilon, z_j, \hat{q}_{\varepsilon,j}}-\hat{q}_{\varepsilon,j}|\ln\varepsilon|\right )^{p-1}_+\frac{\partial }{\partial z_{l,\bar{h}}}\left (V_{\varepsilon, z_j, \hat{q}_{\varepsilon,j}}-\hat{q}_{\varepsilon,j}|\ln\varepsilon|\right )   \\
&+p\left (V_{\varepsilon,Z}+\omega_{\varepsilon,Z}-q|\ln\varepsilon|\right )^{p-1}_+\frac{\partial V_{\varepsilon,Z}}{\partial z_{l,\bar{h}}}+\sum_{j=1}^N\sum_{\tilde{h}=1}^2\frac{\partial b_{j,\tilde{h}}}{\partial z_{l,\bar{h}}}p\left(   V_{\varepsilon, z_j, \hat{q}_{\varepsilon,j}}-\hat{q}_{\varepsilon,j}|\ln\varepsilon| \right )^{p-1}_+\frac{\partial V_{\varepsilon, z_j, \hat{q}_{\varepsilon,j}}}{\partial x_{\tilde{h}}}\\
&+\sum_{j=1}^N\sum_{\tilde{h}=1}^2b_{j,\tilde{h}}\frac{\partial }{\partial z_{l,\bar{h}}}\left( p\left(   V_{\varepsilon, z_j, \hat{q}_{\varepsilon,j}}-\hat{q}_{\varepsilon,j}|\ln\varepsilon| \right )^{p-1}_+\frac{\partial V_{\varepsilon, z_j, \hat{q}_{\varepsilon,j}}}{\partial x_{\tilde{h}}}\right) \\
=&A_1+A_2+A_3+A_4+A_5+A_6+A_7.
\end{split}
\end{equation}
By \eqref{5-01}, \eqref{5-06}, \eqref{5-17} and \eqref{5-20}, one computes directly that
\begin{equation*}
\begin{split}
A_1+A_2=&\left( |C_{j,\tilde{h}}|\varepsilon^{\frac{2}{p}-1}\right)=O\left(  \varepsilon^{\gamma+\frac{2}{p}-1}\right)\ \ \text{in}\ L^p(\Omega),
\end{split}
\end{equation*}
\begin{equation*}
\begin{split}
A_3=O\left(  \varepsilon^{\gamma+\frac{2}{p}}\bigg|\bigg|\frac{\partial \omega_{\varepsilon,Z} }{\partial z_{l,\bar{h}}}\bigg|\bigg|_{L^\infty}\right)\ \ \text{in}\ L^p(\Omega),
\end{split}
\end{equation*}
\begin{equation*}
\begin{split}
A_4+A_5=&O\left( \varepsilon^{\frac{2}{p}}|\ln\varepsilon|^{4}+  \varepsilon^{\gamma+\frac{2}{p}-1}\right) =O\left(  \varepsilon^{\gamma+\frac{2}{p}-1}\right)\ \ \text{in}\ L^p(\Omega),
\end{split}
\end{equation*}
\begin{equation*}
\begin{split}
A_6&=O\left(   \frac{\varepsilon^{ \frac{2}{p}}}{\varepsilon}\bigg|\frac{\partial b_{j,\tilde{h}}}{\partial z_{l,\bar{h}}}\bigg|\right)=O\left(   \varepsilon^{ \frac{2}{p}-1}\left(   \varepsilon^\gamma +  \varepsilon^{1+\gamma}   \bigg|\bigg|\frac{\partial \omega_{\varepsilon,Z} }{\partial z_{l,\bar{h}}}\bigg|\bigg|_{L^\infty}\right) \right)\ \ \text{in}\ L^p(\Omega),
\end{split}
\end{equation*}
\begin{equation*}
\begin{split}
A_7=O\left(  \varepsilon^{\gamma+\frac{2}{p}-1}\right)\ \ \text{in}\ L^p(\Omega).
\end{split}
\end{equation*}
Combining these with Proposition \ref{coercive esti} and \eqref{5-21}, we are led to
\begin{equation}\label{5-22}
||\omega_{\varepsilon}^*||_{L^\infty }\leq C\varepsilon^{-\frac{2}{p}}||L_\varepsilon\omega_{\varepsilon}^*||_{L^p }\leq C\varepsilon^{-\frac{2}{p}} \left(\varepsilon^{\gamma+\frac{2}{p}-1}+ \varepsilon^{\gamma+\frac{2}{p}}\bigg|\bigg|\frac{\partial \omega_{\varepsilon,Z} }{\partial z_{l,\bar{h}}}\bigg|\bigg|_{L^\infty} \right).
\end{equation}
From the decomposition \eqref{5-19}, we have
\begin{equation*}
\bigg|\bigg|\frac{\partial \omega_{\varepsilon,Z} }{\partial z_{l,\bar{h}}}\bigg|\bigg|_{L^\infty}\leq ||\omega_{\varepsilon}^*||_{L^\infty }+O\left( \frac{1}{\varepsilon^{1-\gamma}}\right) .
\end{equation*}
Taking this into \eqref{5-22}, we obtain
\begin{equation*}
||\omega_{\varepsilon}^*||_{L^\infty }\leq C  \frac{1}{\varepsilon^{1-\gamma}},
\end{equation*}
from which we deduce, $ \big|\big|\frac{\partial \omega_{\varepsilon,Z} }{\partial z_{l,\bar{h}}}\big|\big|_{L^\infty(\Omega)}=O\left( \frac{1}{\varepsilon^{1-\gamma}} \right)  $.

Finally, we prove that $ \omega_{\varepsilon,Z} $ is a $ C^1 $ map of $ Z\in \Lambda_{\varepsilon, N} $ in $ H^1(\Omega) $. To prove the continuity of $ \omega_{\varepsilon,Z} $ of $ Z $, let $ Z_j\to Z_0 $. By Proposition \ref{exist and uniq of w}, $ \omega_{\varepsilon,Z_j}  $ is uniformly bounded in $ L^\infty(\Omega) $. Thus  using \eqref{5-000} and \eqref{5-06}, we conclude that $ ||\omega_{\varepsilon,Z_j}||_{H^1_0(\Omega)} $ is bounded by a constant $ C $  which is independent of $ j $. Then there is a subsequence (still denoted by $ Z_j $) such that
\begin{equation*}
\omega_{\varepsilon,Z_j}\rightarrow \omega^{**}\ \ \ \  \text{weakly in } H^1_0(\Omega)
\end{equation*}
and
\begin{equation*}
\omega_{\varepsilon,Z_j}\rightarrow \omega^{**}\ \ \ \  \text{strongly in } L^2(\Omega).
\end{equation*}
Using the equation again, we can get that
\begin{equation*}
\omega_{\varepsilon,Z_j}\rightarrow \omega^{**}\ \ \ \  \text{strongly in } H^1_0(\Omega),
\end{equation*}
from which  we deduce that $ \omega^{**}\in E_{\varepsilon,Z_0} $ and $ \omega^{**} $ satisfies \eqref{216} with $ Z_j $ replaced by $ Z_0 $.  By the uniqueness, we get $ \omega^{**}=\omega_{\varepsilon,Z_0} $ and hence  $ \omega_{\varepsilon,Z} $ is continuous with respect  to  $Z$ in the norm of $ H^1_0(\Omega) $. Moreover, using similar method as  Proposition 3.7 in  \cite{CPY}, we can show that  $ \frac{\partial \omega_{\varepsilon,Z}}{\partial z_{l,\tilde{h}}} $ is continuous with respect  to $ Z $ in $ H^1(\Omega) $. The proof is thus complete.
\end{proof}

\section{The finite-dimensional energy expansion}

Our aim in the rest part of this paper is to find $Z$ properly so that $V_{\varepsilon, Z}+\omega_{\varepsilon, Z}$ is a solution of \eqref{key equa}.  In this section we give expansion of the main term $I(V_{\varepsilon, Z})$ in term of $Z$. From Proposition \ref{exist and uniq of w}, given any $ \varepsilon $ small and $ Z\in \Lambda_{\varepsilon, N} $, there exists a unique $ \omega_{\varepsilon,Z}\in E_{\varepsilon,Z} $ satisfying $ Q_\varepsilon L_\varepsilon \omega_{\varepsilon,Z}=Q_\varepsilon l_\varepsilon+Q_\varepsilon R_\varepsilon(\omega_{\varepsilon,Z}), $
i.e., for some $ b_{j,\tilde{h}}=b_{j,\tilde{h}}(Z) $
\begin{equation*}
 L_\varepsilon \omega_{\varepsilon,Z}= l_\varepsilon+  R_\varepsilon(\omega_{\varepsilon,Z})+\sum_{j=1}^N\sum_{\tilde{h}=1}^2b_{j,\tilde{h}}\left( -\varepsilon^2\text{div}\left (K(z_j)\nabla   \frac{\partial V_{\varepsilon, z_j, \hat{q}_{\varepsilon,j}}}{\partial x_{\tilde{h}}}\right )\right).
\end{equation*}
Thus, it suffices to find $Z$ properly so that
\begin{equation*}
 b_{j,\tilde{h}}(Z)=0,\ \ \forall\ j=1,\cdots,N,\ \tilde{h}=1,2.
\end{equation*}

Define
\begin{equation}\label{functional I}
I_\varepsilon(u)=\frac{\varepsilon^2}{2}\int_{\Omega}\left( K(x)\nabla u|\nabla u\right)-\frac{1}{p+1}\int_{\Omega}\left (u-q|\ln\varepsilon|\right )^{p+1}_+,
\end{equation}
and
\begin{equation*}
K_\varepsilon(Z)=I_\varepsilon(V_{\varepsilon,Z}+\omega_{\varepsilon,Z}).
\end{equation*}
From Proposition \ref{differ of w}, $ K_\varepsilon(\cdot) $ is a $ C^1 $ function.
The following lemma shows that, to find solutions of \eqref{key equa}, it suffices to find critical points of  $ K_\varepsilon(Z) $.
\begin{lemma}\label{choice of Z}
If $ Z\in \Lambda_{\varepsilon, N}$ is a critical point of $ K_\varepsilon(Z) $, then
$ V_{\varepsilon,Z}+\omega_{\varepsilon,Z} $  is a solution to \eqref{key equa}.
\end{lemma}
\begin{proof}
It follows from Proposition \ref{exist and uniq of w} that
\begin{equation}\label{611}
\begin{split}
\langle I'(V_{\varepsilon,Z}+\omega_{\varepsilon,Z}), \phi\rangle=\sum_{j=1}^N\sum_{\tilde{h}=1}^2b_{j,\tilde{h}}\int_{\Omega} -\varepsilon^2\text{div}\left (K(z_j)\nabla   \frac{\partial V_{\varepsilon, z_j, \hat{q}_{\varepsilon,j}}}{\partial x_{\tilde{h}}}\right )\phi,\ \ \forall \phi\in H^1_0(\Omega).
\end{split}
\end{equation}
If we can choose $ Z $ such that the corresponding constants $ b_{j,\tilde{h}} $  are all zero, then $ V_{\varepsilon,Z}+\omega_{\varepsilon,Z} $  is a solution to \eqref{key equa}.
Suppose that $ Z $ is a critical point of $ K_\varepsilon(Z) $. Then from \eqref{611} and Proposition \ref{differ of w}, for $ i=1,\cdots,N, \hbar=1,2 $
\begin{equation}\label{603}
\begin{split}
0=&\frac{\partial K_\varepsilon(Z)}{\partial z_{i,\hbar}}=\left \langle I'(V_{\varepsilon,Z}+\omega_{\varepsilon,Z}), \frac{\partial (V_{\varepsilon,Z}+\omega_{\varepsilon,Z})}{\partial z_{i,\hbar}}\right \rangle\\
=&\sum_{j=1}^N\sum_{\tilde{h}=1}^2b_{j,\tilde{h}}\int_{\Omega} -\varepsilon^2\text{div}\left (K(z_j)\nabla   \frac{\partial V_{\varepsilon, z_j, \hat{q}_{\varepsilon,j}}}{\partial x_{\tilde{h}}}\right )\frac{\partial (V_{\varepsilon,Z}+\omega_{\varepsilon,Z})}{\partial z_{i,\hbar}}\\
=&\sum_{j=1}^N\sum_{\tilde{h}=1}^2b_{j,\tilde{h}}((M_i)_{\tilde{h},\hbar}\delta_{i,j}+o(\varepsilon^\gamma))+O\left (\varepsilon\sum_{j=1}^N\sum_{\tilde{h}=1}^2|b_{j,\tilde{h}}|\bigg|\bigg|\frac{\partial \omega_{\varepsilon,Z}}{\partial z_{j,\tilde{h}}}\bigg|\bigg|_{L^\infty}\right )\\
=&\sum_{j=1}^N\sum_{\tilde{h}=1}^2b_{j,\tilde{h}}((M_i)_{\tilde{h},\hbar}\delta_{i,j}+o(\varepsilon^\gamma))+O\left (\varepsilon^\gamma\sum_{j=1}^N\sum_{\tilde{h}=1}^2|b_{j,\tilde{h}}|\right ),
\end{split}
\end{equation}
from which we deduce that  $ b_{j,\tilde{h}}(Z)=0. $

\end{proof}

Now we give the  expansion of the energy functional $ K_\varepsilon(Z) $. First we obtain that $I_\varepsilon(V_{\varepsilon,Z})$ is its main part.
\begin{proposition}\label{pro401}
There holds
\begin{equation*}
K_{\varepsilon}(Z)=I_\varepsilon(V_{\varepsilon,Z})+O\left( \varepsilon^{2+\gamma}\right).
\end{equation*}
\end{proposition}

\begin{proof}
Note  that
\begin{equation*}
\begin{split}
K_{\varepsilon}(Z)=&I_\varepsilon(V_{\varepsilon,Z})+\varepsilon^2\int_{\Omega}\left( K(x)\nabla V_{\varepsilon,Z}|\nabla \omega_{\varepsilon,Z}\right) +\frac{\varepsilon^2}{2}\int_{\Omega}\left( K(x)\nabla \omega_{\varepsilon,Z}|\nabla \omega_{\varepsilon,Z}\right)\\
&-\frac{1}{p+1}\bigg( \int_{\Omega}\left (V_{\varepsilon,Z}+\omega_{\varepsilon,Z}-q|\ln\varepsilon|\right )^{p+1}_+-\int_{\Omega}\left (V_{\varepsilon,Z}-q|\ln\varepsilon|\right )^{p+1}_+\bigg).
\end{split}
\end{equation*}
It follows from Proposition \ref{exist and uniq of w} that
\begin{equation*}
\begin{split}
\int_{\Omega}&\left (V_{\varepsilon,Z}+\omega_{\varepsilon,Z}-q|\ln\varepsilon|\right )^{p+1}_+-\int_{\Omega}\left (V_{\varepsilon,Z}-q|\ln\varepsilon|\right )^{p+1}_+\\
&=(p+1)\sum_{j=1}^N\int_{B_{Ls_{\varepsilon,j}}(z_j)}\left (V_{\varepsilon,Z}-q|\ln\varepsilon|\right )^{p}_+\omega_{\varepsilon,Z}
+O\left( \sum_{j=1}^N\int_{B_{Ls_{\varepsilon,j}}(z_j)}\left (V_{\varepsilon,Z}-q|\ln\varepsilon|\right )^{p-1}_+\omega_{\varepsilon,Z}^2\right)\\
&=O\left( \varepsilon^{2+\gamma}\right).
\end{split}
\end{equation*}
Since $  -\varepsilon^2\text{div}(K(x)\nabla V_{\varepsilon, Z,j})=\left (V_{\varepsilon,z_j,\hat{q}_{\varepsilon,j}}-\hat{q}_{\varepsilon,j}|\ln\varepsilon|\right )^p_+, $
we get
\begin{equation*}
\begin{split}
\varepsilon^2\int_{\Omega}\left( K(x)\nabla V_{\varepsilon,Z}|\nabla \omega_{\varepsilon,Z}\right)
=&\sum_{j=1}^N\int_{B_{Ls_{\varepsilon,j}}(z_j)}\left (V_{\varepsilon,z_j,\hat{q}_{\varepsilon,j}}-\hat{q}_{\varepsilon,j}|\ln\varepsilon|\right )^p_+\omega_{\varepsilon,Z}
=O\left(  \varepsilon^{2+\gamma}\right).
\end{split}
\end{equation*}
As for the term $ \frac{\varepsilon^2}{2}\int_{\Omega}\left( K(x)\nabla \omega_{\varepsilon,Z}|\nabla \omega_{\varepsilon,Z}\right), $ since  $ \omega_{\varepsilon,Z}\in E_{\varepsilon,Z} $,  $ -\varepsilon^2\text{div}(K(x)\nabla \omega_{\varepsilon,Z})\in F_{\varepsilon,Z} $. So
\begin{equation*}
\begin{split}
Q_\varepsilon L_\varepsilon \omega_{\varepsilon,Z}
=&-\varepsilon^2\text{div}(K(x)\nabla \omega_{\varepsilon,Z})-Q_\varepsilon\left (p\left (V_{\varepsilon,Z}-q|\ln\varepsilon|\right )^{p-1}_+\omega_{\varepsilon,Z}\right ),
\end{split}
\end{equation*}
which combined with $ Q_\varepsilon L_\varepsilon \omega_{\varepsilon,Z}=Q_\varepsilon l_\varepsilon+Q_\varepsilon R_\varepsilon(\omega_{\varepsilon,Z})  $ yields
\begin{equation*}
-\varepsilon^2\text{div}(K(x)\nabla \omega_{\varepsilon,Z})=Q_\varepsilon\left (p\left (V_{\varepsilon,Z}-q|\ln\varepsilon|\right )^{p-1}_+\omega_{\varepsilon,Z}\right )+Q_\varepsilon l_\varepsilon+Q_\varepsilon R_\varepsilon(\omega_{\varepsilon,Z}).
\end{equation*}
Hence by Lemma \ref{lemA-5} and Proposition \ref{exist and uniq of w}, 
\begin{equation*}
\begin{split}
\varepsilon^2\int_{\Omega}&\left( K(x)\nabla \omega_{\varepsilon,Z}|\nabla \omega_{\varepsilon,Z}\right)\\
=&\int_{\Omega}Q_\varepsilon\left (p\left (V_{\varepsilon,Z}-q|\ln\varepsilon|\right )^{p-1}_+ \omega_{\varepsilon,Z}\right )\omega_{\varepsilon,Z}+\int_{\Omega}Q_\varepsilon l_\varepsilon\omega_{\varepsilon,Z}+\int_{\Omega}Q_\varepsilon R_\varepsilon(\omega_{\varepsilon,Z})\omega_{\varepsilon,Z}\\
=&O\bigg(|| \left (V_{\varepsilon,Z}-q|\ln\varepsilon|\right )^{p-1}_+\omega_{\varepsilon,Z}||_{L^1}||\omega_{\varepsilon,Z}||_{L^\infty}+|| l_\varepsilon||_{L^1}||\omega_{\varepsilon,Z}||_{L^\infty}+|| R_\varepsilon(\omega_{\varepsilon,Z})||_{L^1} ||\omega_{\varepsilon,Z}||_{L^\infty}\bigg )\\
=&O\left( \varepsilon^{2+\gamma}\right).
\end{split}
\end{equation*}
To conclude, we have $ K_{\varepsilon}(Z)=I_\varepsilon(V_{\varepsilon,Z})+O\left( \varepsilon^{2+\gamma}\right). $

\end{proof}
For $ I_\varepsilon(V_{\varepsilon,Z}) $,  the energy expansion is as follows.
\begin{proposition}\label{order of main term}
There holds
\begin{equation}\label{344}
\begin{split}
I_\varepsilon(V_{\varepsilon,Z})=&\sum_{j=1}^N \pi\varepsilon^2 |\ln \varepsilon|q(z_j)^2\sqrt{\det K(z_j)}+\sum_{j=1}^N\frac{(p-1)}{4}\pi\varepsilon^2q(z_j)^2\sqrt{\det K(z_j)}\\
&-\sum_{j=1}^N2\pi^2\varepsilon^2 q(z_j)^2  \det K(z_j)\bar{S}_K(z_j,z_j)\\
&-\sum_{1\leq i\neq j\leq N} 2\pi^2\varepsilon^2q(z_i)q(z_j)\sqrt{\det K(z_i)}\sqrt{\det K(z_j)} G_K(z_i,z_j)+O\left( \frac{\varepsilon^2 \ln|\ln\varepsilon| }{|\ln\varepsilon|}\right).
\end{split}
\end{equation}
As a consequence, 
\begin{equation}\label{345}
\begin{split}
K_\varepsilon(Z)=&\sum_{j=1}^N \pi\varepsilon^2 |\ln \varepsilon|q(z_j)^2\sqrt{\det K(z_j)}+\sum_{j=1}^N\frac{(p-1)}{4}\pi\varepsilon^2q(z_j)^2\sqrt{\det K(z_j)}\\
&-\sum_{j=1}^N2\pi^2\varepsilon^2 q(z_j)^2  \det K(z_j)\bar{S}_K(z_j,z_j)\\
&-\sum_{1\leq i\neq j\leq N} 2\pi^2\varepsilon^2q(z_i)q(z_j)\sqrt{\det K(z_i)}\sqrt{\det K(z_j)} G_K(z_i,z_j)+O\left( \frac{\varepsilon^2 \ln|\ln\varepsilon| }{|\ln\varepsilon|}\right).
\end{split}
\end{equation}
\end{proposition}
\begin{proof}
Note that
\begin{equation}\label{218}
\begin{split}
I_\varepsilon(V_{\varepsilon,Z})=&\frac{1}{2}\int_{\Omega}-\varepsilon^2\text{div}\left( K(x)\nabla V_{\varepsilon,Z}  \right)V_{\varepsilon,Z}-\frac{1}{p+1}\int_{\Omega}\left (V_{\varepsilon,Z}-q|\ln\varepsilon|\right )^{p+1}_+ \\
=&\frac{1}{2}\sum_{j=1}^N\int_{\Omega}\left (V_{\varepsilon,z_j,\hat{q}_{\varepsilon,j}}-\hat{q}_{\varepsilon,j}|\ln\varepsilon|\right )^p_+V_{\varepsilon,Z,j} +\frac{1}{2}\sum_{1\leq i\neq j\leq N}\int_{\Omega}\left (V_{\varepsilon,z_j,\hat{q}_{\varepsilon,j}}-\hat{q}_{\varepsilon,j}|\ln\varepsilon|\right )^p_+V_{\varepsilon,Z,i}\\
&-\frac{1}{p+1}\int_{\Omega}\left (V_{\varepsilon,Z}-q|\ln\varepsilon|\right )^{p+1}_+.
\end{split}
\end{equation}
By   the definition of $ V_{\varepsilon,Z,j} $, 
\begin{equation*}
\begin{split}
\int_{\Omega}\left (V_{\varepsilon,z_j,\hat{q}_{\varepsilon,j}}-\hat{q}_{\varepsilon,j}|\ln\varepsilon|\right )^p_+V_{\varepsilon,Z,j}
=&\hat{q}_{\varepsilon,j}|\ln\varepsilon|\int_{\Omega}\left (V_{\varepsilon,z_j,\hat{q}_{\varepsilon,j}}-\hat{q}_{\varepsilon,j}|\ln\varepsilon|\right )^p_+
+\int_{\Omega}\left (V_{\varepsilon,z_j,\hat{q}_{\varepsilon,j}}-\hat{q}_{\varepsilon,j}|\ln\varepsilon|\right )^{p+1}_+
\\&+\int_{\Omega}\left (V_{\varepsilon,z_j,\hat{q}_{\varepsilon,j}}-\hat{q}_{\varepsilon,j}|\ln\varepsilon|\right )^p_+H_{\varepsilon,z_j,\hat{q}_{\varepsilon,j}}.
\end{split}
\end{equation*}
By $ T_{z_j}^{-1}(T_{z_j}^{-1})^T=K(z_j) $ and \eqref{PI}, we get
\begin{equation*}
\begin{split}
\hat{q}_{\varepsilon,j}|\ln\varepsilon|\int_{\Omega}\left (V_{\varepsilon,z_j,\hat{q}_{\varepsilon,j}}-\hat{q}_{\varepsilon,j}|\ln\varepsilon|\right )^p_+
=&\hat{q}_{\varepsilon,j}|\ln\varepsilon| s_{\varepsilon,j}^2\left (\frac{\varepsilon}{s_{\varepsilon,j}}\right )^{\frac{2p}{p-1}}\int_{|T_{z_j}x|\leq 1}\phi\left (T_{z_j}x\right )^pdx\\
=&\frac{2\pi\varepsilon^2|\ln\varepsilon|^2}{|\ln s_{\varepsilon,j}|}\hat{q}_{\varepsilon,j}^2\sqrt{\det(K(z_j))},
\end{split}
\end{equation*}
and
\begin{equation*}
\begin{split}
\int_{\Omega}\left (V_{\varepsilon,z_j,\hat{q}_{\varepsilon,j}}-\hat{q}_{\varepsilon,j}|\ln\varepsilon|\right )^{p+1}_+
=& s_{\varepsilon,j}^2\left( \frac{\varepsilon}{s_{\varepsilon,j}}\right)^{\frac{2(p+1)}{p-1}}\sqrt{\det(K(z_j))}\cdot \frac{(p+1)\pi}{2}|\phi'(1)|^2\\
=&\frac{(p+1)\pi\varepsilon^2|\ln\varepsilon|^2}{2 |\ln s_{\varepsilon,j}|^2}\hat{q}_{\varepsilon,j}^2\sqrt{\det(K(z_j))}.
\end{split}
\end{equation*}
By Lemma \ref{H estimate}, we have
\begin{equation*}
\begin{split}
\int_{\Omega}&\left (V_{\varepsilon,z_j,\hat{q}_{\varepsilon,j}}-\hat{q}_{\varepsilon,j}|\ln\varepsilon|\right )^p_+H_{\varepsilon,z_j,\hat{q}_{\varepsilon,j}}\\
&=\frac{2\pi\hat{q}_{\varepsilon,j}\sqrt{\det K(z_j)}|\ln\varepsilon|}{|\ln  s_{\varepsilon,j}|}\int_{\Omega}\left (V_{\varepsilon,z_j,\hat{q}_{\varepsilon,j}}-\hat{q}_{\varepsilon,j}|\ln\varepsilon|\right )^p_+\bar{S}_K(x,z_j)dx+O\left( \varepsilon^{2+\gamma}\right)  \\
&=\frac{4\pi^2\varepsilon^2 \bar{S}_K(z_j,z_j)|\ln\varepsilon|^2}{|\ln s_{\varepsilon,j}|^2}\hat{q}_{\varepsilon,j}^2\cdot \det(K(z_j))+O\left( \varepsilon^{2+\gamma} \right),
\end{split}
\end{equation*}
from which we deduce that, 
\begin{equation}\label{219}
\begin{split}
\int_{\Omega}\left (V_{\varepsilon,z_j,\hat{q}_{\varepsilon,j}}-\hat{q}_{\varepsilon,j}|\ln\varepsilon|\right )^p_+V_{\varepsilon,Z,j}
=&\frac{2\pi\varepsilon^2|\ln\varepsilon|^2}{|\ln s_{\varepsilon,j}|}\hat{q}_{\varepsilon,j}^2\sqrt{\det(K(z_j))}+\frac{(p+1)\pi\varepsilon^2|\ln\varepsilon|^2}{2 |\ln s_{\varepsilon,j}|^2}\hat{q}_{\varepsilon,j}^2\sqrt{\det(K(z_j))}\\
&+\frac{4\pi^2\varepsilon^2 \bar{S}_K(z_j,z_j)|\ln\varepsilon|^2}{|\ln s_{\varepsilon,j}|^2}\hat{q}_{\varepsilon,j}^2\cdot \det(K(z_j))+O\left( \varepsilon^{2+\gamma} \right).
\end{split}
\end{equation}
By Lemmas \ref{Green expansion}, \ref{H estimate}, the definition of $ \Lambda_{\varepsilon, N} $ and the fact that $ \lim_{\varepsilon\to 0}\varepsilon|\ln\varepsilon|=0 $, for $ 1\leq i\neq j\leq N $
\begin{equation}\label{220}
\begin{split}
\int_{\Omega}&\left (V_{\varepsilon,z_j,\hat{q}_{\varepsilon,j}}-\hat{q}_{\varepsilon,j}|\ln\varepsilon|\right )^p_+V_{\varepsilon,Z,i}\\
&=\frac{2\pi\hat{q}_{\varepsilon,i}\sqrt{\det K(z_i)}|\ln\varepsilon|}{\ln|s_{\varepsilon,i}|}\int_{\Omega}\left (V_{\varepsilon,z_j,\hat{q}_{\varepsilon,j}}-\hat{q}_{\varepsilon,j}|\ln\varepsilon|\right )^p_+G_K(x,z_i)dx+O\left( \varepsilon^{2+\gamma} \right)\\
&=\frac{4\pi^2\varepsilon^2G_K(z_j,z_i)|\ln\varepsilon|^2}{|\ln s_{\varepsilon,i}||\ln s_{\varepsilon,j}|}\hat{q}_{\varepsilon,i}\hat{q}_{\varepsilon,j}\sqrt{\det(K(z_i))}\sqrt{\det(K(z_j))}+O\left( \varepsilon^{2+\gamma} \right).
\end{split}
\end{equation}

Finally by \eqref{203},
\begin{equation}\label{223}
\begin{split}
\int_{\Omega}&\left (V_{\varepsilon,Z}-q|\ln\varepsilon|\right )^{p+1}_+\\
=&\sum_{j=1}^N\int_{B_{Ls_{\varepsilon,j}}(z_j)}\left( V_{\varepsilon, z_j, \hat{q}_{\varepsilon,j}}-\hat{q}_{\varepsilon,j}|\ln\varepsilon|+O\left(  \varepsilon^\gamma \right) \right)^{p+1}_+\\
=&\sum_{j=1}^N\int_{\Omega}\left (V_{\varepsilon, z_j, \hat{q}_{\varepsilon,j}}-\hat{q}_{\varepsilon,j}|\ln\varepsilon|\right )^{p+1}_++O\left(  \varepsilon^\gamma \sum_{j=1}^N\int_{\Omega}\left (V_{\varepsilon, z_j, \hat{q}_{\varepsilon,j}}-\hat{q}_{\varepsilon,j}|\ln\varepsilon|\right )^{p}_+\right) \\
=&\sum_{j=1}^N\frac{(p+1)\pi\varepsilon^2|\ln\varepsilon|^2}{2 |\ln s_{\varepsilon,j}|^2}\hat{q}_{\varepsilon,j}^2\sqrt{\det(K(z_j))}+O\left( \varepsilon^{2+\gamma} \right).
\end{split}
\end{equation}
Taking \eqref{219}, \eqref{220} and \eqref{223} into \eqref{218}, one has
\begin{equation}\label{225}
\begin{split}
I_\varepsilon(V_{\varepsilon,Z})=&\sum_{j=1}^N\frac{\pi\varepsilon^2|\ln\varepsilon|^2}{|\ln s_{\varepsilon,j}|}\hat{q}_{\varepsilon,j}^2\sqrt{\det(K(z_j))}+\sum_{j=1}^N\frac{(p+1)\pi\varepsilon^2|\ln\varepsilon|^2}{4 |\ln s_{\varepsilon,j}|^2}\hat{q}_{\varepsilon,j}^2\sqrt{\det(K(z_j))}\\
&+\sum_{j=1}^N\frac{2\pi^2\varepsilon^2 \bar{S}_K(z_j,z_j)|\ln\varepsilon|^2}{|\ln s_{\varepsilon,j}|^2}\hat{q}_{\varepsilon,j}^2\cdot \det(K(z_j))\\
&+\sum_{1\leq i\neq j\leq N}\frac{2\pi^2\varepsilon^2G_K(z_j,z_i)|\ln\varepsilon|^2}{|\ln s_{\varepsilon,i}||\ln s_{\varepsilon,j}|}\hat{q}_{\varepsilon,i}\hat{q}_{\varepsilon,j}\sqrt{\det(K(z_i))}\sqrt{\det(K(z_j))}\\
&-\sum_{j=1}^N\frac{\pi\varepsilon^2|\ln\varepsilon|^2}{2 |\ln s_{\varepsilon,j}|^2}\hat{q}_{\varepsilon,j}^2\sqrt{\det(K(z_j))}+O\left( \varepsilon^{2+\gamma} \right).
\end{split}
\end{equation}
Taking \eqref{3-003}, \eqref{q_i choice} and \eqref{2000} into \eqref{225}, we get
\eqref{344}. By \eqref{344} and Proposition \ref{pro401}, we have \eqref{345}.


\end{proof}
\section{Proof of Theorem \ref{thm0}}

Having made all preparations, we are now ready to prove Theorem \ref{thm0} in this section. Assume that  $ \nabla \left (q^2\sqrt{\det K}\right )(0)=0 $ and $ \left( \tilde{z}_1^*,\cdots, \tilde{z}_N^*\right)  $ is a local maximizer or minimizer of $ H_N $ defined by \eqref{def of H_N}. We will find critical points of $ K_\varepsilon(Z) $  of the form $ Z=\left (\frac{\tilde{z}_1}{\sqrt{|\ln\varepsilon|}},\frac{\tilde{z}_2}{\sqrt{|\ln\varepsilon|}},\cdots,\frac{\tilde{z}_N}{\sqrt{|\ln\varepsilon|}}\right ) $ with $ \tilde{z}_i=\tilde{z}_i^*+o(1) $. Using Proposition \ref{order of main term}, Lemma \ref{Green expansion} and Taylor's expansion, we have
\begin{equation*}\label{601}
\begin{split}
K_\varepsilon&\left (\frac{\tilde{z}_1}{\sqrt{|\ln\varepsilon|}},\frac{\tilde{z}_2}{\sqrt{|\ln\varepsilon|}},\cdots,\frac{\tilde{z}_N}{\sqrt{|\ln\varepsilon|}}\right )\\
=&\sum_{j=1}^N\pi\varepsilon^2|\ln\varepsilon|\left(\left( q^2\sqrt{\det K}\right)(0)+ \frac{1}{2|\ln\varepsilon|}\tilde{z}_j\cdot \nabla^2\left(q^2\sqrt{\det K} \right)(0 )\cdot\tilde{z}_j^T+O\left( \frac{1}{|\ln\varepsilon|^{\frac{3}{2}}}\right)    \right) \\
&+\sum_{j=1}^N\frac{(p-1)}{4}\pi\varepsilon^2\left( \left( q^2\sqrt{\det K}\right) (0)+O\left( \frac{1}{|\ln\varepsilon|^{\frac{1}{2}}}\right) \right) \\
&-\sum_{j=1}^N2\pi^2\varepsilon^2 \left( \left( q^2  \det K\right)(0) \bar{S}_K(0,0)+O\left( \frac{1}{|\ln\varepsilon|^{\frac{1}{2}}}\right)\right) \\
&-\sum_{1\leq i\neq j\leq N}2\pi^2\varepsilon^2\bigg( \left( q^2  \sqrt{\det K}\right)(0)\cdot-\frac{1}{2\pi}\ln\bigg|\frac{K(0)^{-\frac{1}{2}}\left( \tilde{z}_i-\tilde{z}_j\right) }{\sqrt{|\ln\varepsilon|}}\bigg|+\left( q^2  \det K\right)(0) \bar{S}_K(0,0)\\
&+O\left( \frac{\ln|\ln\varepsilon|}{|\ln\varepsilon|^{\frac{1}{2}}}\right)\bigg) +O\left( \frac{\varepsilon^2\ln|\ln\varepsilon|}{|\ln\varepsilon|}\right),
\end{split}
\end{equation*}
from which we deduce that
\begin{equation}\label{602}
\begin{split}
K_\varepsilon&\left (\frac{\tilde{z}_1}{\sqrt{|\ln\varepsilon|}},\frac{\tilde{z}_2}{\sqrt{|\ln\varepsilon|}},\cdots,\frac{\tilde{z}_N}{\sqrt{|\ln\varepsilon|}}\right )\\
=&N\pi\varepsilon^2|\ln\varepsilon|\left( q^2\sqrt{\det K}\right)(0)-\frac{N(N-1)\pi\varepsilon^2}{2}\left( q^2 \sqrt{\det K} \right)(0)\ln|\ln\varepsilon|\\
&+\pi\varepsilon^2\bigg(  \frac{1}{2}\sum_{j=1}^N\tilde{z}_j\cdot \nabla^2\left(q^2\sqrt{\det K} \right)(0 )\cdot\tilde{z}_j^T+\sum_{1\leq i\neq j\leq N}\left( q^2  \sqrt{\det K}\right)(0)\ln| K(0)^{-\frac{1}{2}}\left( \tilde{z}_i - \tilde{z}_j\right)  |     \\
&-2N^2\pi   \left( q^2  \det K\right)(0) \bar{S}_K(0,0)+ \frac{N(p-1)}{4}   \left( q^2\sqrt{\det K}\right) (0)   \bigg)
+O\left( \frac{\varepsilon^2\ln|\ln\varepsilon|}{|\ln\varepsilon|^{\frac{1}{2}}}\right).
\end{split}
\end{equation}

Note that 
\begin{equation*} 
\begin{split}
H_N(\tilde{z}_1,\cdots, \tilde{z}_N)=&\frac{1}{2}\sum_{j=1}^N\tilde{z}_j\cdot \nabla^2\left(q^2\sqrt{\det K} \right)(0 )\cdot\tilde{z}_j^T\\
&+\sum_{1\leq i\neq j\leq N}\left( q^2  \sqrt{\det K}\right)(0)\ln| K(0)^{-\frac{1}{2}}\left( \tilde{z}_i - \tilde{z}_j\right)  |.
\end{split}
\end{equation*}

From assumptions, if $ \left( \tilde{z}_1^*,\cdots, \tilde{z}_N^*\right)  $ is a local maximizer or minimizer of $ H_N $, then using perturbation theory we have 
\begin{proposition}\label{prop6-1}
	There exists $$ Z_\varepsilon=\left( \tilde{z}_{1,\varepsilon},\cdots, \tilde{z}_{N,\varepsilon}\right)=\left( \tilde{z}_1^*,\cdots, \tilde{z}_N^*\right)+o(1)$$ such that $ \left( \frac{\tilde{z}_{1,\varepsilon}}{\sqrt{|\ln\varepsilon|}},\cdots, \frac{\tilde{z}_{N,\varepsilon}}{\sqrt{|\ln\varepsilon|}}\right) $ is a critical point of $ K_\varepsilon(Z).$ As a consequence,  \eqref{key equa} has a solution $ v_\varepsilon= V_{\varepsilon,Z_\varepsilon}+\omega_{\varepsilon,Z_\varepsilon} $, where $ V_{\varepsilon,Z_\varepsilon} $ satisfies \eqref{solu config} and $ \omega_{\varepsilon,Z_\varepsilon} $ is determined by Proposition \ref{exist and uniq of w}.
\end{proposition}
\begin{proof}
Using \eqref{602} and the former deduction, we can easily get the result.
\end{proof}

We can also estimate     $ L^1 $ norm of $ v_\varepsilon $, which corresponds to the circulation of each component of vorticity fields in our construction in subsequent sections.
\begin{lemma}\label{circulation}
For $ i=1,\cdots,N $, we have
	\begin{equation*}
	\lim_{\varepsilon\to 0}\frac{1}{\varepsilon^2}\int_{B_{|\ln\varepsilon|^{-1}}\left (\frac{\tilde{z}_{i,\varepsilon}}{\sqrt{|\ln\varepsilon|}}\right )}\left (v_\varepsilon-q|\ln \varepsilon|\right )^{p}_+dx=2\pi \left( q\sqrt{\det K}\right) (0).
	\end{equation*}
\end{lemma}
\begin{proof}
	
	It follows from  \eqref{q_i choice}, \eqref{2000}, \eqref{203} and Proposition \ref{exist and uniq of w} that
	\begin{equation*}
	\begin{split}
	\frac{1}{\varepsilon^2}&\int_{B_{|\ln\varepsilon|^{-1}}\left (\frac{\tilde{z}_{i,\varepsilon}}{\sqrt{|\ln\varepsilon|}}\right )}\left (v_\varepsilon-q|\ln \varepsilon|\right )^{p}_+dx
	\\
	&=\frac{1}{\varepsilon^2}\int_{B_{Ls_{\varepsilon,i}}\left (\frac{\tilde{z}_{i,\varepsilon}}{\sqrt{|\ln\varepsilon|}}\right )}\left(V_{\varepsilon, z_{i,\varepsilon}, \hat{q}_{\varepsilon,i}}(x)-\hat{q}_{\varepsilon,i}|\ln\varepsilon|+O\left( \varepsilon^\gamma\right)\right)^{p}_+dx\\
	&=\frac{s_\varepsilon^2}{\varepsilon^2}\left( \frac{\hat{q}_{\varepsilon,i}}{|\phi'(1)|}\right) ^{p}\left( \frac{|\ln\varepsilon|}{|\ln s_{\varepsilon,i}|}\right)^{p}\sqrt{\det K}\left (\frac{\tilde{z}_{i,\varepsilon}}{\sqrt{|\ln\varepsilon|}}\right )\cdot 2\pi|\phi'(1)|+O(\varepsilon^\gamma)\\
	&= \frac{2\pi\hat{q}_{\varepsilon,i} |\ln\varepsilon|}{|\ln s_{\varepsilon,i}|}\cdot \sqrt{\det K }\left (\frac{\tilde{z}_{i,\varepsilon}}{\sqrt{|\ln\varepsilon|}}\right )+O(\varepsilon^\gamma) \\
	&= 2\pi \left( q \sqrt{\det K }\right) (0)+O\left (\frac{1}{\sqrt{|\ln\varepsilon|}}\right ).
	\end{split}
	\end{equation*}
\end{proof}

\textbf{Proof of Theorem \ref{thm0}:}

The rest of properties of $ v_\varepsilon $ can be easily deduced from the decomposition of $ v_\varepsilon $ in \eqref{solu config}, Proposition \ref{prop6-1} and Lemma \ref{circulation}. The proof of Theorem \ref{thm0} is thus finished.

\section{Proof of Theorems \ref{thm01}, \ref{thm02}, \ref{thm03}, \ref{thm04}, \ref{thm05}}

The proof of Theorems \ref{thm01}-\ref{thm05} are similar to that of Theorem \ref{thm0} and we therefore just show the difference to keep this paper not too long.

\textbf{Proof of Theorem \ref{thm01}:}
 Let $ N\geq 2, R^*>0, r_*>0, h>0, \kappa>0 $. Consider  the following equations
\begin{equation}\label{equa 1}
\begin{cases}
-\varepsilon^2\text{div}(K_H(x)\nabla v)=  \left (v-\left(\frac{\alpha}{2}|x|^2+\beta \right)|\ln\varepsilon| \right )^{p}_+\textbf{1}_{D_{\frac{\rho_0}{\sqrt{|\ln\varepsilon|}}}\left (C_{\frac{r_*}{\sqrt{|\ln\varepsilon|}}}\right )},\ \ &\text{in}\  B_{R^*}(0),\\
v=0,\ \ &\text{on}\ \partial B_{R^*}(0).
\end{cases}
\end{equation}
Here  $ \rho_0>0 $ sufficiently small, $ \textbf{1}_{D} $ is Heaviside function of some set $ D $, $ C_{\frac{r_*}{\sqrt{|\ln\varepsilon|}}}=\left \{x\mid |x|=\frac{r_*}{\sqrt{|\ln\varepsilon|}}\right \} $ is the circle centered at the origin of radius $ \frac{r_*}{\sqrt{|\ln\varepsilon|}} $, and $ D_{\frac{\rho_0}{\sqrt{|\ln\varepsilon|}}}\left (C_{\frac{r_*}{\sqrt{|\ln\varepsilon|}}}\right )=\left \{x\mid |x|\in \left (\frac{r_*-\rho_0}{\sqrt{|\ln\varepsilon|}},\frac{r_*+\rho_0}{\sqrt{|\ln\varepsilon|}} \right )\right \} $ is the $ \frac{\rho_0}{\sqrt{|\ln\varepsilon|}} $-neighborhood of $ C_{\frac{r_*}{\sqrt{|\ln\varepsilon|}}} $.  The constants $ \alpha $ and $ \beta $ in \eqref{equa 1} are chosen by
\begin{equation*} 
\alpha=\frac{\kappa}{4\pi}\left( \frac{1}{h^2}-\frac{N-1}{r_*^2}\right),\ \ \ \ \beta=\frac{\kappa}{2\pi}.
\end{equation*}

So if we take  $ K(x)=K_H(x) $  and $ q(x)=\frac{\alpha}{2}|x|^2+\beta $ and $ \Omega=B_{R^*}(0) $ in \eqref{key equa}, we get \eqref{equa 1}. 
In this case we choose 
\begin{equation*} 
\begin{split}
\Lambda_{\varepsilon, N,1}=\bigg\{&Z=(z_1,\cdots, z_N)\in\left(  B_{R^*}(0)\right) ^{(N)}\mid z_j\in D_{\frac{\rho_0}{2\sqrt{|\ln\varepsilon|}}}\left (C_{\frac{r_*}{\sqrt{|\ln\varepsilon|}}}\right ),\ \  z_j=e^{\textbf{i}\frac{2\pi(j-1)}{N}}z_1.\bigg\}
\end{split}
\end{equation*}
Our goal is to construct solutions of \eqref{equa 1} being of the form $ V_{\varepsilon,Z}+\omega_{\varepsilon,Z} $, where $ Z\in \Lambda_{\varepsilon, N,1} $.
Following the proof of Theorem \ref{thm0}, we only need to calculate critical points of $ K_\varepsilon(Z) $ in \eqref{functional I}. Note that $  K_H$ and  $ q $ are rotational-invariant. Since $ z_j=e^{\textbf{i}\frac{2\pi(j-1)}{N}}z_1 $, one can prove that $ K_\varepsilon(Z) $ is a   function of $ z_1 $ 
and is rotational-invariant, i.e., $ K_\varepsilon(Z)=\tilde{K}_\varepsilon(|z_1|) $ for some $ \tilde{K}_\varepsilon $. Let $ z_1=\left( \frac{r}{\sqrt{|\ln\varepsilon|}},0\right)  $ for $ r\in \left( r_*-\frac{\rho_0}{2},r_*+\frac{\rho_0}{2} \right)  $. 
Since   $ \nabla \left (q^2\sqrt{\det K_H}\right )(0)=0$, using \eqref{602}, one computes that
\begin{equation} \label{701}
\begin{split}
\tilde{K}_\varepsilon\left (\frac{r}{\sqrt{|\ln\varepsilon|}}\right )
=&N\pi\varepsilon^2|\ln\varepsilon|\left( q^2\sqrt{\det K_H}\right)(0)-\frac{N(N-1)\pi\varepsilon^2}{2}\left( q^2\det K_H\right)(0)\ln|\ln\varepsilon|\\
&+\pi\varepsilon^2\bigg(  \frac{N}{2} \left(q^2\sqrt{\det K_H} \right)^{''}(0 )r^2+ \frac{N(p-1)}{4}   \left( q^2\sqrt{\det K_H}\right) (0)      \\
&-2N\pi   \left( q^2  \det K_H\right)(0) \bar{S}_K(0,0)+N(N-1)\left( q^2  \det K_H\right)(0)\ln r\\
&+\sum_{1\leq j\neq k\leq N}\left( q^2  \det K_H\right)(0)\ln| e^{\textbf{i}\frac{2\pi(j-1)}{N}} - e^{\textbf{i}\frac{2\pi(k-1)}{N}} |\bigg)  
+O\left( \frac{\varepsilon^2\ln|\ln\varepsilon|}{|\ln\varepsilon|^{\frac{1}{2}}}\right).
\end{split}
\end{equation}
 So the existence of critical points of  $ \tilde{K}_\varepsilon $ is guaranteed by the existence of critical points of the function 
$$H_1(r)=\frac{N}{2} \left(q^2\sqrt{\det K_H} \right)^{''}(0 )r^2+N(N-1)\left( q^2  \det K_H\right)(0)\ln r.$$
We have
\begin{lemma}\label{lm701}
The function 
$ H_1(r) $
has the unqiue critical point $ r=r_* $ in $  \left( r_*-\frac{\rho_0}{2},r_*+\frac{\rho_0}{2} \right) $, which is a maximizer.
\end{lemma}  
\begin{proof}
We can prove that $  \left(q^2\sqrt{\det K_H} \right)^{''}(0 )=\frac{\beta(2\alpha h^2-\beta)}{h^2} $, $ \left( q^2  \det K_H\right)(0)=\beta^2 $ and $ H'_1(r)=N\left(\frac{\beta(2\alpha h^2-\beta)}{h^2} r+\frac{(N-1)\beta^2}{r} \right)  $. Taking $ \alpha,\beta $ into $ H'_1 $, we show that $ r_* $ is the unique maximizer.
\end{proof}

Note that from \eqref{701},
\begin{equation*}  
\begin{split}
K_\varepsilon\left (Z\right )
=&N\pi\varepsilon^2|\ln\varepsilon|\left( q^2\sqrt{\det K_H}\right)(0)-\frac{N(N-1)\pi\varepsilon^2}{2}\left( q^2\det K_H\right)(0)\ln|\ln\varepsilon|\\
&+\pi\varepsilon^2\bigg(  H_1(r)+ \frac{N(p-1)}{4}   \left( q^2\sqrt{\det K_H}\right) (0)      -2N\pi   \left( q^2  \det K_H\right)(0) \bar{S}_K(0,0)\\
&+\sum_{1\leq j\neq k\leq N}\left( q^2  \det K_H\right)(0)\ln| e^{\textbf{i}\frac{2\pi(j-1)}{N}} - e^{\textbf{i}\frac{2\pi(k-1)}{N}} |\bigg)  
+\tilde{N}_\varepsilon(z_1),
\end{split}
\end{equation*}
where $ \tilde{N}_\varepsilon(z_1) $ is a $O\left( \frac{\varepsilon^2\ln|\ln\varepsilon|}{|\ln\varepsilon|^{\frac{1}{2}}}\right) -$perturbation term which is invariant under rotation. 
By Lemma \ref{lm701}, there exists $ z_{1,\varepsilon}=\left( \frac{r_\varepsilon}{\sqrt{|\ln\varepsilon|}},0\right)  $ near $ \left( \frac{r_*}{\sqrt{|\ln\varepsilon|}},0\right) $ being a maximizer of $ \tilde{K}_\varepsilon  $.  Let $ z_{j,\varepsilon}=e^{\textbf{i}\frac{2\pi(j-1)}{N}}z_{1,\varepsilon} $ for $ j=1,\cdots,N. $ Then $ \left( z_{1,\varepsilon},\cdots,z_{N,\varepsilon} \right) $ is a maximizer of $ K_\varepsilon(Z) $, which yields a solution $ v_\varepsilon $ of \eqref{equa 1}. Similar to Theorem \ref{thm0}, 
\begin{equation*}
\left \{v_\varepsilon>q|\ln\varepsilon|\right \}\subseteq \cup_{i=1}^N B_{|\ln\varepsilon|^{-1}}\left (z_{i,\varepsilon}\right ),
\end{equation*}
and if we define $ A_{\varepsilon,i}=\left\{v_\varepsilon>q|\ln\varepsilon|\right\}\cap B_{|\ln\varepsilon|^{-1}}\left (z_{i,\varepsilon}\right )  $, then there are $ R_1,R_2>0 $ independent of $ \varepsilon $  such that $ B_{R_1\varepsilon}\left ( z_{i,\varepsilon}\right )\subseteq A_{\varepsilon,i}\subseteq B_{R_2\varepsilon}\left (z_{i,\varepsilon}\right ) $ and 
\begin{equation*}
\frac{1}{\varepsilon^2}\int_{A_{\varepsilon,i}}\left( v_\varepsilon-q|\ln\varepsilon|\right)^{p}_+dx=2\pi \left( q\sqrt{\det K_H }\right) (0)+O\left(|\ln\varepsilon|^{-\frac{1}{2}} \right) = \kappa+O\left(|\ln\varepsilon|^{-\frac{1}{2}} \right).
\end{equation*}

Define for any $ (x_1,x_2,x_3)\in B_{R^*}(0)\times\mathbb{R}, t\in\mathbb{R}  $
\begin{equation*}
\mathbf{w}_\varepsilon(x_1,x_2,x_3,t)= \frac{w_\varepsilon(x_1,x_2,x_3,t)}{h}(x_2,-x_1,h),
\end{equation*}
where $ w_\varepsilon(x_1,x_2,x_3,t) $ is a helical function satisfying
\begin{equation*} 
w_\varepsilon(x_1,x_2,0,t)=\frac{1}{\varepsilon^2}\left( v_\varepsilon(\bar{R}_{-\alpha|\ln\varepsilon| t}(x_1,x_2))-\left( \frac{\alpha|(x_1,x_2)|^2}{2}+\beta\right)|\ln \varepsilon|\right)^p_+\textbf{1}_{D_{\frac{\rho_0}{\sqrt{|\ln\varepsilon|}}}\left (C_{\frac{r_*}{\sqrt{|\ln\varepsilon|}}}\right )}.
\end{equation*}
Then $ w_\varepsilon(x_1,x_2,0,t) $ satisfies \eqref{vor str eq} and $ \mathbf{w}_\varepsilon $ is a  helical vorticity field of \eqref{Euler eq2}.  
The circulation of $ \mathbf{w}_\varepsilon $ is
\begin{equation*}
\iint_{A_{\varepsilon,i}}\mathbf{w}_\varepsilon\cdot\mathbf{n}d\sigma=\frac{1}{\varepsilon^2} \int_{A_{\varepsilon,i}}\left( v_\varepsilon-\left( \frac{\alpha|x|^2}{2}+\beta\right)|\ln \varepsilon|\right)^p_+dx=\kappa+O\left(|\ln\varepsilon|^{-\frac{1}{2}} \right).
\end{equation*}

Recall $ \tilde{X}_{j,\varepsilon}(s, t)=\frac{r_*}{\sqrt{|\ln\varepsilon|}}e^{-\mathbf{i}\alpha|\ln\varepsilon|t}e^{\mathbf{i}\frac{s}{h}}e^{\mathbf{i}\frac{2\pi (j-1)}{N}} $ for $ j=1,\cdots,N $. Based on the above derivation,   $ \mathbf{w}_\varepsilon $ tends asymptotically to $ \cup_{j=1}^N\tilde{X}_{j,\varepsilon} $ in sense of both topology and distribution. 
The proof of Theorem \ref{thm01} is complete.

\textbf{Proof of Theorem \ref{thm02}:}
 Let $ N\geq 2, R^*>0, r_*>0, h>0, \kappa>0, \mu>0 $. Consider  the following problem
\begin{equation}\label{equa 2}
\begin{cases}
-\varepsilon^2\text{div}(K_H(x)\nabla v)=&\left (v-\left(\frac{\alpha}{2}|x|^2+\beta_1 \right)|\ln\varepsilon| \right )^{p}_+\textbf{1}_{D_{\frac{\rho_0}{\sqrt{|\ln\varepsilon|}}}\left (C_{\frac{r_*}{\sqrt{|\ln\varepsilon|}}}\right )}\\
&+\left (v-\left(\frac{\alpha}{2}|x|^2+\beta_2 \right)|\ln\varepsilon| \right )^{p}_+\textbf{1}_{B_{\frac{\rho_0}{\sqrt{|\ln\varepsilon|}}}\left (0\right )},\ \ \text{in}\  B_{R^*}(0),\\
v=0,&\ \ \ \ \ \ \ \ \ \ \ \ \ \ \ \ \ \ \ \ \ \ \ \ \ \ \ \ \ \ \ \ \ \ \ \ \ \ \ \ \ \ \ \ \ \ \ \ \ \ \ \ \ \text{on}\ \partial B_{R^*}(0),
\end{cases}
\end{equation}
where  $ \rho_0>0 $ sufficiently small, 
 $ \alpha $ and $ \beta_i $ in \eqref{equa 2} are chosen by
\begin{equation*}
\alpha=\frac{\kappa}{4\pi}\left( \frac{1}{h^2}-\frac{N-1}{r_*^2}\right)-\frac{\mu}{2\pi r_*^2},\ \ \ \ \beta_1=\frac{\kappa}{2\pi}, \ \ \ \ \beta_2=\frac{\mu}{2\pi}.
\end{equation*}
Denote $ q_1(x)=\frac{\alpha}{2}|x|^2+\beta_1 $ and $ q_2(x)=\frac{\alpha}{2}|x|^2+\beta_2. $

We choose 
\begin{equation*} 
\begin{split}
\Lambda_{\varepsilon, N+1}=\bigg\{&Z=(z_0, z_1,\cdots, z_N) \mid z_0\in B_{\frac{\rho_0}{2\sqrt{|\ln\varepsilon|}}}\left (0\right ), z_j\in D_{\frac{\rho_0}{2\sqrt{|\ln\varepsilon|}}}\left (C_{\frac{r_*}{\sqrt{|\ln\varepsilon|}}}\right ),\   z_j=e^{\textbf{i}\frac{2\pi(j-1)}{N}}z_1\bigg\}.
\end{split}
\end{equation*}
Similar to the proof of Theorem \ref{thm0}, we find critical points of $ K_\varepsilon(Z) $ over $ \Lambda_{\varepsilon, N+1} $. Since $ K_\varepsilon\left (z_0, z_1,\cdots, z_N\right )=K_\varepsilon\left (e^{\textbf{i}\frac{2\pi(j-1)}{N}}z_0, z_1,\cdots, z_N\right ) $ for $ j=1,\cdots,N $, we have $\nabla K_\varepsilon(0, z_1,\cdots, z_N)=0$ for any $ (z_1,\cdots,z_N) $.  Since $ z_j=e^{\textbf{i}\frac{2\pi(j-1)}{N}}z_1 $,   $ K_\varepsilon(0,z_1,\cdots,z_N) $ is a   function of $ z_1 $ 
and is rotational-invariant, i.e., $ K_\varepsilon(0,z_1,\cdots,z_N)=\tilde{K}_\varepsilon(|z_1|) $ for some $ \tilde{K}_\varepsilon $. Let $ z_1=\left( \frac{r}{\sqrt{|\ln\varepsilon|}},0\right)  $ for $ r\in \left (r_*-\frac{\rho_0}{2},r_*+\frac{\rho_0}{2} \right ) $. 
Using \eqref{602}, one computes that
\begin{equation} \label{702}
\begin{split}
\tilde{K}_\varepsilon&\left (\frac{r}{\sqrt{|\ln\varepsilon|}}\right )\\
=&\pi\varepsilon^2|\ln\varepsilon|\left( N\left( q_1^2\sqrt{\det K_H}\right)(0)+\left( q_2^2\sqrt{\det K_H}\right)(0)\right) \\
&-\frac{N(N-1)\pi\varepsilon^2}{2}\left( q_1^2 \det K_H \right)(0)\ln|\ln\varepsilon|-N\pi\varepsilon^2\left( q_1q_2 \det K_H \right)(0)\ln|\ln\varepsilon|\\
&+\pi\varepsilon^2\bigg(  \frac{N}{2} \left(q_1^2\sqrt{\det K_H} \right)^{''}(0 )r^2+ \frac{N(p-1)}{4}   \left( q_1^2\sqrt{\det K_H}\right) (0)+ \frac{(p-1)}{4}   \left( q_2^2\sqrt{\det K_H}\right) (0)     \\
&-2N\pi   \left( q_1^2  \det K_H\right)(0) \bar{S}_K(0,0)-2\pi   \left( q_2^2  \det K_H\right)(0) \bar{S}_K(0,0)+N(N-1)\left( q_1^2  \det K_H\right)(0)\ln r\\
&+2N\left( q_1q_2 \det K_H\right)(0)\ln r+\sum_{1\leq j\neq k\leq N}\left( q_1^2  \det K_H\right)(0)\ln| e^{\textbf{i}\frac{2\pi(j-1)}{N}} - e^{\textbf{i}\frac{2\pi(k-1)}{N}} |\bigg)  
+O\left( \frac{\varepsilon^2\ln|\ln\varepsilon|}{|\ln\varepsilon|^{\frac{1}{2}}}\right).
\end{split}
\end{equation}
Define 
$$H_2(r)=\frac{N}{2} \left(q_1^2\sqrt{\det K_H} \right)^{''}(0 )r^2+N(N-1)\left( q_1^2  \det K_H\right)(0)\ln r+2N\left( q_1q_2 \det K_H\right)(0)\ln r.$$
We have
\begin{lemma}\label{lm702}
	The function 
	$ H_2(r) $
	has the unqiue critical point $ r=r_* $ in $  \left (r_*-\frac{\rho_0}{2},r_*+\frac{\rho_0}{2} \right ) $, which is a maximizer.
\end{lemma}  
\begin{proof}
Note that $  \left(q_1^2\sqrt{\det K_H} \right)^{''}(0 )=\frac{\beta_1(2\alpha h^2-\beta_1)}{h^2} $, $ \left( q_1q_i \det K_H\right)(0)=\beta_1\beta_i $ and $$ H'_2(r)=N\left(\frac{\beta_1(2\alpha h^2-\beta_1)}{h^2} r+\frac{(N-1)\beta_1^2+2\beta_1\beta_2}{r} \right).  $$ Taking $ \alpha,\beta_i $ into $ H'_2 $, we show that $ r_* $ is the unique maximizer.
\end{proof}
%
By Lemma \ref{lm702}, there exists $ z_{1,\varepsilon}=\left( \frac{r_\varepsilon}{\sqrt{|\ln\varepsilon|}},0\right)  $ near $ \left( \frac{r_*}{\sqrt{|\ln\varepsilon|}},0\right) $ being  a maximizer of $ \tilde{K}_\varepsilon  $. Let $ z_{0,\varepsilon}=0 $ and $ z_{j,\varepsilon}=e^{\textbf{i}\frac{2\pi(j-1)}{N}}z_{1,\varepsilon} $ for $ j=1,\cdots,N. $ Then $ \left( z_{0,\varepsilon},z_{1,\varepsilon},\cdots,z_{N,\varepsilon} \right)  $ is a critical point  of $ K_\varepsilon  $, which yields a solution $ v_\varepsilon $ of \eqref{equa 1}. Define 
\begin{equation*} 
\begin{split}
w_\varepsilon(x_1,x_2,0,t)=&\frac{1}{\varepsilon^2}\left (v_\varepsilon\left(\bar{R}_{-\alpha|\ln\varepsilon| t}(x_1,x_2) \right) -q_1(x_1,x_2)|\ln\varepsilon| \right )^{p}_+\textbf{1}_{D_{\frac{\rho_0}{\sqrt{|\ln\varepsilon|}}}\left (C_{\frac{r_*}{\sqrt{|\ln\varepsilon|}}}\right )}\\
&+\frac{1}{\varepsilon^2}\left (v_\varepsilon\left(\bar{R}_{-\alpha|\ln\varepsilon| t}(x_1,x_2) \right)-q_2(x_1,x_2)|\ln\varepsilon| \right )^{p}_+\textbf{1}_{B_{\frac{\rho_0}{\sqrt{|\ln\varepsilon|}}}\left (0\right )}.
\end{split}
\end{equation*}
Then $ supp(w_\varepsilon(\cdot,0,0))\subseteq \cup_{i=0}^N B_{|\ln\varepsilon|^{-1}}\left (z_{i,\varepsilon}\right ) $
and if we define $ A_{\varepsilon,i}=supp(w_\varepsilon(\cdot,0,0))\cap B_{|\ln\varepsilon|^{-1}}\left (z_{i,\varepsilon}\right )  $ for $ i=0,1,\cdots,N $, then there are $ R_1,R_2>0 $ independent of $ \varepsilon $  such that $ B_{R_1\varepsilon}\left ( z_{i,\varepsilon}\right )\subseteq A_{\varepsilon,i}\subseteq B_{R_2\varepsilon}\left (z_{i,\varepsilon}\right ) $ and  
\begin{equation*}
\int_{A_{\varepsilon,i}}w_\varepsilon dx=\begin{cases}
&2\pi \left( q_1\sqrt{\det K_H }\right) (0)+O\left(|\ln\varepsilon|^{-\frac{1}{2}} \right)=\kappa+O\left(|\ln\varepsilon|^{-\frac{1}{2}} \right)\ \ \ \ i=1,\cdots,N,\\
 &2\pi \left( q_2\sqrt{\det K_H }\right) (0)+O\left(|\ln\varepsilon|^{-\frac{1}{2}} \right)=\mu+O\left(|\ln\varepsilon|^{-\frac{1}{2}} \right)\ \ \ \ i=0.
\end{cases}
\end{equation*}

Recall $ 	 
\tilde{X}_{0,\varepsilon}(s,t)=0$ and $ \tilde{X}_{j,\varepsilon}(s,t)=\frac{r_*}{\sqrt{|\ln\varepsilon|}}e^{-\mathbf{i}\alpha|\ln\varepsilon|t}e^{\mathbf{i}\frac{s}{h}}e^{\mathbf{i}\frac{2\pi (j-1)}{N}} 
  $ for $ j=1,\cdots,N $. Based on the above derivation, we get Theorem \ref{thm02}.

\textbf{Proof of Theorem \ref{thm03}:}
Let $   R^*>0, h>0, $ $ \kappa_1,\kappa_2>0 $ and $ \lambda_{1,*},\lambda_{2,*}>0 $ are constants satisfying the compatibility condition
\begin{equation*} 
\frac{\kappa_2}{\kappa_1}=\frac{2\lambda_{1,*}h^2+\lambda_{1,*}\lambda_{2,*}(\lambda_{1,*}+\lambda_{2,*})}{2\lambda_{2,*}h^2+\lambda_{1,*}\lambda_{2,*}(\lambda_{1,*}+\lambda_{2,*})}.
\end{equation*}
Consider  the following problem
\begin{equation}\label{equa 3}
\begin{cases}
-\varepsilon^2\text{div}(K_H(x)\nabla v)=&\left (v-\left(\frac{\alpha}{2}|x|^2+\beta_1 \right)|\ln\varepsilon| \right )^{p}_+\textbf{1}_{B_{\frac{\rho_0}{\sqrt{|\ln\varepsilon|}}}\left (\frac{(\lambda_{1,*},0)}{\sqrt{|\ln\varepsilon|}}\right )}\\
&+\left (v-\left(\frac{\alpha}{2}|x|^2+\beta_2 \right)|\ln\varepsilon| \right )^{p}_+\textbf{1}_{B_{\frac{\rho_0}{\sqrt{|\ln\varepsilon|}}}\left (\frac{(-\lambda_{2,*},0)}{\sqrt{|\ln\varepsilon|}}\right )},\ \ \text{in}\  B_{R^*}(0),\\
v=0,&\ \ \ \ \ \ \ \ \ \ \ \ \ \ \ \ \ \ \ \ \ \ \ \ \ \ \ \ \ \ \ \ \ \ \ \ \ \ \ \ \ \ \ \ \ \ \ \ \ \ \ \ \ \ \ \ \ \ \ \ \text{on}\ \partial B_{R^*}(0),
\end{cases}
\end{equation}
where  $ \rho_0>0 $ sufficiently small, 
$ \alpha $ and $ \beta_i $ in \eqref{equa 3} are chosen by
\begin{equation*}
\alpha=\frac{\kappa_1}{4\pi h^2}  -\frac{\kappa_2}{2\pi \lambda_{1,*}(\lambda_{1,*}+\lambda_{2,*})}=\frac{\kappa_2}{4\pi h^2}  -\frac{\kappa_1}{2\pi \lambda_{2,*}(\lambda_{1,*}+\lambda_{2,*})},\ \ \ \ \beta_i=\frac{\kappa_i}{2\pi}.
\end{equation*}
Denote $ q_1(x)=\frac{\alpha}{2}|x|^2+\beta_1 $ and $ q_2(x)=\frac{\alpha}{2}|x|^2+\beta_2. $

We choose 
\begin{equation*} 
\begin{split}
\Lambda_{\varepsilon, 2}=\bigg\{Z=(  z_1,  z_2) \mid &z_1\in B_{\frac{\rho_0}{2\sqrt{|\ln\varepsilon|}}}\left (\frac{(\lambda_{1,*},0)}{\sqrt{|\ln\varepsilon|}}\right ), z_2\in B_{\frac{\rho_0}{2\sqrt{|\ln\varepsilon|}}}\left (\frac{(-\lambda_{2,*},0)}{\sqrt{|\ln\varepsilon|}}\right )\bigg\}.
\end{split}
\end{equation*}
We will find critical points of $ K_\varepsilon(Z) $ over $ \Lambda_{\varepsilon, 2} $. By rotational-invariance of $ K_\varepsilon(z_1,z_2) $, we assume that $ z_1 $ is on $ x_1 $-axis. Define $ \overline{z_2} $ as the symmetry point of $ z_2 $ with respect to the $ x_1 $-axis. Since $ K_\varepsilon(z_1,z_2)=K_\varepsilon(z_1,\overline{z_2}) $, by  principle of symmetric criticality (see \cite{W}), we  can assume that $ z_2 $ is on $ x_1 $-axis. 
Let $ z_1=\left( \frac{\lambda_1}{\sqrt{|\ln\varepsilon|}},0\right)  $ and $ z_2=\left( \frac{-\lambda_2}{\sqrt{|\ln\varepsilon|}},0\right)  $ for $\lambda_i\in \left( \lambda_{i,*}-\frac{\rho_0}{2}, \lambda_{i,*}+\frac{\rho_0}{2}\right)   $. 
Using \eqref{602}, one computes that
\begin{equation} \label{703}
\begin{split}
K_\varepsilon&\left (z_1,z_2\right )\\
=&\pi\varepsilon^2|\ln\varepsilon|\sum_{i=1}^2\left( q_i^2\sqrt{\det K_H}\right) (0) -\pi\varepsilon^2\left( q_1q_2 \det K_H \right)(0)\ln|\ln\varepsilon|\\
&+\pi\varepsilon^2\bigg(  \frac{1}{2}\sum_{i=1}^2 \left(q_i^2\sqrt{\det K_H} \right)^{''}(0 )\lambda_i^2+ \frac{(p-1)}{4}   \sum_{i=1}^2\left( q_i^2\sqrt{\det K_H}\right) (0)     \\
&-2\pi  \sum_{i=1}^2 \left( q_i^2  \det K_H\right)(0) \bar{S}_K(0,0)+2\left( q_1q_2  \det K_H\right)(0)\ln (\lambda_1+\lambda_2)\bigg)  
+O\left( \frac{\varepsilon^2\ln|\ln\varepsilon|}{|\ln\varepsilon|^{\frac{1}{2}}}\right).
\end{split}
\end{equation}
Define 
$$H_3(\lambda_1,\lambda_2)=\frac{1}{2}\sum_{i=1}^2 \left(q_i^2\sqrt{\det K_H} \right)^{''}(0 )\lambda_i^2+2\left( q_1q_2 \det K_H\right)(0)\ln (\lambda_1+\lambda_2).$$
We have
\begin{lemma}\label{lm703}
	The function 
	$ H_3(\lambda_1,\lambda_2) $
	has the unqiue critical point $ \lambda_1=\lambda_{1,*}, \lambda_2=\lambda_{2,*} $ in $  \left( \lambda_{1,*}-\frac{\rho_0}{2}, \lambda_{1,*}+\frac{\rho_0}{2}\right)\times \left( \lambda_{2,*}-\frac{\rho_0}{2}, \lambda_{2,*}+\frac{\rho_0}{2}\right) $, which is a maximizer.
\end{lemma}  
\begin{proof}
	Note that $  \left(q_i^2\sqrt{\det K_H} \right)^{''}(0 )=\frac{\beta_i(2\alpha h^2-\beta_i)}{h^2}=-\frac{\kappa_1,\kappa_2}{2\pi^2\lambda_{i,*}(\lambda_{1,*}+\lambda_{2,*})}<0 $, $ \left( q_1q_2 \det K_H\right)(0)=\beta_1\beta_2=\frac{\kappa_1\kappa_2}{4\pi^2} $ and $$ \nabla H_3(\lambda_1,\lambda_2)=\begin{pmatrix}
   \frac{\beta_1(2\alpha h^2-\beta_1)}{h^2} \lambda_1+\frac{2\beta_1\beta_2}{\lambda_1+\lambda_2}    \\
 \frac{\beta_2(2\alpha h^2-\beta_2)}{h^2} \lambda_2+\frac{2\beta_1\beta_2}{\lambda_1+\lambda_2} 
	\end{pmatrix}.  $$ 
So $ \nabla H_3=0 $ implies that $ \lambda_1=\lambda_{1,*}, \lambda_2=\lambda_{2,*} $. Since
$$\nabla^2 H_3(\lambda_{1,*},\lambda_{2,*})=\begin{pmatrix}
\frac{\beta_1(2\alpha h^2-\beta_1)}{h^2}  -\frac{2\beta_1\beta_2}{(\lambda_{1,*}+\lambda_{2,*})^2} &  -\frac{2\beta_1\beta_2}{(\lambda_{1,*}+\lambda_{2,*})^2}   \\
-\frac{2\beta_1\beta_2}{(\lambda_{1,*}+\lambda_{2,*})^2} & \frac{\beta_2(2\alpha h^2-\beta_2)}{h^2}  -\frac{2\beta_1\beta_2}{(\lambda_{1,*}+\lambda_{2,*})^2}
\end{pmatrix}$$
is negative definite, we conclude that $  (\lambda_{1,*},\lambda_{2,*}) $ is the unique maximizer.
\end{proof}
%
By Lemma \ref{lm703}, there exists $ z_{i,\varepsilon}=\left( \frac{(-1)^{i-1}\lambda_{i,\varepsilon}}{\sqrt{|\ln\varepsilon|}},0\right)  $ near $ \left( \frac{(-1)^{i-1}\lambda_{i,*}}{\sqrt{|\ln\varepsilon|}},0\right) $  for $ i=1,2 $ being a maximizer of $ K_\varepsilon  $, which yields a solution $ v_\varepsilon $ of \eqref{equa 3}.  Define 
\begin{equation*} 
\begin{split}
w(x_1,x_2,0,0)=&\frac{1}{\varepsilon^2}\sum_{i=1}^2\left (v_\varepsilon\left(x_1,x_2 \right) -q_i(x_1,x_2)|\ln\varepsilon| \right )^{p}_+\textbf{1}_{B_{\frac{\rho_0}{\sqrt{|\ln\varepsilon|}}}\left (\left( \frac{(-1)^{i-1}\lambda_{i,*}}{\sqrt{|\ln\varepsilon|}},0\right)\right )} 
\end{split}
\end{equation*}
and $ w(x_1,x_2,0,t)=w(\bar{R}_{-\alpha|\ln\varepsilon| t}(x_1,x_2),0,0)  $.
Then $ supp(w_\varepsilon(\cdot,0,0))\subseteq \cup_{i=1}^2 B_{|\ln\varepsilon|^{-1}}\left (z_{i,\varepsilon}\right ) $
and if we define $ A_{\varepsilon,i}=supp(w_\varepsilon(\cdot,0,0))\cap B_{|\ln\varepsilon|^{-1}}\left (z_{i,\varepsilon}\right )  $ for $ i=1,2 $, then 
\begin{equation*}
\int_{A_{\varepsilon,i}}w_\varepsilon dx=\begin{cases}
&\kappa_1+O\left(|\ln\varepsilon|^{-\frac{1}{2}} \right),\ \ \ \ i=1,\\
&\kappa_2+O\left(|\ln\varepsilon|^{-\frac{1}{2}} \right),\ \ \ \ i=2.
\end{cases}
\end{equation*}

Recall  $ \tilde{X}_{j,\varepsilon}(s,t)=\frac{(-1)^{j-1}\lambda_{j,*}}{\sqrt{|\ln\varepsilon|}}e^{-\mathbf{i}\alpha|\ln\varepsilon|t}e^{\mathbf{i}\frac{s}{h}} 
$ for $ j=1,2 $. Based on the above derivation, we get Theorem \ref{thm03}.

\textbf{Proof of Theorem \ref{thm04}:}
Let $   R^*>0, h>0, N=4, $ $ \kappa,\mu>0 $ and $ \lambda_{1,*},\lambda_{2,*}>0 $ are constants satisfying the compatibility condition
\begin{equation*} 
\frac{\kappa}{4\pi h^2}-\frac{\kappa}{4\pi\lambda_{1,*}^2}-\frac{\mu}{\pi(\lambda_{1,*}^2+\lambda_{2,*}^2)}=\frac{\mu}{4\pi h^2}-\frac{\mu}{4\pi\lambda_{2,*}^2}-\frac{\kappa}{\pi(\lambda_{1,*}^2+\lambda_{2,*}^2)}.
\end{equation*}. 
Consider  the following problem
\begin{equation}\label{equa 4}
	\begin{cases}
		-\varepsilon^2\text{div}(K_H(x)\nabla v)=&\sum_{i=1}^2\left (v-\left(\frac{\alpha}{2}|x|^2+\beta_1 \right)|\ln\varepsilon| \right )^{p}_+\textbf{1}_{B_{\frac{\rho_0}{\sqrt{|\ln\varepsilon|}}}\left (\frac{((-1)^{i-1}\lambda_{1,*},0)}{\sqrt{|\ln\varepsilon|}}\right )}\\
		&+\sum_{i=1}^2\left (v-\left(\frac{\alpha}{2}|x|^2+\beta_2 \right)|\ln\varepsilon| \right )^{p}_+\textbf{1}_{B_{\frac{\rho_0}{\sqrt{|\ln\varepsilon|}}}\left (\frac{(0,(-1)^{i-1}\lambda_{2,*})}{\sqrt{|\ln\varepsilon|}}\right )},\ \ \text{in}\  B_{R^*}(0),\\
		v=0,&\ \ \ \ \ \ \ \ \ \ \ \ \ \ \ \ \ \ \ \ \ \ \ \ \ \ \ \ \ \ \ \ \ \ \ \ \ \ \ \ \ \ \ \ \ \ \ \ \ \ \ \ \ \ \ \ \ \ \ \ \text{on}\ \partial B_{R^*}(0),
	\end{cases}
\end{equation}
where  $ \rho_0>0 $ sufficiently small, 
$ \alpha $ and $ \beta_i $ in \eqref{equa 4} are chosen by
\begin{equation*}
	\alpha=\frac{\kappa}{4\pi h^2}-\frac{\kappa}{4\pi\lambda_{1,*}^2}-\frac{\mu}{\pi(\lambda_{1,*}^2+\lambda_{2,*}^2)},\ \ \ \ \beta_1=\frac{\kappa}{2\pi},\ \ \ \ \beta_2=\frac{\mu}{2\pi}.
\end{equation*}
Denote $ q_i(x)=\frac{\alpha}{2}|x|^2+\beta_i $ for $ i=1,2. $

We choose 
\begin{equation*} 
	\begin{split}
		\Lambda_{\varepsilon, 4}=\bigg\{Z=(  z_1,  z_2, z_3, z_4) \mid &z_1\in B_{\frac{\rho_0}{2\sqrt{|\ln\varepsilon|}}}\left (\frac{(\lambda_{1,*},0)}{\sqrt{|\ln\varepsilon|}}\right ), z_2\in B_{\frac{\rho_0}{2\sqrt{|\ln\varepsilon|}}}\left (\frac{(0,\lambda_{2,*})}{\sqrt{|\ln\varepsilon|}}\right ),\\
		&z_3=-z_1,z_4=-z_2\bigg\}.
	\end{split}
\end{equation*}
From the symmetry of $ K_\varepsilon $, we can assume $ z_1=\left( \frac{\lambda_1}{\sqrt{|\ln\varepsilon|}},0\right)  $ and $ z_2=\left( 0,\frac{\lambda_2}{\sqrt{|\ln\varepsilon|}}\right)  $ for $\lambda_i\in \left( \lambda_{i,*}-\frac{\rho_0}{2}, \lambda_{i,*}+\frac{\rho_0}{2}\right)   $. 
Similar to \eqref{602}, we have
\begin{equation} \label{704}
	\begin{split}
		K_\varepsilon&\left (z_1,z_2,z_3,z_4\right )\\
		=&2\pi\varepsilon^2|\ln\varepsilon|\sum_{i=1}^2\left( q_i^2\sqrt{\det K_H}\right) (0) -\pi\varepsilon^2\left( \left( q_1^2+4q_1q_2+q_2^2 \right) \det K_H \right)(0)\ln|\ln\varepsilon|\\
		&+\pi\varepsilon^2\bigg(    H_4(\lambda_1,\lambda_2)+\frac{(p-1)}{2}   \sum_{i=1}^2\left( q_i^2\sqrt{\det K_H}\right) (0)      -4\pi  \sum_{i=1}^2 \left( q_i^2  \det K_H\right)(0) \bar{S}_K(0,0)\bigg) \\ 
		&+O\left( \frac{\varepsilon^2\ln|\ln\varepsilon|}{|\ln\varepsilon|^{\frac{1}{2}}}\right).
	\end{split}
\end{equation}
where
\begin{equation*} 
\begin{split}
H_4(\lambda_1,\lambda_2)=&\sum_{i=1}^2 \left(q_i^2\sqrt{\det K_H} \right)^{''}(0 )\lambda_i^2+2\sum_{i=1}^2\left( q_i^2 \det K_H\right)(0)\ln 2\lambda_i\\
& +4\left( q_1q_2  \det K_H\right)(0)\ln (\lambda_1^2+\lambda_2^2) .
\end{split}
\end{equation*}
$$$$
It is not hard to verify that 	$ H_4(\lambda_1,\lambda_2) $
	has the unqiue maximizer $ \lambda_1=\lambda_{1,*}, \lambda_2=\lambda_{2,*} $ in $  \left( \lambda_{1,*}-\frac{\rho_0}{2}, \lambda_{1,*}+\frac{\rho_0}{2}\right)\times \left( \lambda_{2,*}-\frac{\rho_0}{2}, \lambda_{2,*}+\frac{\rho_0}{2}\right) $.   
%
Thus from \eqref{704}, there exists $ z_{2i-1,\varepsilon}=\left( \frac{(-1)^{i-1}\lambda_{1,\varepsilon}}{\sqrt{|\ln\varepsilon|}},0\right)  $ near $ \left( \frac{(-1)^{i-1}\lambda_{1,*}}{\sqrt{|\ln\varepsilon|}},0\right) $ and $ z_{2i,\varepsilon}=\left( 0,\frac{(-1)^{i-1}\lambda_{2,\varepsilon}}{\sqrt{|\ln\varepsilon|}}\right)  $ near $ \left( 0,\frac{(-1)^{i-1}\lambda_{2,*}}{\sqrt{|\ln\varepsilon|}}\right) $ for $ i=1,2 $ being a maximizer of $ K_\varepsilon  $, which yields a solution $ v_\varepsilon $ of \eqref{equa 4}.   Similar to Theorem \ref{thm03}, we get Theorem \ref{thm04}.

\textbf{Proof of Theorem \ref{thm05}:}
Let $   R^*>0, h>0, N=5, $ $ \kappa_0,\kappa,\mu>0 $ and $ \lambda_{1,*},\lambda_{2,*}>0 $ are constants satisfying the compatibility condition
\begin{equation*} 
\frac{\kappa}{4\pi h^2}-\frac{\kappa_0}{2\pi\lambda_{1,*}^2}-\frac{\kappa}{4\pi\lambda_{1,*}^2}-\frac{\mu}{\pi(\lambda_{1,*}^2+\lambda_{2,*}^2)}=\frac{\mu}{4\pi h^2}-\frac{\kappa_0}{2\pi\lambda_{2,*}^2}-\frac{\mu}{4\pi\lambda_{2,*}^2}-\frac{\kappa}{\pi(\lambda_{1,*}^2+\lambda_{2,*}^2)}.
\end{equation*}. 
Consider  the following problem
\begin{equation}\label{equa 5}
\begin{cases}
-\varepsilon^2\text{div}(K_H(x)\nabla v)=&\sum_{i=1}^2\left (v-\left(\frac{\alpha}{2}|x|^2+\beta_1 \right)|\ln\varepsilon| \right )^{p}_+\textbf{1}_{B_{\frac{\rho_0}{\sqrt{|\ln\varepsilon|}}}\left (\frac{((-1)^{i-1}\lambda_{1,*},0)}{\sqrt{|\ln\varepsilon|}}\right )}\\
&+\sum_{i=1}^2\left (v-\left(\frac{\alpha}{2}|x|^2+\beta_2 \right)|\ln\varepsilon| \right )^{p}_+\textbf{1}_{B_{\frac{\rho_0}{\sqrt{|\ln\varepsilon|}}}\left (\frac{(0,(-1)^{i-1}\lambda_{2,*})}{\sqrt{|\ln\varepsilon|}}\right )}\\
&+\left (v-\left(\frac{\alpha}{2}|x|^2+\beta_0 \right)|\ln\varepsilon| \right )^{p}_+\textbf{1}_{B_{\frac{\rho_0}{\sqrt{|\ln\varepsilon|}}}\left (0\right )},\ \ \ \ \ \ \ \ \ \text{in}\  B_{R^*}(0),\\
v=0,&\ \ \ \ \ \ \ \ \ \ \ \ \ \ \ \ \ \ \ \ \ \ \ \ \ \ \ \ \ \ \ \ \ \ \ \ \ \ \ \ \ \ \ \ \ \ \ \ \ \ \ \ \ \ \ \ \ \ \ \ \text{on}\ \partial B_{R^*}(0),
\end{cases}
\end{equation}
where  $ \rho_0>0 $ sufficiently small, 
$ \alpha $ and $ \beta_i $ in \eqref{equa 5} are chosen by
\begin{equation*}
\alpha=\frac{\kappa}{4\pi h^2}-\frac{\kappa_0}{2\pi\lambda_{1,*}^2}-\frac{\kappa}{4\pi\lambda_{1,*}^2}-\frac{\mu}{\pi(\lambda_{1,*}^2+\lambda_{2,*}^2)},\ \ \ \ \beta_0=\frac{\kappa_0}{2\pi},\ \ \ \ \beta_1=\frac{\kappa}{2\pi},\ \ \ \ \beta_2=\frac{\mu}{2\pi}.
\end{equation*}
Denote $ q_i(x)=\frac{\alpha}{2}|x|^2+\beta_i $ for $ i=0,1,2. $

We choose 
\begin{equation*} 
\begin{split}
\Lambda_{\varepsilon, 5}=\bigg\{Z=(z_0,  z_1,  z_2, z_3, z_4) \mid &z_1\in B_{\frac{\rho_0}{2\sqrt{|\ln\varepsilon|}}}\left (\frac{(\lambda_{1,*},0)}{\sqrt{|\ln\varepsilon|}}\right ), z_2\in B_{\frac{\rho_0}{2\sqrt{|\ln\varepsilon|}}}\left (\frac{(0,\lambda_{2,*})}{\sqrt{|\ln\varepsilon|}}\right ),\\
&z_3=-z_1,z_4=-z_2, z_0\in B_{\frac{\rho_0}{2\sqrt{|\ln\varepsilon|}}}\left (0\right )\bigg\}.
\end{split}
\end{equation*}
From the symmetry of $ K_\varepsilon $, we can assume $ z_0=0 $, $ z_1=\left( \frac{\lambda_1}{\sqrt{|\ln\varepsilon|}},0\right)  $ and $ z_2=\left( 0,\frac{\lambda_2}{\sqrt{|\ln\varepsilon|}}\right)  $ for $\lambda_i\in \left( \lambda_{i,*}-\frac{\rho_0}{2}, \lambda_{i,*}+\frac{\rho_0}{2}\right)   $. 
Similar to \eqref{602}, we have
\begin{equation} \label{705}
\begin{split}
K_\varepsilon&\left (z_0,z_1,z_2,z_3,z_4\right )\\
=&\pi\varepsilon^2|\ln\varepsilon|\left( 2\sum_{i=1}^2\left( q_i^2\sqrt{\det K_H}\right) (0)+\left( q_0^2\sqrt{\det K_H}\right) (0) \right) \\
&-\pi\varepsilon^2\left( \left( q_1^2+4q_1q_2+q_2^2+2q_1q_0+2q_2q_0 \right) \det K_H \right)(0)\ln|\ln\varepsilon|\\
&+\pi\varepsilon^2\bigg(    H_5(\lambda_1,\lambda_2)+\frac{(p-1)}{2}   \sum_{i=1}^2\left( q_i^2\sqrt{\det K_H}\right) (0)   + \frac{(p-1)}{4}\left( q_0^2\sqrt{\det K_H}\right) (0)  \\
& -4\pi  \sum_{i=1}^2 \left( q_i^2  \det K_H\right)(0) \bar{S}_K(0,0)-2\pi   \left( q_0^2  \det K_H\right)(0) \bar{S}_K(0,0)\bigg) \\ 
&+O\left( \frac{\varepsilon^2\ln|\ln\varepsilon|}{|\ln\varepsilon|^{\frac{1}{2}}}\right).
\end{split}
\end{equation}
where
\begin{equation*} 
\begin{split}
H_5(\lambda_1,\lambda_2)=&\sum_{i=1}^2 \left(q_i^2\sqrt{\det K_H} \right)^{''}(0 )\lambda_i^2+2\sum_{i=1}^2\left( q_i^2 \det K_H\right)(0)\ln 2\lambda_i \\
&+4\left( q_1q_2  \det K_H\right)(0)\ln (\lambda_1^2+\lambda_2^2)+4\sum_{i=1}^2\left( q_0q_i \det K_H\right)(0)\ln \lambda_i .
\end{split}
\end{equation*}
$$$$
Since	$ H_5(\lambda_1,\lambda_2) $
has the unqiue maximizer $ \lambda_1=\lambda_{1,*}, \lambda_2=\lambda_{2,*} $ in $  \left( \lambda_{1,*}-\frac{\rho_0}{2}, \lambda_{1,*}+\frac{\rho_0}{2}\right)\times \left( \lambda_{2,*}-\frac{\rho_0}{2}, \lambda_{2,*}+\frac{\rho_0}{2}\right) $,  
%
from \eqref{705}, there exists $ z_0=0 $, $ z_{2i-1,\varepsilon}=\left( \frac{(-1)^{i-1}\lambda_{1,\varepsilon}}{\sqrt{|\ln\varepsilon|}},0\right)  $ near $ \left( \frac{(-1)^{i-1}\lambda_{1,*}}{\sqrt{|\ln\varepsilon|}},0\right) $ and $ z_{2i,\varepsilon}=\left( 0,\frac{(-1)^{i-1}\lambda_{2,\varepsilon}}{\sqrt{|\ln\varepsilon|}}\right)  $ near $ \left( 0,\frac{(-1)^{i-1}\lambda_{2,*}}{\sqrt{|\ln\varepsilon|}}\right) $ for $ i=1,2 $ being a critical point of $ K_\varepsilon  $, which yields a solution $ v_\varepsilon $ of \eqref{equa 5}.   Similar to Theorem \ref{thm03}, we get Theorem \ref{thm05}.

\setcounter{equation}{0}\renewcommand\theequation{A.\arabic{equation}}
\section{Appendix: Proofs of some results used before}
In this Appendix, we give proofs for some results  which have been repeatedly used before. 

We first give some explicit helical solutions of  the nearly parallel vortex filament model deduced in \cite{KMD} (see also \cite{GM2}). 
We have  exact helical solutions of  \eqref{NP vor fi} as follows.

\begin{customthm}{A.1}  \label{lemA1}
Let $ h>0 $. The following point vortex configurations are co-rotating helical solutions of \eqref{NP vor fi}.
\begin{enumerate}
	\item[\textbf{Case 1:}]  ``$ N $ vortices'' having identical circulation and placed  at each vertex of a regular $ N $-polygon with radius $ r_* $, i.e., $ N\geq 2 $, $ \kappa_j=\kappa $ for some $ \kappa>0 $ and 
	\begin{equation*} 
	X_{j}(s,\tau)=r_*e^{-\mathbf{i}\alpha\tau}e^{\mathbf{i}\frac{s}{h}}e^{\mathbf{i}\frac{2\pi (j-1)}{N}},\ \ \ \ \ \ \ \ \  j=1,\cdots,N, 
	\end{equation*}
	where $ \alpha=\frac{\kappa}{4\pi}\left( \frac{1}{h^2}-\frac{N-1}{r_*^2}\right)  $.
	\item[\textbf{Case 2:}] ``$ N+1  $ vortices'', $ N $ of them having identical circulation and placed  at each vertex of a regular $ N $-polygon with radius $ r_* $ and the $ N +1 $ vortex of arbitrary strength
	at the center of vorticity, i.e., $ N\geq 2 $, $ \kappa_0=\mu $, $ \kappa_j=\kappa $ for $ j=1,\cdots,N $ for   $ \mu, \kappa>0 $  and
	\begin{equation*} 
	X_0(s,\tau)=0,\ \ \ \ X_{j}(s,\tau)=r_*e^{-\mathbf{i}\alpha\tau}e^{\mathbf{i}\frac{s}{h}}e^{\mathbf{i}\frac{2\pi (j-1)}{N}},\ \ \ \ \ \ \  j=1,\cdots,N,
	\end{equation*}
	where $ \alpha=\frac{\kappa}{4\pi}\left( \frac{1}{h^2}-\frac{N-1}{r_*^2}\right)-\frac{\mu}{2\pi r_*^2}  $.
		\item[\textbf{Case 3:}] ``Co-rotating asymmetric $ 2 $ vortices'' with different circulations and rotation radii. In this case, $ N= 2 $, $ \kappa_1,\kappa_2>0 $ and $ \lambda_{1,*},\lambda_{2,*}>0 $ are constants satisfying the compatibility condition
		\begin{equation*} 
		\frac{\kappa_2}{\kappa_1}=\frac{2\lambda_{1,*}h^2+\lambda_{1,*}\lambda_{2,*}(\lambda_{1,*}+\lambda_{2,*})}{2\lambda_{2,*}h^2+\lambda_{1,*}\lambda_{2,*}(\lambda_{1,*}+\lambda_{2,*})},
		\end{equation*}
		and
	\begin{equation*} 
	X_{1}(s,\tau)=\lambda_{1,*}e^{-\mathbf{i}\alpha\tau}e^{\mathbf{i}\frac{s}{h}},\ \ \ \ 	X_{2}(s,\tau)=-\lambda_{2,*}e^{-\mathbf{i}\alpha\tau}e^{\mathbf{i}\frac{s}{h}},
	\end{equation*}
	where $ \alpha=\frac{\kappa_1}{4\pi h^2}  -\frac{\kappa_2}{2\pi \lambda_{1,*}(\lambda_{1,*}+\lambda_{2,*})}=\frac{\kappa_2}{4\pi h^2}  -\frac{\kappa_1}{2\pi \lambda_{2,*}(\lambda_{1,*}+\lambda_{2,*})}  $.
\item[\textbf{Case 4:}] ``Co-rotating  $  2\times2  $ vortices'' with different circulations and rotation radii. In this case, $ N= 4 $, $ \kappa_1=\kappa_3=\kappa>0 $, $ \kappa_2=\kappa_4=\mu>0$. $ \kappa,\mu, \lambda_{1,*},\lambda_{2,*}  $ are constants satisfying the compatibility condition
\begin{equation*} 
\frac{\kappa}{4\pi h^2} -\frac{\kappa}{4\pi \lambda_{1,*}^2} -\frac{\mu}{\pi \left (\lambda_{1,*}^2+\lambda_{2,*}^2\right )}=\frac{\mu}{4\pi h^2} -\frac{\mu}{4\pi \lambda_{2,*}^2} -\frac{\kappa}{\pi \left (\lambda_{1,*}^2+\lambda_{2,*}^2\right )},
\end{equation*}
and
\begin{equation*} 
X_{2j-1}(s,\tau)=\lambda_{1,*}e^{-\mathbf{i}\alpha\tau}e^{\mathbf{i}\frac{s}{h}}e^{\mathbf{i}\pi(j-1)},\ \ \ \ 	X_{2j}(s,\tau)=\lambda_{2,*}e^{-\mathbf{i}\alpha\tau}e^{\mathbf{i}\frac{s}{h}}e^{\mathbf{i}\frac{\pi(2j-1)}{2}},\ \ j=1,2,
\end{equation*}
where $ \alpha=\frac{\kappa}{4\pi h^2} -\frac{\kappa}{4\pi \lambda_{1,*}^2} -\frac{\mu}{\pi \left (\lambda_{1,*}^2+\lambda_{2,*}^2\right )}=\frac{\mu}{4\pi h^2} -\frac{\mu}{4\pi \lambda_{2,*}^2} -\frac{\kappa}{\pi \left (\lambda_{1,*}^2+\lambda_{2,*}^2\right )}  $.
\item[\textbf{Case 5:}] ``Co-rotating  $  2\times2+1  $ vortices'' with different circulations and rotation radii and the $ 5 $-th vortex of arbitrary strength
at the center of vorticity. In this case, $ N= 5 $, $ \kappa_1=\kappa_3=\kappa>0 $, $ \kappa_2=\kappa_4=\mu>0$ and $ \kappa_0>0 $. $ \kappa_0,\kappa,\mu, \lambda_{1,*},\lambda_{2,*}  $ are constants satisfying the compatibility condition
\begin{equation*} 
\frac{\kappa}{4\pi h^2}-\frac{\kappa_0}{2\pi \lambda_{1,*}^2} -\frac{\kappa}{4\pi \lambda_{1,*}^2} -\frac{\mu}{\pi \left (\lambda_{1,*}^2+\lambda_{2,*}^2\right )}=\frac{\mu}{4\pi h^2}-\frac{\kappa_0}{2\pi \lambda_{2,*}^2} -\frac{\mu}{4\pi \lambda_{2,*}^2} -\frac{\kappa}{\pi \left (\lambda_{1,*}^2+\lambda_{2,*}^2\right )},
\end{equation*}
and $ X_0(s,\tau)=0 $,
\begin{equation*} 
X_{2j-1}(s,\tau)=\lambda_{1,*}e^{-\mathbf{i}\alpha\tau}e^{\mathbf{i}\frac{s}{h}}e^{\mathbf{i}\pi(j-1)},\ \ \ \ 	X_{2j}(s,\tau)=\lambda_{2,*}e^{-\mathbf{i}\alpha\tau}e^{\mathbf{i}\frac{s}{h}}e^{\mathbf{i}\frac{\pi(2j-1)}{2}},\ \ j=1,2,
\end{equation*}
where $ \alpha=\frac{\kappa}{4\pi h^2}-\frac{\kappa_0}{2\pi \lambda_{1,*}^2} -\frac{\kappa}{4\pi \lambda_{1,*}^2} -\frac{\mu}{\pi \left (\lambda_{1,*}^2+\lambda_{2,*}^2\right )}=\frac{\mu}{4\pi h^2}-\frac{\kappa_0}{2\pi \lambda_{2,*}^2} -\frac{\mu}{4\pi \lambda_{2,*}^2} -\frac{\kappa}{\pi \left (\lambda_{1,*}^2+\lambda_{2,*}^2\right )}  $.
\end{enumerate}
 
\end{customthm}
\begin{proof}
For case 1, taking \eqref{fila 1} into \eqref{NP heli vor}, we have $ \alpha=\frac{\kappa}{4\pi h^2}-\frac{\kappa}{2\pi r_*^2}\sum_{k\neq j}\frac{1-e^{\mathbf{i}\frac{2\pi(j-k)}{N}}}{|1-e^{\mathbf{i}\frac{2\pi(j-k)}{N}}|^2} $.  Note that

\begin{equation}\label{A-01}
\sum_{k\neq j}\frac{1-e^{\mathbf{i}\frac{2\pi(j-k)}{N}}}{|1-e^{\mathbf{i}\frac{2\pi(j-k)}{N}}|^2}=\frac{N-1}{2}.
\end{equation} 
Then $ \alpha=\frac{\kappa}{4\pi}\left( \frac{1}{h^2}-\frac{N-1}{r_*^2}\right). $

For case 2, taking \eqref{fila 2} into \eqref{NP heli vor}, we have $ \alpha=\frac{\kappa}{4\pi h^2}-\frac{\kappa}{2\pi r_*^2}\sum_{k\neq j}\frac{1-e^{\mathbf{i}\frac{2\pi(j-k)}{N}}}{|1-e^{\mathbf{i}\frac{2\pi(j-k)}{N}}|^2}-\frac{\mu}{2\pi r_*^2} $. Using \eqref{A-01}, we have $ \alpha=\frac{\kappa}{4\pi}\left( \frac{1}{h^2}-\frac{N-1}{r_*^2}\right)-\frac{\mu}{2\pi r_*^2}  $.

For case 3, taking \eqref{fila 3} into \eqref{NP heli vor}, we have $ \alpha=\frac{\kappa_1}{4\pi h^2}  -\frac{\kappa_2}{2\pi \lambda_{1,*}(\lambda_{1,*}+\lambda_{2,*})}=\frac{\kappa_2}{4\pi h^2}  -\frac{\kappa_1}{2\pi \lambda_{2,*}(\lambda_{1,*}+\lambda_{2,*})}  $. This implies that $ \kappa_1,\kappa_2,  \lambda_{1,*},\lambda_{2,*}  $  satisfies $ 	\frac{\kappa_2}{\kappa_1}=\frac{2\lambda_{1,*}h^2+\lambda_{1,*}\lambda_{2,*}(\lambda_{1,*}+\lambda_{2,*})}{2\lambda_{2,*}h^2+\lambda_{1,*}\lambda_{2,*}(\lambda_{1,*}+\lambda_{2,*})} $. So if this condition holds, \eqref{NP vor fi} has a solution being of the form \eqref{fila 3}.

For case 4, taking \eqref{fila 4} into \eqref{NP heli vor}, we have $ \alpha=\frac{\kappa}{4\pi h^2} -\frac{\kappa}{4\pi \lambda_{1,*}^2} -\frac{\mu}{\pi \left (\lambda_{1,*}^2+\lambda_{2,*}^2\right )}=\frac{\mu}{4\pi h^2} -\frac{\mu}{4\pi \lambda_{2,*}^2} -\frac{\kappa}{\pi \left (\lambda_{1,*}^2+\lambda_{2,*}^2\right )}  $. This implies that $ \kappa_1,\kappa_2,  \lambda_{1,*},\lambda_{2,*}  $  satisfies $\frac{\kappa}{4\pi h^2} -\frac{\kappa}{4\pi \lambda_{1,*}^2} -\frac{\mu}{\pi \left (\lambda_{1,*}^2+\lambda_{2,*}^2\right )}=\frac{\mu}{4\pi h^2} -\frac{\mu}{4\pi \lambda_{2,*}^2} -\frac{\kappa}{\pi \left (\lambda_{1,*}^2+\lambda_{2,*}^2\right )} $. Under this assumption, \eqref{NP vor fi} has a solution being of the form \eqref{fila 4}. The case 5 is similar to the case 4 and so we omit the proof here.  
\end{proof}

We now turn to the proof of Lemma \ref{Green expansion2}. 
\begin{customthm}{A.2}\label{lemA2}
There holds
\begin{equation*}
\bar{S}_K(x,y)=-F_{1,y}(x)-F_{2,y}(x)+\bar{H}_1(x,y) \ \ \ \ \forall\ x,y\in\Omega,
\end{equation*}
where $ F_{1,y}(\cdot) $ is defined by \eqref{exp of F_{1,y}} and $ F_{2,y}(\cdot) $ is defined by \eqref{exp of F_{2,y}},
and $ x\to \bar{H}_1(x,y)\in C^{1,\gamma}(\overline{\Omega})$ for all $ y\in \Omega $, $ \gamma\in (0, 1) $.  Moreover, the function $(x, y)\to \bar{H}_1(x,y)\in C^1(\Omega\times\Omega),$ and in particular the corresponding Robin function $x\to \bar{S}_K(x,x)\in C^1(\Omega)$.
\end{customthm}

\begin{proof}
	Let  $ y\in\Omega $ be fixed. In the following,  we always denote $ T_{ij}=(T_y)_{ij} $ the component of row $ i $, column $ j $ of the matrix $ T_y $ for $ i,j=1,2 $. From Lemma \ref{Green expansion}, the regular part $ \bar{S}_K(x,y) $ satisfies
	\begin{equation}\label{equa of S_K}
	\begin{cases}
	-\text{div}\left (K(x)\nabla \bar{S}_K(x,y)\right )=\text{div}\left (\left (K(x)-K(y)\right )\nabla \left( \sqrt{\det K(y)}^{-1}\Gamma\left (T_y(x-y)\right )\right)  \right )\ &\text{in}\ \Omega,\\
	\bar{S}_K(x,y)=-\sqrt{\det K(y)}^{-1}\Gamma\left (T_y(x-y)\right )\ &\text{on}\ \partial\Omega.
	\end{cases}
	\end{equation}
	This implies that
	\begin{equation}\label{3-1}
	\begin{split}
	-\text{div}\left (K(x)\nabla \bar{S}_K(x,y)\right )=&\sum_{i,j=1}^2\partial_{x_i}K_{ij}(x)\partial_{x_j}\left(\sqrt{\det K(y)}^{-1}\Gamma\left (T_y(x-y)\right ) \right) \\
	&+\sum_{i,j=1}^2\left( \left( K_{ij}(x)-K_{ij}(y)\right) \partial_{x_ix_j}\left(\sqrt{\det K(y)}^{-1}\Gamma\left (T_y(x-y)\right ) \right)\right)\\
	=&:A_1+A_2.
	\end{split}
	\end{equation}
	
	As for $ A_1 $, for $ x\in\mathbb{R}^2 $, we denote $ J(x)=-\frac{1}{8\pi}|x|^2\ln|x| $. Then $ \Delta \left( J(x-y)\right)=\Gamma(x-y)-\frac{1}{2\pi}.  $
	Using transformation of coordinates, one computes directly that
	\begin{equation*}
	\text{div}\left(K(y)\nabla J(T_y(x-y)) \right)= \Gamma(T_y(x-y))-\frac{1}{2\pi},
	\end{equation*}
	from which we deduce,
	\begin{equation*}
	\text{div}\left(K(y)\nabla   \partial_{x_j}\left( J(T_y(x-y))\right)  \right)= \partial_{x_j} \left( \Gamma(T_y(x-y))\right) .
	\end{equation*}
	We define for $ x\in\Omega $
	\begin{equation}\label{F_{1,y}}
	\begin{split}
	F_{1,y}(x)=\sum_{i,j=1}^2\partial_{x_i}K_{ij}(y)\cdot\sqrt{\det K(y)}^{-1} \partial_{x_j}\left( J\left (T_y(x-y)\right )\right) ,
	\end{split}
	\end{equation}
	then one has
	\begin{equation}\label{3-2}
	\text{div}\left(K(y)\nabla F_{1,y}(x)  \right)= \sum_{i,j=1}^2\partial_{x_i}K_{ij}(y)\partial_{x_j}\left(\sqrt{\det K(y)}^{-1}\Gamma\left (T_y(x-y)\right ) \right).
	\end{equation}
	
	As for $ A_2 $, using Taylor's expansion we obtain
	\begin{equation}\label{3-3}
	\begin{split}
	\sum_{i,j=1}^2\left( K_{ij}(x)-K_{ij}(y)\right) \partial_{x_ix_j}&\left(\sqrt{\det K(y)}^{-1}\Gamma\left (T_y(x-y)\right ) \right)\\
	=\sum_{\alpha,i,j=1}^2\sqrt{\det K(y)}^{-1}&\partial_{x_\alpha}K_{ij}(y) (x-y)_\alpha\cdot\partial_{x_ix_j}\left( \Gamma\left (T_y(x-y)\right ) \right)+\phi_y(x),
	\end{split}
	\end{equation}
	where $ \phi_y(\cdot)\in L^p(\Omega) $ for all $ p>1. $
	Since
	\begin{equation*}
	\partial_{x_ix_j}\Gamma(x)=-\frac{1}{2\pi}\left(\frac{\delta_{i,j}}{|x|^2}-\frac{2x_ix_j}{|x|^4} \right),
	\end{equation*}
	where $ \delta_{i,j}=1 $ for $ i=j $ and $ \delta_{i,j}=0 $ for $ i\neq j $, we have
	\begin{equation}\label{3-6}
	\begin{split}
	\partial_{x_ix_j}\left( \Gamma(T_y(x-y))\right) =-\frac{1}{2\pi}\sum_{m,n=1}^2T_{mj}T_{ni}\left(\frac{\delta_{m,n}}{|T_y(x-y)|^2}-\frac{2\left( T_y(x-y)\right)_m\left( T_y(x-y)\right)_n}{|T_y(x-y)|^4} \right).
	\end{split}
	\end{equation}
	Taking \eqref{3-6} into \eqref{3-3}, we get
	\begin{equation}\label{3-7}
	\begin{split}
	\sum_{i,j=1}^2\left( K_{ij}(x)-K_{ij}(y)\right) \partial_{x_ix_j}&\left(\sqrt{\det K(y)}^{-1}\Gamma\left (T_y(x-y)\right ) \right)\\
	=\sum_{\alpha,\beta,i,j,m,n=1}^2\sqrt{\det K(y)}^{-1}&\partial_{x_\alpha}K_{ij}(y)T^{-1}_{\alpha\beta} \left (T_y(x-y)\right )_\beta\cdot  \\ -\frac{1}{2\pi}T_{mj}T_{ni}&\left(\frac{\delta_{m,n}}{|T_y(x-y)|^2}-\frac{2\left( T_y(x-y)\right)_m\left( T_y(x-y)\right)_n}{|T_y(x-y)|^4} \right)  +\phi_y(x).
	\end{split}
	\end{equation}

	Note that
	\begin{equation}\label{3-4}
	\frac{x^p}{|x|^4}=-\frac{1}{8}\Delta\left( \frac{x^p}{|x|^2}\right)+\frac{1}{8}\frac{\Delta x^p}{|x|^2}\ \ \ \ \text{for}\ |p|=3,
	\end{equation}
	where $p=(p_1,p_2)$ is the multi-index and $ x^p=x_1^{p_1}x_2^{p_2} $. From \eqref{3-4}, it is not hard to check that for $ 1\leq m\neq n\leq 2 $
	\begin{equation*}
	\begin{cases}
	\frac{x_m}{|x|^2}&=\Delta\left( \frac{1}{2}x_m\ln|x|\right),\\
	\frac{x_m^2x_n}{|x|^4}&=\Delta\left( -\frac{1}{8}\frac{x_m^2x_n}{|x|^2}+\frac{1}{8}x_n\ln|x|\right),\\
	\frac{x_m^3}{|x|^4}&=\Delta\left( -\frac{1}{8}\frac{x_m^3}{|x|^2}+\frac{3}{8}x_m\ln|x|\right),
	\end{cases}
	\end{equation*}
	which implies that
	\begin{equation}\label{3-5}
	\begin{cases}
	\frac{\left (T_y\left (x-y\right )\right )_m}{|T_y\left (x-y\right )|^2}&=\text{div}\left(K(y) \nabla\left( \frac{1}{2}\left (T_y\left (x-y\right )\right )_m\ln|T_y\left (x-y\right )|\right)\right) ,\\
	\frac{\left (T_y\left (x-y\right )\right )_m^2\left (T_y\left (x-y\right )\right )_n}{|T_y\left (x-y\right )|^4}&=\text{div}\left(K(y) \nabla \left( -\frac{1}{8}\frac{\left (T_y\left (x-y\right )\right )_m^2\left (T_y\left (x-y\right )\right )_n}{|T_y\left (x-y\right )|^2}+\frac{1}{8}\left (T_y\left (x-y\right )\right )_n\ln|T_y\left (x-y\right )|\right)\right) ,\\
	\frac{\left (T_y\left (x-y\right )\right )_m^3}{|T_y\left (x-y\right )|^4}&=\text{div}\left(K(y) \nabla\left( -\frac{1}{8}\frac{\left (T_y\left (x-y\right )\right )_m^3}{|T_y\left (x-y\right )|^2}+\frac{3}{8}\left (T_y\left (x-y\right )\right )_m\ln|T_y\left (x-y\right )|\right)\right).
	\end{cases}
	\end{equation}
	We define for $ x\in \Omega $
	\begin{equation*}
	\begin{split}
	F_{2,y}(x)=-&\frac{1}{2\pi}\sqrt{\det K(y)}^{-1}\sum_{\alpha,\beta,i,j,m,n=1}^2 \partial_{x_\alpha}K_{ij}(y)T^{-1}_{\alpha\beta} T_{mj}T_{ni}\cdot\frac{1}{2}\left (T_y\left (x-y\right )\right )_\beta\ln|T_y\left (x-y\right )|\delta_{m,n}\\
	+&\frac{1}{\pi}\sqrt{\det K(y)}^{-1} \sum_{\alpha,i,j=1}^2\partial_{x_\alpha}K_{ij}(y)\cdot \\
	&\bigg[T^{-1}_{\alpha1}T_{1j}T_{1i}\left( -\frac{1}{8}\frac{\left( T_y\left (x-y\right )\right) _1^3}{|T_y\left (x-y\right )|^2}+\frac{3}{8}\left( T_y\left (x-y\right )\right)_1\ln|T_y\left (x-y\right )|\right) \\
	+&T^{-1}_{\alpha1}T_{1j}T_{2i}\left( -\frac{1}{8}\frac{\left( T_y\left (x-y\right )\right) _1^2\left( T_y\left (x-y\right )\right) _2}{|T_y\left (x-y\right )|^2}+\frac{1}{8}\left( T_y\left (x-y\right )\right)_2\ln|T_y\left (x-y\right )|\right) \\
	+&T^{-1}_{\alpha1}T_{2j}T_{1i}\left( -\frac{1}{8}\frac{\left( T_y\left (x-y\right )\right) _1^2\left( T_y\left (x-y\right )\right) _2}{|T_y\left (x-y\right )|^2}+\frac{1}{8}\left( T_y\left (x-y\right )\right)_2\ln|T_y\left (x-y\right )|\right)\\
	+&T^{-1}_{\alpha1}T_{2j}T_{2i}\left( -\frac{1}{8}\frac{\left( T_y\left (x-y\right )\right) _1\left( T_y\left (x-y\right )\right) _2^2}{|T_y\left (x-y\right )|^2}+\frac{1}{8}\left( T_y\left (x-y\right )\right)_1\ln|T_y\left (x-y\right )|\right)
	\end{split}
	\end{equation*}
	\begin{equation}\label{F_{2,y}}
	\begin{split}
	+&T^{-1}_{\alpha2}T_{1j}T_{1i}\left( -\frac{1}{8}\frac{\left( T_y\left (x-y\right )\right) _1^2\left( T_y\left (x-y\right )\right) _2}{|T_y\left (x-y\right )|^2}+\frac{1}{8}\left( T_y\left (x-y\right )\right)_2\ln|T_y\left (x-y\right )|\right)\\
	+&T^{-1}_{\alpha2}T_{1j}T_{2i}\left( -\frac{1}{8}\frac{\left( T_y\left (x-y\right )\right) _1\left( T_y\left (x-y\right )\right) _2^2}{|T_y\left (x-y\right )|^2}+\frac{1}{8}\left( T_y\left (x-y\right )\right)_1\ln|T_y\left (x-y\right )|\right)\\
	+&T^{-1}_{\alpha2}T_{2j}T_{1i}\left( -\frac{1}{8}\frac{\left( T_y\left (x-y\right )\right) _1\left( T_y\left (x-y\right )\right) _2^2}{|T_y\left (x-y\right )|^2}+\frac{1}{8}\left( T_y\left (x-y\right )\right)_1\ln|T_y\left (x-y\right )|\right)\\
	+&T^{-1}_{\alpha2}T_{2j}T_{2i}\left( -\frac{1}{8}\frac{\left( T_y\left (x-y\right )\right) _2^3 }{|T_y\left (x-y\right )|^2}+\frac{3}{8}\left( T_y\left (x-y\right )\right)_2\ln|T_y\left (x-y\right )|\right)\bigg].
	\end{split}
	\end{equation}
	Combining \eqref{F_{2,y}} with \eqref{3-7} and \eqref{3-5}, we get
	\begin{equation}\label{3-8}
	\begin{split}
	\text{div}\left(K(y) \nabla F_{2,y}(x)  \right)=&\sum_{\alpha,\beta,i,j,m,n=1}^2\sqrt{\det K(y)}^{-1}\partial_{x_\alpha}K_{ij}(y)T^{-1}_{\alpha\beta} \left (T_y(x-y)\right )_\beta\cdot \\
	&-\frac{1}{2\pi}T_{mj}T_{ni}\left(\frac{\delta_{m,n}}{|T_y(x-y)|^2}-\frac{2\left( T_y(x-y)\right)_m\left( T_y(x-y)\right)_n}{|T_y(x-y)|^4} \right)\\
	=&
	\sum_{i,j=1}^2\left( K_{ij}(x)-K_{ij}(y)\right) \partial_{x_ix_j}\left(\sqrt{\det K(y)}^{-1}\Gamma\left (T_y(x-y)\right ) \right)-\phi_y(x).
	\end{split}
	\end{equation}
	
	Now we define $ \bar{H}_{1,y}(x)=\bar{S}_K(x,y)+F_{1,y}(x)+F_{2,y}(x) $. Taking \eqref{3-2} and \eqref{3-8} into \eqref{3-1}, we obtain
	\begin{equation}\label{3-9}
	\begin{split}
	-\text{div}&\left (K(x)\nabla \bar{H}_{1,y}(x)\right )\\
	=&-\text{div}\left (\left( K(x)-K(y)\right) \nabla\left( F_{1,y}(x)+F_{2,y}(x)\right) \right) \\
	&+\sum_{i,j=1}^2\left( \partial_{x_i}K_{ij}(x)-\partial_{x_i}K_{ij}(y)\right) \partial_{x_j}\left(\sqrt{\det K(y)}^{-1}\Gamma\left (T_y(x-y)\right ) \right)+\phi_y(x).
	\end{split}
	\end{equation}
	We can verify that for all $ y\in\Omega $, the right-hand side of \eqref{3-9} belongs to $ L^p(\Omega) $ for all $ p>1. $ Note also that
	\begin{equation*}
	\bar{H}_{1,y}(x)=-\sqrt{\det K(y)}^{-1}\Gamma\left (T_y(x-y)\right )+F_{1,y}(x)+F_{2,y}(x)\ \ \ \ x\in\partial\Omega.
	\end{equation*}
	For $ x,y\in\Omega $, we define $ \bar{H}_1(x,y)=\bar{H}_{1,y}(x) $. Applying the elliptic theory, we obtain that $ x\to \bar{H}_1(x,y) $ is in $ C^{1,\gamma}(\overline{\Omega}) $, for all $ \gamma\in(0,1) $. Furthermore,
	by the continuity of the right-hand side of \eqref{3-9} and the boundary condition with respect to $ y $ in $ L^p(\Omega) $   and $ C^2(\partial\Omega) $, respectively, we can get $ \bar{H}_1(x,y)=\bar{H}_{1,y}(x)\in C(\Omega, C^{1,\gamma}(\overline{\Omega})) $ and thus $ \nabla_x\bar{H}_1(x,y)\in C(\Omega\times \Omega) $.
	
	Similarly,  taking $ \nabla_y $ to both sides of \eqref{3-9}, we can check that $ \nabla_y\bar{H}_{1,y}(x)\in C(\Omega, C^{0,\gamma}(\overline{\Omega})) $, which implies that $ \nabla_y\bar{H}_1(x,y)\in C(\Omega\times \Omega) $, then $ \bar{H}_1 $ is a $ C^1 $ function over $ \Omega\times \Omega $. From \eqref{F_{1,y}} and \eqref{F_{2,y}}, we can prove that \eqref{exp of F_{1,y}} and \eqref{exp of F_{2,y}} hold. Finally, $ \bar{S}_K(x, x) =  \bar{H}_1(x,x) $ is clearly in $ C^1(\Omega) $.
	
\end{proof}

We now give the proof of Lemma \ref{H estimate}.

\begin{customthm}{A.3}\label{lemA3}
	Define $ \zeta_{\varepsilon, \hat{x}, \hat{q}}(x)=H_{\varepsilon, \hat{x}, \hat{q}}(x)+\frac{2\pi\hat{q}\sqrt{\det K(\hat{x})}\ln\frac{1}{\varepsilon}}{\ln s_\varepsilon}\bar{S}_K(x,\hat{x}) $ for $ x\in\Omega $. Then for any $ p\in(1,2) $, there exists a constant $ C>0 $ independent of $ \varepsilon $ such that
\begin{equation*}
||\zeta_{\varepsilon, \hat{x}, \hat{q}}||_{C^{0,2-\frac{2}{p}}(\Omega)}\leq Cs_\varepsilon^{\frac{2}{p}-1}.
\end{equation*}
\end{customthm}

\begin{proof}
We follow the idea of proof of theorem 1.3 in \cite{CW1}. By the construction of Green's function $ G_K $,  on $ \{x\mid |T_{\hat{x}}(x-\hat{x})|>s_\varepsilon\}$ we have
	\begin{equation}\label{3000}
	\text{div}(K(x)\nabla G_K(x,\hat{x}))=0 ,
	\end{equation}
thus by Lemma \ref{Green expansion} we deduce that
	\begin{equation}\label{3001}
	-\sqrt{\det K(\hat{x})}^{-1}\text{div}(K(x)\nabla\Gamma\left (T_{\hat{x}}(x-\hat{x})\right ))=\text{div}(K(x)\nabla\bar{S}_K(x,\hat{x})).
	\end{equation}
Taking \eqref{eq2},\eqref{eq2-2} and \eqref{3001} into \eqref{eq3}, we have on $ \{x\mid |T_{\hat{x}}(x-\hat{x})|>s_\varepsilon\} $,
	\begin{equation*}
	\begin{split}
	-\varepsilon^2\text{div}\left (K(x)\nabla H_{\varepsilon, \hat{x}, \hat{q}}\right )&=\varepsilon^2\text{div}(K(x)\nabla V_{\varepsilon, \hat{x}, \hat{q}})\\
	&=\frac{\varepsilon^2\hat{q}\ln\frac{1}{\varepsilon}}{\ln s_{\varepsilon}}\text{div}(K(x)\nabla \ln|T_{\hat{x}}(x-\hat{x})|)\\
	&=\frac{2\pi\varepsilon^2\hat{q}\sqrt{\det K(\hat{x})}\ln\frac{1}{\varepsilon}}{\ln s_{\varepsilon}}\text{div}(K(x)\nabla \bar{S}_K(x,\hat{x})),
	\end{split}
	\end{equation*}
	which implies that $ -\text{div}(K(x)\nabla \zeta_{\varepsilon, \hat{x}, \hat{q}})\equiv 0 $ on $ \{x\mid |T_{\hat{x}}(x-\hat{x})|>s_\varepsilon\} $.
	
	On the set $ \{x\mid |T_{\hat{x}}(x-\hat{x})|<s_\varepsilon\} $, we note that
	\begin{equation}\label{3-01}
	-\text{div}(K(x)\nabla \zeta_{\varepsilon, \hat{x}, \hat{q}})(x)=\text{div}((K(x)-K(\hat{x}))\nabla V_{\varepsilon, \hat{x}, \hat{q}})(x)-\frac{2\pi\hat{q}\sqrt{\det K(\hat{x})}\ln\frac{1}{\varepsilon}}{\ln s_\varepsilon}\text{div}(K(x)\nabla\bar{S}_K(x,\hat{x})).
	\end{equation}
	From \eqref{eq2-2} and Lemma \ref{Green expansion2}, we get on
	$ \{x\mid |T_{\hat{x}}(x-\hat{x})|<s_\varepsilon\} $
	\begin{equation*}
	|\text{div}((K(x)-K(\hat{x}))\nabla V_{\varepsilon, \hat{x}, \hat{q}})(x)|\leq \frac{C}{s_\varepsilon}\left( \phi'\left(\frac{|T_{\hat{x}}(x-\hat{x})|}{s_\varepsilon}\right)+\phi''\left(\frac{|T_{\hat{x}}(x-\hat{x})|}{s_\varepsilon}\right)\right) 
	\end{equation*}
and
	\begin{equation*}
	\bigg|\frac{2\pi\hat{q}\sqrt{\det K(\hat{x})}\ln\frac{1}{\varepsilon}}{\ln s_\varepsilon}\text{div}(K(x)\nabla\bar{S}_K(x,\hat{x}))\bigg|\leq C\left ( \ln\frac{1}{|T_{\hat{x}}(x-\hat{x})|}+\frac{1}{|T_{\hat{x}}(x-\hat{x})|}+1\right ).
	\end{equation*}
	Taking these into \eqref{3-01}, for any $ p\in(1,2) $ we have
	\begin{equation*}\label{3-02}
	||\text{div}(K(x)\nabla \zeta_{\varepsilon, \hat{x}, \hat{q}})||_{L^p(\{x\mid |T_{\hat{x}}(x-\hat{x})|<s_\varepsilon\})}\leq \frac{C}{s_\varepsilon}\cdot s_\varepsilon^{\frac{2}{p}}=Cs_\varepsilon^{\frac{2}{p}-1}.
	\end{equation*}
Note also that $ \zeta_{\varepsilon, \hat{x}, \hat{q}}\equiv0 $ on $ \partial \Omega. $  By the $ L^p $ theory for second-order elliptic equations, for any $ p\in(1,2) $
	\begin{equation*}
	||\zeta_{\varepsilon, \hat{x}, \hat{q}}||_{W^{2,p}(\Omega)}\leq C||\text{div}(K(x)\nabla \zeta_{\varepsilon, \hat{x}, \hat{q}})||_{L^p(\{x\mid |T_{\hat{x}}(x-\hat{x})|<s_\varepsilon\})}\leq  C s_\varepsilon^{\frac{2}{p}-1}.
	\end{equation*}
Using the Sobolev embedding theorem, the proof is complete.
\end{proof}

\subsection*{Acknowledgements:}

\par
D. Cao was partially supported by National Key R \& D
Program (2022YFA1005602). J. Wan was supported by NNSF of China (grant No. 12271539 and 12471190).

\subsection*{Conflict of interest statement} On behalf of all authors, the corresponding author states that there is no conflict of interest.

\subsection*{Data availability statement} All data generated or analysed during this study are included in this published article  and its supplementary information files.

\end{document}